\documentclass[10pt]{article}
\usepackage{epsf}
\usepackage{amsmath}

\allowdisplaybreaks

\usepackage[showframe=false]{geometry}
\usepackage{changepage}

\usepackage{epsfig}
\usepackage{amssymb}

\usepackage{amsthm}
\usepackage{setspace}
\usepackage{cite}

\usepackage{mcite}

\usepackage{algorithmic}  
\usepackage{algorithm}

\usepackage{shadow}
\usepackage{fancybox}
\usepackage{fancyhdr}

\usepackage{color}
\usepackage[usenames,dvipsnames,svgnames,table]{xcolor}
\newcommand{\bl}[1]{\textcolor{blue}{#1}}
\newcommand{\red}[1]{\textcolor{red}{#1}}
\newcommand{\gr}[1]{\textcolor{green}{#1}}

\definecolor{mypurple}{rgb}{.4,.0,.5}

\usepackage[hyphens]{url}

\usepackage[colorlinks=true,
            linkcolor=black,
            urlcolor=blue,
            citecolor=purple]{hyperref}

\usepackage{breakurl}

\def\w{{\bf w}}

\def\y{{\bf y}}

\def\x{{\bf x}}

\def\x{{\mathbf x}}

\def\w{{\bf w}}

\def\x{{\bf x}}
\def\y{{\bf y}}

\def\h{{\bf h}}

\def\be{\begin{equation}}
\def\ee{\end{equation}}
\def\ba{\left[\begin{array}}
\def\ea{\end{array}\right]}

\def\w{{\bf w}}

\def\x{{\bf x}}
\def\y{{\bf y}}

\def\1{{\bf 1}}

\def\g{{\bf g}}
\def\0{{\bf 0}}

\def\erfinv{\mbox{erfinv}}
\def\erf{\mbox{erf}}
\def\erfc{\mbox{erfc}}

\def\erfinv{\mbox{erfinv}}

\def\Sw{S_w}

\def\hw{\bar{\h}}

\def\Sw{S_w}

\def\mR{{\mathbb R}}

\def\psiint{\Psi_{int}}
\def\psiext{\Psi_{ext}}
\def\psicom{\Psi_{com}}

\def\lp{\left (}
\def\rp{\right )}

\newtheorem{theorem}{Theorem}

\begin{document}

\begin{singlespace}

\title {Partial $\ell_1$ optimization in random linear systems -- phase transitions and large deviations 
}
\author{
\textsc{Mihailo Stojnic
\footnote{e-mail: {\tt flatoyer@gmail.com}} }}
\date{}
\maketitle

\centerline{{\bf Abstract}} \vspace*{0.1in}

$\ell_1$ optimization is a well known heuristic often employed for solving various forms of sparse linear problems. In this paper we look at its a variant that we refer to as the \emph{partial} $\ell_1$ and discuss its mathematical properties when used for solving linear under-determined systems of equations. We will focus on large random systems and discuss the phase transition (PT) phenomena and how they connect to the large deviation principles (LDP). Using a variety of probabilistic and geometric techniques that we have developed in recent years we will first present general guidelines that conceptually fully characterize both, the PTs and the LDPs. After that we will put an emphasis on providing a collection of explicit analytical solutions to all of the underlying mathematical problems. As a nice bonus to the developed concepts, the forms of the analytical solutions will, in our view,  turn out to be fairly elegant as well.

\vspace*{0.25in} \noindent {\bf Index Terms: Phase transitions;
linear systems of equations; partial $\ell_1$; large deviations}.

\end{singlespace}

\section{Introduction}
\label{sec:back}

Over last several decades, studying various properties of the $\ell_1$ heuristic has been the subject of great interest in a variety of scientific communities. We would think that there are two main reasons for that: 1) its excellent performance characteristics and 2) the breakthrough results of \cite{CRT,Donoho06CS} that for the first time provided fully mathematically rigorous results that on a qualitative level were accurately describing/emphasizing $\ell_1$'s performance abilities. Of course, the excellent performance characteristics have been known for a long time (in fact, for at least a half of the last century) and as such have served as cornerstones supporting further $\ell_1$'s developments, adaptations, and applications in a variety of different fields. From a mathematical purist standpoint, these were typically on a heuristic level. Although there were quite a few theoretical results in earlier decades as well, they typically weren't capturing what mathematicians would consider the key $\ell_1$ properties. These would eventually be proven for the first time in a statistical context in \cite{CRT,Donoho06CS}, and, in our view, it is with these works that a new era in studying of mathematical aspects of $\ell_1$ basically jump-started.

Nowdays of course the mathematical side of the $\ell_1$ story substantially expanded as well and one might even say that it is pretty much catching up with the diversity of the heuristic applications. In this paper we provide a contribution along these lines as well. Namely, we will study a particular form of $\ell_1$ that has been thought of for some time now as a strategy potentially offering an algorithmic path for eventually improving over the standard $\ell_1$. Before we introduce the form that we will study and say a few words what of its aspects we will specifically focus on we will take a few moments to put everything on a right mathematical track and introduce the standard linear systems setup that we will use as a benchmark to present our results. We should also mention that these problems are well known and we will occasionally assume a high degree of familiarity with many of the concepts that we will associate with them (the interested reader though can get a bit more detailed introduction to many of these in a long line of our works initiated in \cite{StojnicCSetamBlock09,StojnicUpperBlock10,StojnicICASSP09block,StojnicJSTSP09,StojnicCSetam09,StojnicICASSP10knownsupp,StojnicICASSP10var,StojnicICASSP10block,StojnicUpper10}; also, we will try to maintain as much of a parallelism with some of these works so that the transition and reading are rather smooth).

The standard linear systems assume an $m\times n$ ($m\leq n$) system matrix $A$ and an $n$ dimensional vector $\tilde{\x}$ with real entries (for short we write $A\in \mR^{m\times n}$ and $\tilde{\x}\in \mR^{n}$). Then the standard matrix-vector multiplication of $A$ and $\tilde{\x}$ gives
\begin{equation}
\y=A\tilde{\x}. \label{eq:defy}
\end{equation}
The key problem in studying linear inverse problems is finding $\tilde{\x}$ if $A$ and $\y$ in (\ref{eq:defy}) are given. By the formation of $\y$ such an $\tilde{\x}$ obviously exists. How many of them are there though is a whole another story. Namely, if $m=n$ and $A$ is full rank there is only one $\tilde{\x}$ such that (\ref{eq:defy}) holds. If $m<n$ and $A$ is full rank things become a bit more interesting. In fact, the possibility to have more than one $\tilde{\x}$ that satisfies (\ref{eq:defy}) would make finding a particular one as a typically not well posed problem. Of course, there is a catch and it is hidden in an additional structuring of $\tilde{\x}$ that makes the above problem typically well posed and ultimately of interest in many applications. The type of the structure that is typically of interest when the above mentioned $\ell_1$ optimization is involved (and that will also be assumed throughout this paper) is the so-called sparsity of $\tilde{\x}$.  In such scenarios, instead of finding just an $\tilde{\x}$ in (\ref{eq:defy}) one  focuses on solving the following problem
\begin{eqnarray}
\mbox{min} & & \|\x\|_{0}\nonumber \\
\mbox{subject to} & & A\x=\y, \label{eq:l0}
\end{eqnarray}
where $\|\x\|_{0}$ is what is typically called the $\ell_0$ (quasi) norm of vector $\x$  ($\|\x\|_{0}$ is, mathematically speaking, clearly not a norm; however, we will use the norm notation/terminology while essentially thinking of it as being the number of the nonzero entries of $\x$). In a nutshell, (\ref{eq:l0}) will find the so-called sparsest $\x$ that satisfies the constraints in (\ref{eq:l0}) (of course, from this point on the assumption will always be that there is at least one such $\x$). To mathematically deal with the sparsity we will introduce the $k$-sparse vectors as vectors that have no more than $k$ nonzero elements. An easy algebraic exercise then shows that for $k<m/2$ the solution of (\ref{eq:l0}) is in fact unique; moreover, we will also additionally assume that there is no $\x$ that satisfies constraints of (\ref{eq:l0}) that is less than $k$ sparse.


Solving (\ref{eq:l0}) is not so easy in general. The straightforward way of exhaustively selecting all subsets of $k$ columns and then solving the resulting over-determined systems is in the so-called linear regime (which will eventually be of our interest in this paper and which assumes $k=\beta n$, $m=\alpha n$, $n$ is large, and $\alpha$ and $\beta$ are constant independent of $n$) of exponential complexity in $n$. We will view such a complexity as too high and instead will focus on heuristics of polynomial complexity. Among many successful ones developed over last several decades (see, e.g. \cite{JATGomp,NeVe07,DTDSomp,NT08,DaiMil08,DonMalMon09}) one that supremely stands out in our view is the following $\ell_1$ relaxation of (\ref{eq:l0})
\begin{eqnarray}
\mbox{min} & & \|\x\|_{1}\nonumber \\
\mbox{subject to} & & A\x=\y. \label{eq:l1}
\end{eqnarray}
The above $\ell_1$ optimization is of course a linear program relatively easily solvable in polynomial time. In fact, if one considers the following object $\|\x\|_q=(\sum_{i=1}^{n}|\x_i|_q)^{\frac{1}{q}}$ as $q$ increases starting from (close to) zero then one observes that $q=1$ is the first case when $\|\x\|_q$ becomes a convex function and if used instead of $\|\x\|_0$ in (\ref{eq:l0}) would make (\ref{eq:l0}) a convex optimization problem solvable in polynomial time (in fact for $q=1$ the problem is actually linear).

As mentioned earlier, there are many reasons for $\ell_1$'s popularity but two of them are in our view the most important. The practical applicability due to its polynomial complexity and excellent solving abilities and the existence of rigorous mathematical results that confirm such performance characteristics. While the practical applicability has been known for quite some time the analytical progress flourished over the last decade. Performance characterizations started in \cite{CRT,Donoho06CS} and perfected in \cite{DonohoPol,DonohoUnsigned,StojnicCSetam09,StojnicUpper10} mathematically solidified the importance of (\ref{eq:l1}) in studying the linear under-determined systems with structured solutions.

In this paper we will try to move things a step further and consider a particular modification of the standard $\ell_1$ that has been viewed as a path towards potential improving on $\ell_1$ ($\ell_1$ though continues to be basically a golden standard when it comes to solving (\ref{eq:l0}) in polynomial time; more on this can be found in e.g. \cite{StojnicReDirChall13}). Such a modification will assume a bit of a feedback which the standard $\ell_1$ from (\ref{eq:l1}) does not use. Namely, we will assume that one can beforehand determine/guess the location of a fraction of $\x$'s support (as usual, under support of a vector we assume the location of its nonzero components; for vector $\x$ for example we denote its support as $supp(\x)$). If one recalls that what makes the problem in (\ref{eq:l0}) hard is determining the location of the nonzero components of $\x$ ($supp(\x)$) then one naturally expects that knowing beforehand some of these locations should be beneficial. This was indeed rigorously shown to be true in \cite{StojnicTowBettCompSens13,StojnicICASSP10knownsupp}. Before discussing what exactly was actually shown in \cite{StojnicTowBettCompSens13,StojnicICASSP10knownsupp}, we will say a few words about the rationale for feedback considerations. Additionally, to ensure that everything continues to be on a right mathematical track we will also, before switching to the following section and the discussion about the benefits of knowing parts of $supp(\x)$, introduce formal mathematical definitions for the objects that we have just mentioned.

Adding some knowledge to improve the $\ell_1$'s performance is actually natural. Namely, even when one fails to solve (\ref{eq:l0}) (through any algorithm including $\ell_1$ itself) it is reasonable to believe that at least some of $supp(\x)$ elements will be correctly identified. On its own this may not be of much use as one will not be able to recover unknown $\x$ unless the entire $supp(\x)$ is correctly determined. However, if it is put back in use it may be beneficial in designing/upgrading recovery algorithms. Of course one would be tempted to believe that the larger the correctly identifiable portion of $supp(\x)$ the higher the chance that such information will be useful if reused in the recovery process (among other things \cite{StojnicTowBettCompSens13,StojnicICASSP10knownsupp} precisely quantified the dependency between the size of the a priori known portion of $supp(\x)$ and the recoverable sparsity). Of course, things are not as simple as they may sound and we leave a more complete discussion to \cite{StojnicTowBettCompSens13}. Here we just add, that it is typically very important which fraction of $supp(\x)$ is used as a priori known and in general it is not an easy task to determine those fractions that are the most relevant. It seems though, that it could still be a bit easier task than actually solving the original problem (\ref{eq:l0}); it remains though to be seen if that indeed is the case. We should also say that the above is related to a potential improvement over $\ell_1$. There are scenarios that inherently allow for knowing a portion of $supp(\x)$. In such cases studying effects that a priori known portion of $supp(\x)$ will have on the recovery algorithms is of course of independent interest and should not necessarily be viewed through the prism of the improvement over $\ell_1$.

Now, going back to aligning everything so that it is on a right mathematical track, we start by formally introducing vectors with \emph{partially} known support, (more on this type of vectors as well as on their potential applications can be found in e.g. \cite{VasLu09}).
As mentioned above, under partially known support we assume that \emph{locations} of a fraction of the non-zero components of $\x$ are \emph{a priori} known and that that knowledge can possibly be exploited in solving (\ref{eq:l0}) (basically incorporated in the design of the recovery algorithms used for solving (\ref{eq:l0})). To make everything precise, we will denote by $\Pi$ the set of the indexes of the known locations of the non-zero components of $\x$. We will assume that the cardinality of $\Pi$ is $\eta k$ (where $\eta$ is a constant independent of $n$ and $0\leq \eta\leq 1$). To recover $k$-sparse $\x$ with partially known support one can then design the algorithms using the available knowledge of $\Pi$. Among many ways how one can do so, here, we focus on a particular modification of (\ref{eq:l1}) considered in e.g. \cite{VasLu09,SteChr09,StojnicICASSP10knownsupp} that assumes solving
\begin{eqnarray}
\mbox{min} & & \sum_{i\notin\Pi} |\x_i|\nonumber \\
\mbox{subject to} & & A\x=\y. \label{eq:l1imp}
\end{eqnarray}
The above problem, to which we will refer as \emph{partial} $\ell_1$, is fairly similar to the standard $\ell_1$ from (\ref{eq:l1}) and it is perhaps a bit surprising that it can offer much more than the standard $\ell_1$. \cite{StojnicTowBettCompSens13,StojnicICASSP10knownsupp} showed that it actually can
and in a statistical context these results in fact precisely quantified by how much the algorithm from (\ref{eq:l1imp}) improves on its a counterpart from (\ref{eq:l1}). As mentioned above, we will in the following sections in detail recall on the results from \cite{StojnicTowBettCompSens13,StojnicICASSP10knownsupp}. Here we give a little bit of a flavor as to what was done in \cite{StojnicTowBettCompSens13,StojnicICASSP10knownsupp} and what we will do here. Namely, the results of \cite{StojnicTowBettCompSens13,StojnicICASSP10knownsupp} relate to the so-called phase-transition (PT) phenomena (precisely the same phenomena that  \cite{DonohoPol,DonohoUnsigned,StojnicCSetam09,StojnicUpper10} uncovered for the standard $\ell_1$ from (\ref{eq:l1}) when utilized in statistical contexts). Similarly to what was done in both sets of results, \cite{DonohoPol,DonohoUnsigned} and \cite{StojnicCSetam09,StojnicUpper10}, for the standard $\ell_1$, we in \cite{StojnicTowBettCompSens13,StojnicICASSP10knownsupp}, in addition to uncovering the existence of the phase transition phenomenon for the partial $\ell_1$ from (\ref{eq:l1imp}), precisely characterized the so-called ``breaking points" where these phase transitions happen (essentially the highest possible $\beta$ for which the solution of (\ref{eq:l1imp}) with overwhelming probability matches the sparsest solution of (\ref{eq:l0}) for a fixed $\alpha$; under overwhelming probability we will in this paper consider probability over statistics of $A$ that is no more than a number exponentially decaying in $n$ away from $1$). Here though, we will make several key steps that will help us substantially deepen our understanding of the underlying phase transitions. As in \cite{Stojnicl1RegPosasymldp}, we will essentially connect the phase transitions to the so-called \emph{large deviations property/principle} (LDP) from the classical probability theory and provide their explicit characterizations when viewed through such a prism. We will do so for two types of partial $\ell_1$ (the one from (\ref{eq:l1imp}) and another one, essentially, potentially a more realistic variant of (\ref{eq:l1imp}), that we will introduce later on). Following further into the footsteps of \cite{Stojnicl1RegPosasymldp}, we will do so through two seemingly different approaches, one that is purely probabilistic and another one that relies on high-dimensional geometry.

The paper presentation will be split into several sections. We will start by discussing the standard phase transitions and in what form they appear in the analysis of the partial $\ell_1$. After that we will move to the LDP characterizations and their connections with PTs. In the later sections of the paper we will show how the PT and LDP results that we will create for the partial $\ell_1$ can be modified so that they fit a more realistic form of the partial $\ell_1$ that we will introduce later on and to which we will refer as the \emph{hidden partial} $\ell_1$. Finally, we will attempt to maintain the level of the presentation so that the final results eventually approach, if not match, the elegance of the corresponding ones from \cite{StojnicCSetam09,Stojnicl1RegPosasymldp,StojnicCSetamBlock09}.

\section{Partial $\ell_1$ -- phase transitions}
\label{sec:phasetrans}

In this section we will revisit the phase transitions (PTs) of the partial $\ell_1$. We start by recalling what PTs are when viewed in the context of the partial $\ell_1$. Informally speaking, they of course refer to the relations between systems dimensions and sparsity so that the partial $\ell_1$ (i.e. (\ref{eq:l1imp})) produces or fails to produce the solution of (\ref{eq:l0}). To put this in a proper mathematical context, we say that for any given constants $0< \alpha\leq 1$ and $0\leq \eta\leq 1$ and \emph{any} given $\x$ with a given fixed location and a given fixed set of signs
there will be a maximum allowable value of $\beta$ such that
(\ref{eq:l1imp}) (with $\Pi$ from the given fixed locations) finds that given $\x$ with overwhelming
probability. We will refer to this maximum allowable value of
$\beta$ as the \emph{weak threshold/breaking point} and will denote it by $\beta_{w}$ (see, e.g. \cite{StojnicICASSP09,StojnicCSetam09,StojnicTowBettCompSens13,StojnicICASSP10knownsupp,Stojnicl1RegPosasymldp}). Correspondingly, we also say that the algorithm exhibits the \emph{weak} phase transition (i.e. \emph{weak} PT) and we call the resulting curve in the $(\alpha,\beta)$ plane the \emph{weak} phase transition curve (\emph{weak} PT curve). Now, in a more mathematical language, the phase transition phenomenon essentially means that if the problem dimensions are such that the pair $(\alpha,\beta)$ is below the PT curve then the algorithm (here (\ref{eq:l1imp})) solves (in a probabilistic sense) the problem (\ref{eq:l0}); otherwise it fails. A full asymptotic performance characterization of an algorithm that exhibits the phase transition phenomenon assumes determining this phase transition curve. We should mentioned that in addition to the weak thresholds, PTs, and PT curves, one can define various other forms of phase transitions. As our main concern here will be the above introduced weak PT forms we stop short of discussing the other ones in greater details and instead mention in passing that more on them can be found in e.g. \cite{DonohoPol,StojnicCSetam09,StojnicUpper10,StojnicUpperSec13,StojnicLiftStrSec13,Stojnicl1RegPosasymldp}.

When it comes to the standard $\ell_1$ (i.e. the algorithm from (\ref{eq:l1})), there is a large body of work that deals with various aspects of its PTs. As mentioned earlier, the mathematical studying jump-started with the initial, qualitative characterizations that appeared in \cite{CRT,Donoho06CS}. These results were later on substantially improved, eventually reaching the ultimate level of exact PT characterizations achieved in \cite{DonohoPol,DonohoUnsigned,StojnicCSetam09,StojnicUpper10}. \cite{DonohoPol,DonohoUnsigned} did so through establishing a connection between the (\ref{eq:l1})'s PT properties and studying of neighbourly polytopes in high-dimensional geometry, while our own series of work \cite{StojnicCSetam09,StojnicUpper10} did so by developing a novel purely probabilistic approach. On the other hand, when it comes to the partial $\ell_1$ (i.e. the algorithm from (\ref{eq:l1imp})), a bit less work has been done. Nonetheless, the achieved results are equally successful. After initial characterizations of \cite{VasLu09}, we eventually in \cite{StojnicICASSP10knownsupp,StojnicTowBettCompSens13} fully characterized the partial $\ell_1$'s PT. We below recall on a theorem that essentially summarizes the results obtained in \cite{StojnicICASSP10knownsupp,StojnicTowBettCompSens13} and effectively establishes for any $0\leq \eta\leq 1$ and any $0<\alpha\leq 1$ the exact value of $\beta_w$ for which (\ref{eq:l1}) finds the $k$-sparse solution of (\ref{eq:l0}) with a priori known portion of its support of size $\eta k$.


\begin{theorem}(\cite{StojnicICASSP10knownsupp,StojnicTowBettCompSens13} Exact partial $\ell_1$'s weak threshold/PT)
Let $A$ be an $m\times n$ matrix in (\ref{eq:l0})
with i.i.d. standard normal components. Let
the unknown $\x$ that solves (\ref{eq:l0}) be $k$-sparse. Further, let the location and signs of nonzero elements of $\x$ be arbitrarily chosen but fixed. Assume that the location of $\eta k$ ($0\leq \eta\leq 1$) of non-zero elements is arbitrarily chosen, fixed, and a priori known, and let $\Pi$ be the set of those locations.
Let $k,m,n$ be large
and let $\alpha_w=\frac{m}{n}$ and $\beta_w=\frac{k}{n}$ be constants
independent of $m$ and $n$. Let $\erfinv$ be the inverse of the standard error function associated with zero-mean unit variance Gaussian random variable.  Further, let $\alpha_w$ and $\beta_w$ satisfy the following \textbf{fundamental characterization of the \emph{partial} $\ell_1$'s PT}

\begin{center}
\shadowbox{$
\xi^{(p)}_{\alpha_{w},\eta}(\beta_w)\triangleq\psi^{(p)}_{\beta_w,\eta}(\alpha_{w})\triangleq
\frac{(1-\beta_w)\sqrt{\frac{2}{\pi}}e^{-\lp\erfinv\lp\frac{1-\alpha_w}{1-\beta_w}\rp\rp^2}}{(\alpha_w-\eta\beta_w)\sqrt{2}\erfinv \lp\frac{1-\alpha_w}{1-\beta_w}\rp}=1.
$}
-\vspace{-.5in}\begin{equation}
\label{eq:thmweaktheta2}
\end{equation}
\end{center}

Then:
\begin{enumerate}
\item If $\alpha>\alpha_w$ then with overwhelming probability the solution of (\ref{eq:l1imp}) is the $k$-sparse $\x$ that solves (\ref{eq:l0}).
\item If $\alpha<\alpha_w$ then with overwhelming probability there will be a $k$-sparse $\x$ (from a set of $\x$'s with fixed locations and signs of nonzero components) with a priori known set of locations $\Pi$ that is the solution of (\ref{eq:l0}) and is \textbf{not} the solution of (\ref{eq:l1imp}).
    \end{enumerate}
\label{thm:thmweakthr}
\end{theorem}
\begin{proof}
The first part was established in \cite{StojnicICASSP10knownsupp} and the second one was established in \cite{StojnicTowBettCompSens13}.
\end{proof}

\subsection{Partial $\ell_1$'s PT is unambiguous}
\label{sec:propxi}

In this subsection we will provide brief arguments that functions $\xi^{(p)}_{\alpha,\eta}(\beta)$ and $\psi^{(p)}_{\beta,\eta}(\alpha)$ utilized in Theorem \ref{thm:thmweakthr} in an unambiguous way establish the partial $\ell_1$'s PT. These functions are very similar to the corresponding ones in \cite{Stojnicl1RegPosasymldp} and can be analyzed in a similar way. Instead of repeating the entire analysis we will just highlight the parts that are different from what was presented in \cite{Stojnicl1RegPosasymldp}.

\subsubsection{$\xi^{(p)}_{\alpha,\eta}(\beta)$}
\label{sec:propxi1}

Similarly to what was done in \cite{Stojnicl1RegPosasymldp}, the key observation regarding $\xi^{(p)}_{\alpha,\eta}(\beta)$ is that for any fixed $\eta\in (0,1)$ and any fixed $\alpha\in (0,1)$ there is a unique $\beta$ such that $\xi^{(p)}_{\alpha,\eta}(\beta)=1$ which essentially ensures that (\ref{eq:thmweaktheta2}) is an unambiguous PT characterization. To confirm that this is indeed true we will first show that for any fixed $\eta\in (0,1)$ and any fixed $\alpha\in (0,1)$, $\xi^{(p)}_{\alpha,\eta}(\beta)-1$ is a decreasing function of $\beta$ on interval $[0,\alpha)$. Computing the derivative with respect to $\beta$ gives
\begin{eqnarray}\label{eq:propxi1}
  \frac{d(\xi^{(p)}_{\alpha,\eta}(\beta)-1)}{d\beta} & = & \frac{d\lp\frac{(1-\beta)\sqrt{\frac{2}{\pi}}e^{-\lp\erfinv\lp\frac{1-\alpha}{1-\beta}\rp\rp^2}}{(\alpha-\eta\beta)\sqrt{2}\erfinv \lp\frac{1-\alpha}{1-\beta}\rp}-1\rp}{d\beta}\nonumber\\
  & = & \sqrt{\frac{2}{\pi}}\frac{-\frac{\sqrt{\pi} (1-\alpha)}{(1-\beta) \erfinv((1-\alpha)/(1-\beta))^2}-\frac{2 e^{-\lp \erfinv\lp\frac{1-\alpha}{1-\beta}\rp\rp^2}}{\erfinv((1-\alpha)/(1-\beta))}-\frac{2\sqrt{\pi} (1-\alpha)}{1-\beta}}{2 \sqrt{2} (\alpha-\eta\beta)}\nonumber \\
  & & +\eta\lp\frac{(1-\beta)\sqrt{\frac{2}{\pi}}e^{-\lp\erfinv\lp\frac{1-\alpha}{1-\beta}\rp\rp^2}}{(\alpha-\eta\beta)^2\sqrt{2}\erfinv \lp\frac{1-\alpha}{1-\beta}\rp}\rp\nonumber \\
  & = & \sqrt{\frac{2}{\pi}}\frac{-\frac{\sqrt{\pi} (1-\alpha)}{(1-\beta) \erfinv((1-\alpha)/(1-\beta))^2}-\frac{2\lp 1-\frac{\eta(1-\beta)}{\alpha-\eta\beta}\rp e^{-\lp \erfinv\lp\frac{1-\alpha}{1-\beta}\rp\rp^2}}{\erfinv((1-\alpha)/(1-\beta))}-\frac{2\sqrt{\pi} (1-\alpha)}{1-\beta}}{2 \sqrt{2} (\alpha-\eta\beta)}\nonumber \\
  & \leq & \sqrt{\frac{2}{\pi}}\frac{-\frac{\sqrt{\pi} (1-\alpha)}{(1-\beta) \erfinv((1-\alpha)/(1-\beta))^2}+\frac{2\frac{1-\alpha}{\alpha-\beta} e^{-\lp \erfinv\lp\frac{1-\alpha}{1-\beta}\rp\rp^2}}{\erfinv((1-\alpha)/(1-\beta))}-\frac{2\sqrt{\pi} (1-\alpha)}{1-\beta}}{2 \sqrt{2} (\alpha-\eta\beta)}\nonumber \\
  & = & \sqrt{2}\frac{ (1-\alpha)}{1-\beta}\frac{-\frac{1}{ \erfinv((1-\alpha)/(1-\beta))^2}+\frac{2\frac{1-\beta}{\alpha-\beta} e^{-\lp \erfinv\lp\frac{1-\alpha}{1-\beta}\rp\rp^2}}{\sqrt{\pi}\erfinv((1-\alpha)/(1-\beta))}-2}{2 \sqrt{2} (\alpha-\eta\beta)}.\nonumber \\
\end{eqnarray}
We now recall on the following well known inequalities that relate to $\erfc(\cdot)$
\begin{equation}
\frac{2}{\sqrt{\pi}}\frac{e^{-y^2}}{y+\sqrt{y^2+2}}< \erfc(y)\leq \frac{2}{\sqrt{\pi}}\frac{e^{-y^2}}{y+\sqrt{y^2+\frac{4}{\pi}}}.
 \label{eq:propxi2}
\end{equation}
Setting
\begin{equation}\label{eq:propxi3}
  q=\erfinv\left (\frac{1-\alpha}{1-\beta}\right ),
\end{equation}
and utilizing the first of the inequalities in (\ref{eq:propxi2}) we have
\begin{equation}\label{eq:propxi4}
  \frac{2\frac{1-\beta}{\alpha-\beta} e^{-\lp \erfinv\lp\frac{1-\alpha}{1-\beta}\rp\rp^2}}{\sqrt{\pi}\erfinv((1-\alpha)/(1-\beta))}
  =\frac{2e^{-q^2}}{\sqrt{\pi}q\erfc(q)}<1+\sqrt{1+\frac{2}{q^2}}=1+\sqrt{1+\frac{2}{\lp \erfinv\left (\frac{1-\alpha}{1-\beta}\right )\rp ^2}}.
\end{equation}
A combination of (\ref{eq:propxi1}) and (\ref{eq:propxi4}) then gives
\begin{eqnarray}\label{eq:propxi5}
  \frac{d(\xi^{(p)}_{\alpha,\eta}(\beta)-1)}{d\beta}
  & \leq & \sqrt{2}\frac{ (1-\alpha)}{1-\beta}\frac{-\frac{1}{ \erfinv((1-\alpha)/(1-\beta))^2}+\frac{2\frac{1-\beta}{\alpha-\beta} e^{-\lp \erfinv\lp\frac{1-\alpha}{1-\beta}\rp\rp^2}}{\sqrt{\pi}\erfinv((1-\alpha)/(1-\beta))}-2}{2 \sqrt{2} (\alpha-\eta\beta)}\nonumber \\
& < & \sqrt{2}\frac{ (1-\alpha)}{1-\beta}\frac{-\frac{1}{ \erfinv((1-\alpha)/(1-\beta))^2}+
\sqrt{1+\frac{2}{\lp \erfinv\left (\frac{1-\alpha}{1-\beta}\right )\rp ^2}}-1}{2 \sqrt{2} (\alpha-\eta\beta)}.\nonumber \\
& < & 0.
\end{eqnarray}
In \cite{Stojnicl1RegPosasymldp} we showed that for any fixed $\alpha\in (0,1)$, $\lim_{\beta\rightarrow \alpha} (\alpha-\eta\beta)\xi^{(p)}_{\alpha,\eta}(\beta)-1=-1$ which then implies that for any fixed $\eta\in (0,1)$ and any fixed $\alpha\in (0,1)$  one also has $\lim_{\beta\rightarrow \alpha} \xi^{(p)}_{\alpha,\eta}(\beta)-1=-1$. Moreover, in \cite{Stojnicl1RegPosasymldp} we also showed that for any fixed $\alpha\in (0,1)$, $\xi^{(p)}_{\alpha,\eta}(0)-1>0$. Together with (\ref{eq:propxi4}) this is then enough to conclude that for any fixed $\eta\in (0,1)$ and any fixed $\alpha\in (0,1)$ there is a unique $\beta$ such that $\xi^{(p)}_{\alpha,\eta}(\beta)=1$, which as mentioned above essentially means that (\ref{eq:thmweaktheta2}) is an unambiguous PT characterization. For the completeness, in Figure \ref{fig:propxi} we present a few numerical results related to the behavior of $\xi^{(p)}_{\alpha,\eta}(\beta)$ that indeed confirm the above calculations.
\begin{figure}[htb]
\begin{minipage}[b]{.5\linewidth}
\centering
\centerline{\epsfig{figure=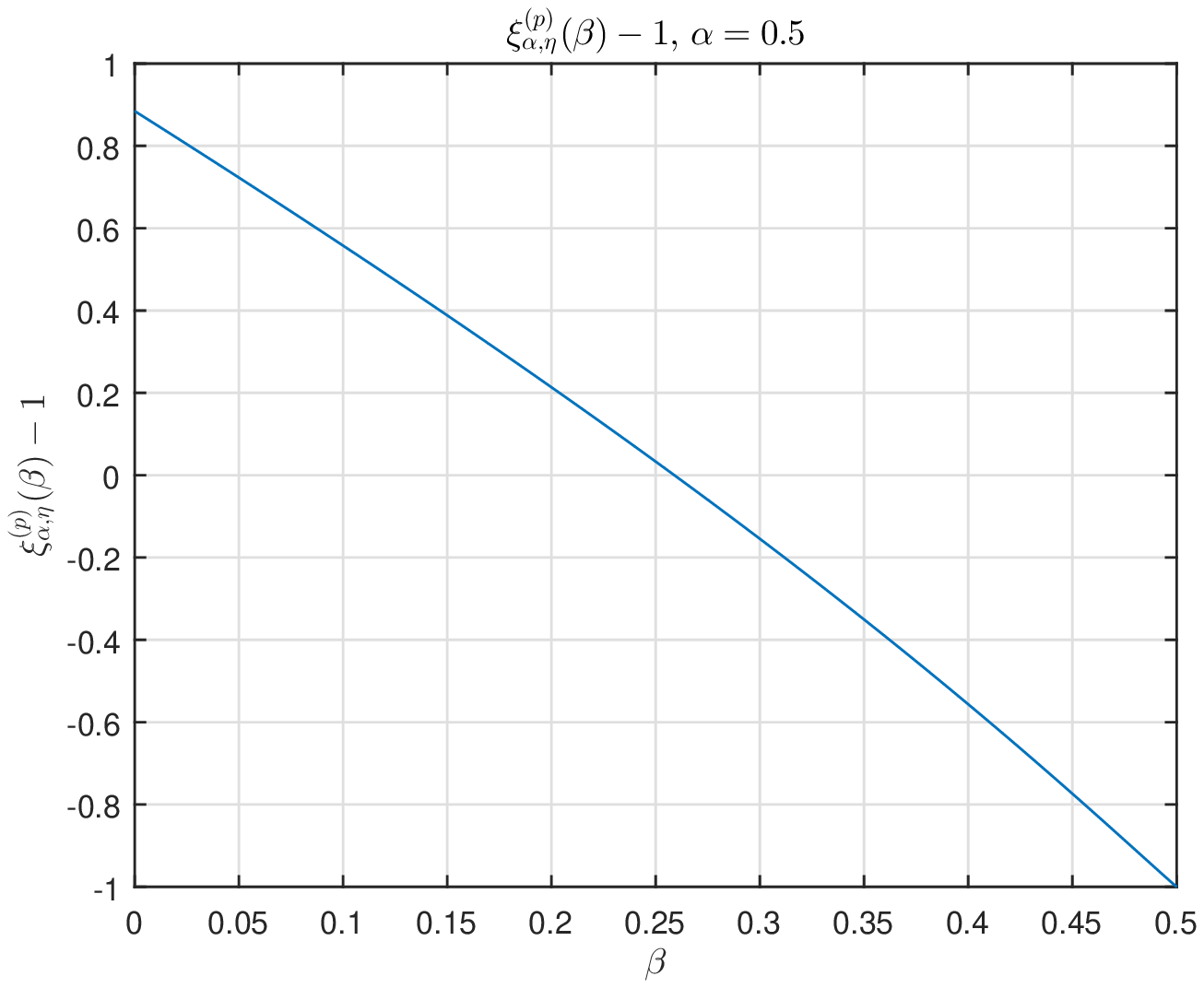,width=9cm,height=7cm}}
\end{minipage}
\begin{minipage}[b]{.5\linewidth}
\centering
\centerline{\epsfig{figure=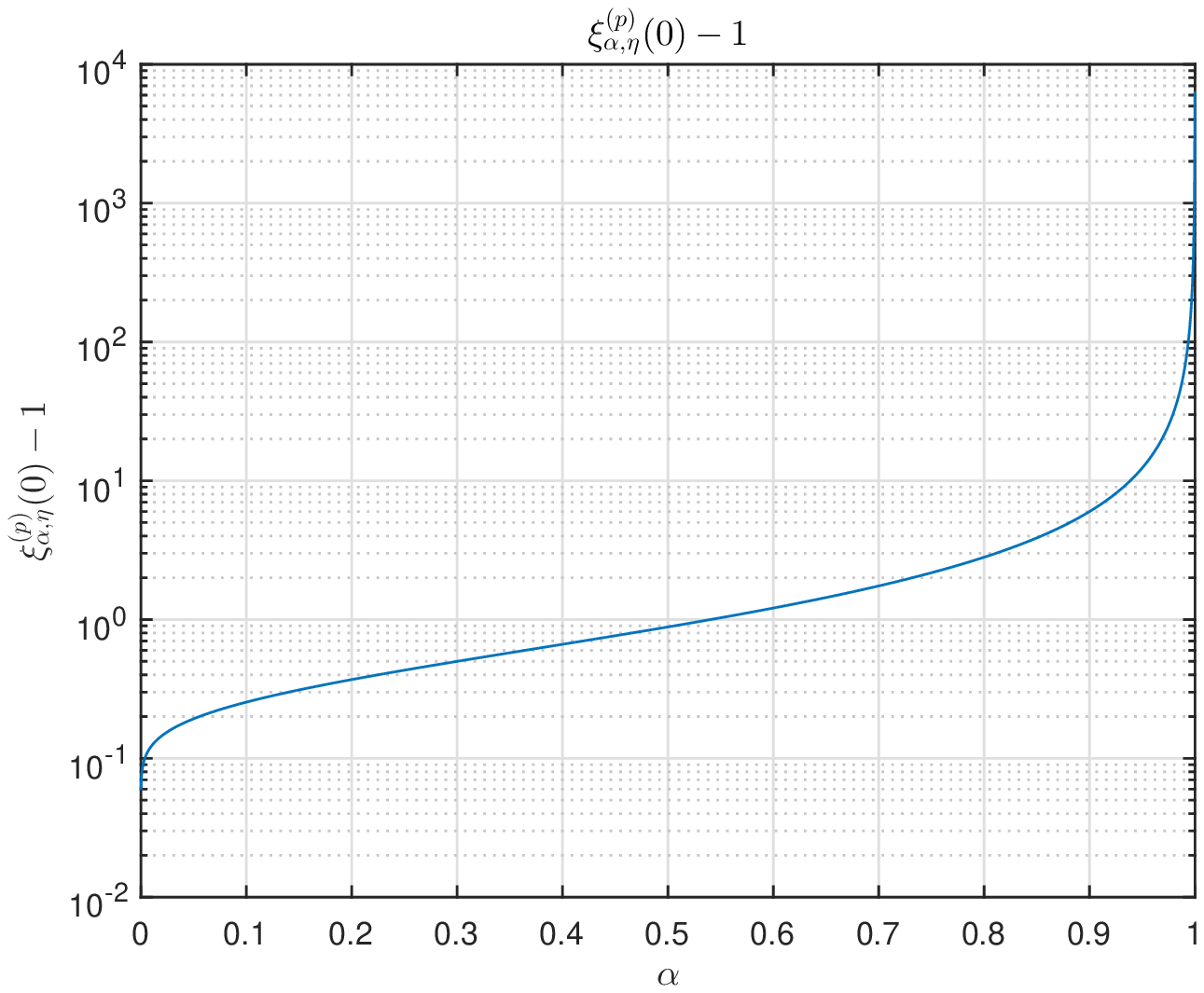,width=9cm,height=7cm}}
\end{minipage}
\caption{Properties of $\xi^{(p)}_{\alpha,\eta}(\beta)$: $(\xi^{(p)}_{\alpha,\eta}(\beta)-1)$ as a function of $\beta$ ($\alpha=0.5$, $\eta=0.5$) -- left; $(\xi^{(p)}_{\alpha,\eta}(0)-1)$ as a function of $\alpha$ ($\eta=0.5$) -- right}
\label{fig:propxi}
\end{figure}

\subsubsection{$\psi^{(p)}_{\beta,\eta}(\alpha)$}
\label{sec:proppsixi1}

In this subsection we will look at $\psi^{(p)}_{\beta,\eta}(\alpha)$ and show that the property proved for $\xi^{(p)}_{\alpha,\eta}(\beta)$ in the previous subsection holds for $\psi^{(p)}_{\beta,\eta}(\alpha)$ as well. Namely, for any fixed $\eta\in (0,1)$ and any fixed $\beta\in (0,1)$ there is a unique $\alpha$ such that $\psi^{(p)}_{\beta,\eta}(\alpha)=1$. This essentially ensures that the $\ell_1$'s fundamental PT from the above theorem is also unambiguous when viewed as a function of $\alpha$. To show this it is then enough to show that for any fixed $\eta\in (0,1)$ and any fixed $\beta\in (0,1)$, $\psi^{(p)}_{\beta,\eta}(\alpha)-1$ is an increasing function of $\alpha$ on interval $(\beta,1]$. We then proceed in a fashion similar to the one from Section \ref{sec:propxi1} and compute the derivative with respect to $\alpha$.
\begin{eqnarray}\label{eq:proppsixi1}
  \frac{d(\psi^{(p)}_{\beta,\eta}(\alpha)-1)}{d\alpha} & = & \frac{d\lp(1-\beta)\frac{\sqrt{\frac{2}{\pi}}e^{-\lp\erfinv\lp\frac{1-\alpha}{1-\beta}\rp\rp^2}}{(\alpha-\eta\beta)\sqrt{2}\erfinv (\frac{1-\alpha}{1-\beta})}-1\rp}{d\alpha}\nonumber\\
   & = & \frac{2 (\beta-1) e^{-\left (\erfinv\left (\frac{1-\alpha}{1-\beta}\right )\right )^2} \left (\erfinv\left (\frac{1-\alpha}{1-\beta}\right )\right )
  +\sqrt{\pi} (\alpha-\eta\beta) \left (2 \left (\erfinv\left (\frac{1-\alpha}{1-\beta}\right )\right )^2+1\right )}{2 (\alpha-\eta\beta)^2 \sqrt{\pi} \left (\erfinv\left (\frac{1-\alpha}{1-\beta}\right )\right )^2}.\nonumber \\
\end{eqnarray}
Utilizing $q$ from (\ref{eq:propxi3}) we then have
\begin{equation}\label{eq:proppsixi3}
  \frac{\frac{2 (\beta-1)}{e^{\left (\erfinv\left (\frac{1-\alpha}{1-\beta}\right )\right )^2}} \left (\erfinv\left (\frac{1-\alpha}{1-\beta}\right )\right )
  +\sqrt{\pi} (\alpha-\eta\beta) \left (2 \left (\erfinv\left (\frac{1-\alpha}{1-\beta}\right )\right )^2+1\right )}{2 (\alpha-\eta\beta)^2 \sqrt{\pi} \left (\erfinv\left (\frac{1-\alpha}{1-\beta}\right )\right )^2}   =  \frac{2 \frac{(\alpha-1)e^{-q^2} q}{\erf(q)}
  +\sqrt{\pi} (\alpha-\eta\beta) \left (2q^2+1\right )}{2 (\alpha-\eta\beta)^2 \sqrt{\pi} q^2}.
\end{equation}
Moreover,
\begin{eqnarray}\label{eq:proppsixi4}
\frac{2 \frac{(\alpha-1)e^{-q^2} q}{\erf(q)}
  +\sqrt{\pi} (\alpha-\eta\beta) \left (2q^2+1\right )}{2 (\alpha-\eta\beta)^2 \sqrt{\pi} q^2}
  & \geq &
  \frac{2 \frac{(\alpha-1)e^{-q^2} q}{\erf(q)}
  +\sqrt{\pi} (\alpha-\beta) \left (2q^2+1\right )}{2 (\alpha-\eta\beta)^2 \sqrt{\pi} q^2}\nonumber \\
  & = &
  \frac{2 \frac{(\alpha-1)e^{-q^2} q}{\erf(q)}
  +\sqrt{\pi} ((\alpha-1)+(1-\beta)) \left (2q^2+1\right )}{2 (\alpha-\eta\beta)^2 \sqrt{\pi} q^2}\nonumber \\
  & = &
  (1-\alpha)\lp \frac{-2 \frac{e^{-q^2} q}{\erf(q)}
  +\sqrt{\pi} \lp -1+\frac{1}{\erf(q)}\rp \left (2q^2+1\right )}{2 (\alpha-\eta\beta)^2 \sqrt{\pi} q^2}\rp\nonumber \\
  & = &
  (1-\alpha)\lp \frac{-2 \frac{e^{-q^2} q}{\erf(q)}
  +\sqrt{\pi} \lp \frac{\erfc(q)}{\erf(q)}\rp \left (2q^2+1\right )}{2 (\alpha-\eta\beta)^2 \sqrt{\pi} q^2}\rp\nonumber \\
  & = &
  (1-\alpha)\lp \frac{-2 e^{-q^2} q
  +\sqrt{\pi} \erfc(q) \left (2q^2+1\right )}{2 (\alpha-\eta\beta)^2 \sqrt{\pi} q^2\erf(q)}\rp\nonumber \\
  & > &
  (1-\alpha)\lp \frac{-2 e^{-q^2}q
  +\frac{2 e^{-q^2}\left (2q^2+1\right )}{\left (q+\sqrt{q^2+2}\right )}}{2 (\alpha-\eta\beta)^2 \sqrt{\pi} q^2\erf(q)}\rp\nonumber \\
  & = &
  (1-\alpha)\frac{2 e^{-q^2}}{\left (q+\sqrt{q^2+2}\right )}\lp \frac{(1+q^2-q\sqrt{q^2+2})}{2 (\alpha-\eta\beta)^2 \sqrt{\pi} q^2\erf(q)}\rp\nonumber \\
  & >& 0,
\end{eqnarray}
where the second to last inequality follows as an application of the first inequality in (\ref{eq:propxi2}). Connecting (\ref{eq:proppsixi1}), (\ref{eq:proppsixi3}), and (\ref{eq:proppsixi4}) we finally have
\begin{equation}\label{proppsoxi5}
    \frac{d(\psi^{(p)}_{\beta,\eta}(\alpha)-1)}{d\alpha}>0,
\end{equation}
and the function $(\psi^{(p)}_{\beta,\eta}(\alpha)-1)$ is indeed increasing on $(\beta,1]$.
Also, for any fixed $\eta\in (0,1)$ and any fixed $\beta\in (0,1)$, in \cite{Stojnicl1RegPosasymldp} we showed that $\lim_{\alpha\rightarrow \beta}(\alpha-\eta\beta)\psi^{(p)}_{\beta,\eta}(\alpha)-1=-1$. This then easily implies that $\lim_{\alpha\rightarrow \beta}\psi^{(p)}_{\beta,\eta}(\alpha)-1=-1$. Finally, in \cite{Stojnicl1RegPosasymldp}, we also argued that for $\eta=0$ and any fixed $\beta\in (0,1)$, $\lim_{\alpha\rightarrow 1}\psi^{(p)}_{\beta,\eta}(\alpha)-1=\infty >0$. This then easily implies that for any fixed $\eta\in (0,1)$ and any fixed $\beta\in (0,1)$, $\lim_{\alpha\rightarrow 1}\psi^{(p)}_{\beta,\eta}(\alpha)-1=\infty >0$. Together with the above proven increasing property this then shows that for any fixed $\eta\in (0,1)$ and any fixed $\beta\in (0,1)$ there is a unique $\alpha$ such that $\psi^{(p)}_{\beta,\eta}(\alpha)=1$, which reconfirms that the partial $\ell_1$'s fundamental PT characterization is unambiguous. For the completeness, in Figure \ref{fig:proppsi} we present a few numerical results related to the behavior of $\psi^{(p)}_{\beta,\eta}(\alpha)$ that are indeed in agreement with the above calculations.
\begin{figure}[htb]
\begin{minipage}[b]{.5\linewidth}
\centering
\centerline{\epsfig{figure=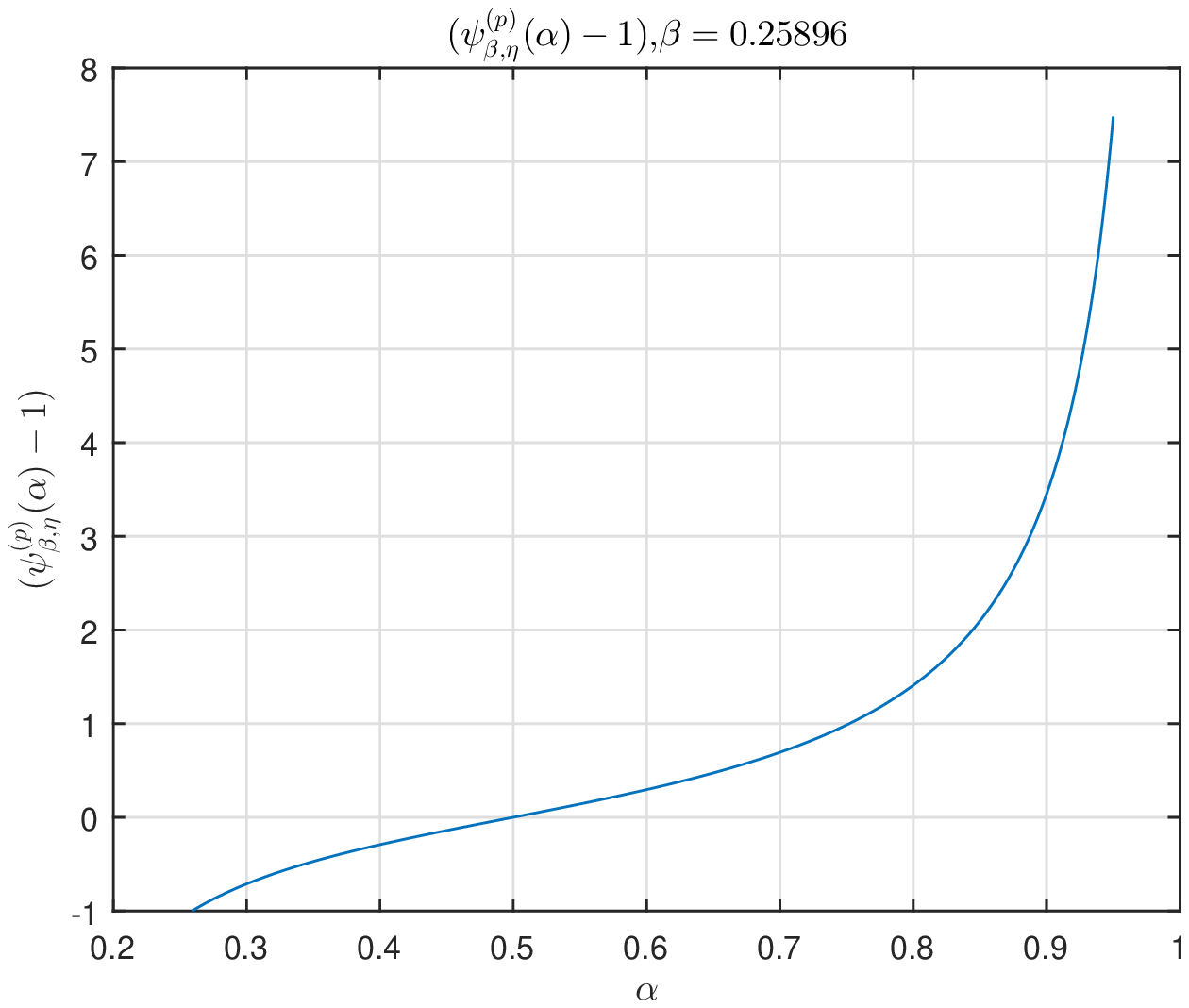,width=9cm,height=7cm}}
\end{minipage}
\begin{minipage}[b]{.5\linewidth}
\centering
\centerline{\epsfig{figure=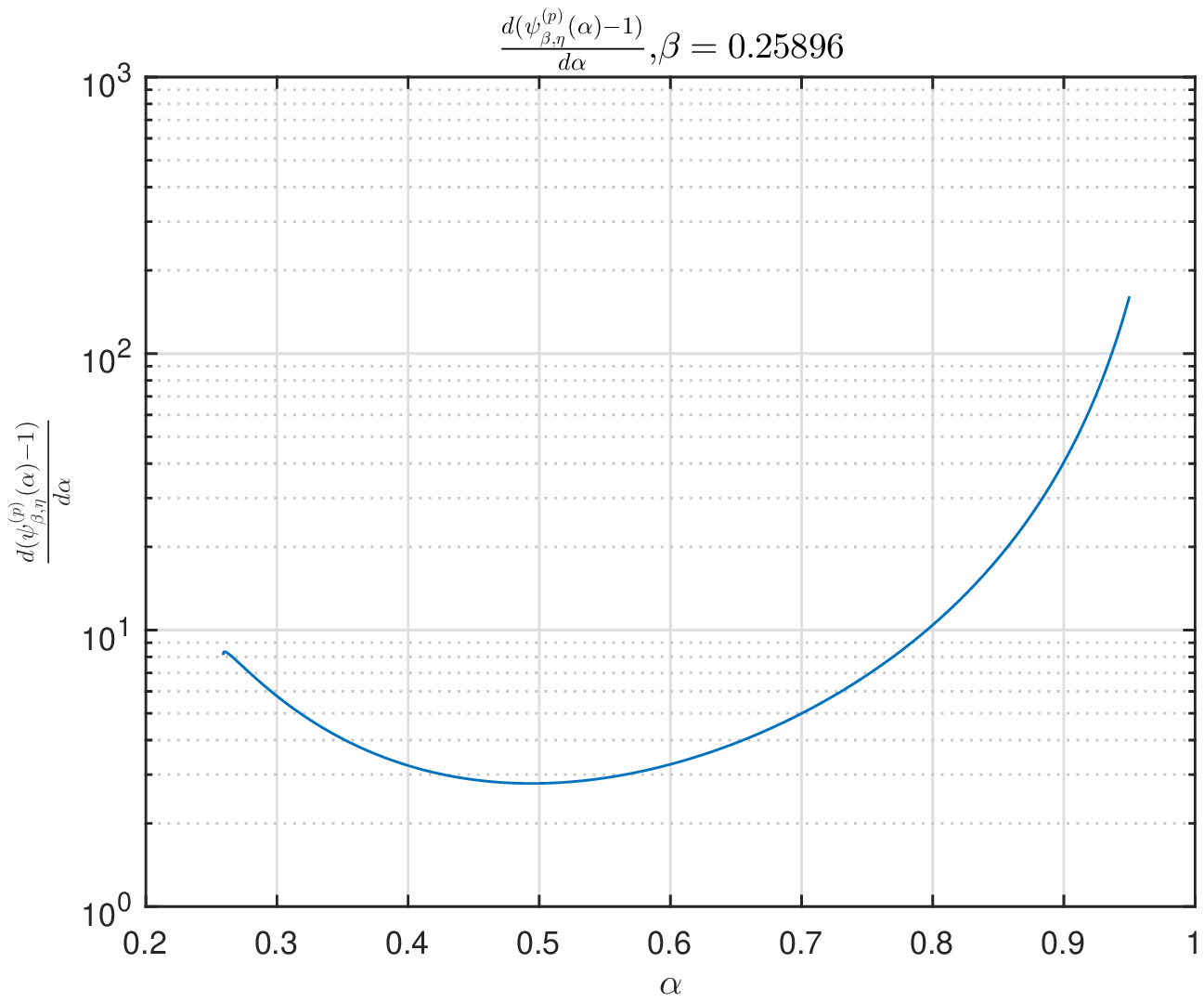,width=9cm,height=7cm}}
\end{minipage}
\caption{Properties of $\psi^{(p)}_{\beta,\eta}(\alpha)$: $\psi^{(p)}_{\beta,\eta}(\alpha)-1$ as a function of $\alpha$ ($\beta=0.25896$) -- left; $\frac{d(\psi^{(p)}_{\beta,\eta}(\alpha)-1)}{d\alpha}$ as a function of $\alpha$ ($\beta=0.25896$) -- right}
\label{fig:proppsi}
\end{figure}

Finally, to give a little bit of a flavor as to what is actually proven in Theorem \ref{thm:thmweakthr} we in Figure \ref{fig:weakl1PT} show the theoretical PT curves that one can obtain based on (\ref{eq:thmweaktheta2}) for several different values of $\eta$. One easily observes that as $\eta$ increases the recoverable sparsity increases as well. In other words, the larger the size of the known subset of $supp(\x)$ the larger the cardinality of $supp(\x)$ that can be recovered through $\ell_1$ as well.
\begin{figure}[htb]
\centering
\centerline{\epsfig{figure=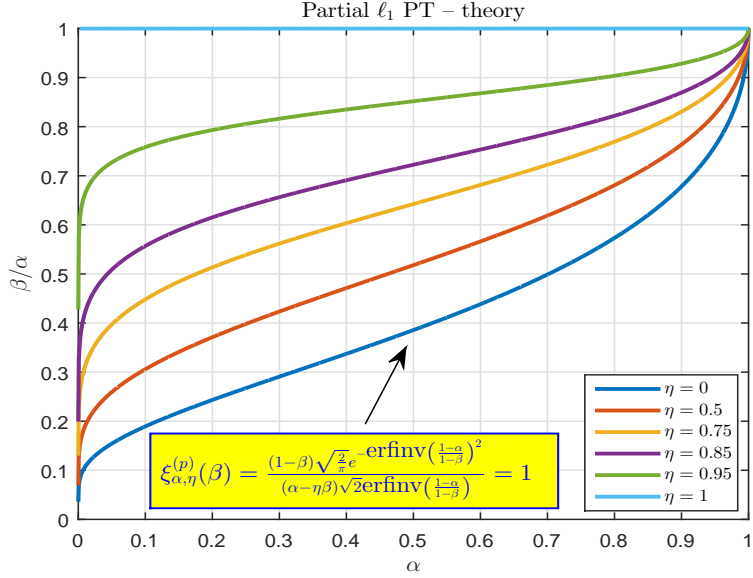,width=11.5cm,height=8cm}}
\caption{Partial $\ell_1$'s weak PT; $\{(\alpha,\beta)|\xi^{(p)}_{\alpha,\eta}(\beta)=1\}$}
\label{fig:weakl1PT}
\end{figure}


\section{Partial $\ell_1$ -- large deviations}
\label{sec:ldp}

In this section we substantially expand the phase transition considerations from the previous section. We will do so by connecting them to the large deviation principles (LDP) and then providing the LDPs thorough characterizations. In a nutshell, the LDP characterizations will provide a significantly richer spectrum of information about the PTs. In addition to determining the PT curves that the standard PT characterizations do, the LDP ones explain the algorithms behaviour in the entire transition zone. They essentially determine the rate at which the probabilities of algorithms' success (failure) tend to zero as the systems dimensions deviate from the ones that satisfy the PT curves (i.e. the breaking points of the algorithms' success). To fully characterize these rates, we will essentially determine them for any point in $(\alpha,\beta)$ plane (also, and as is probably expected, to ensure that the results make sense we will assume $\{(\alpha,\beta)| \beta<\alpha,\alpha\in(0,1)\}$). To achieve these characterizations we will first recall on the connection between the PTs and the LDPs that we established in \cite{Stojnicl1RegPosasymldp} and then try to emulate the strategies designed in \cite{Stojnicl1RegPosasymldp}. To facilitate the following we will, as usual, try to make everything look as parallel to what was done in \cite{Stojnicl1RegPosasymldp} as possible. To that end we will try to skip all the steps that remain the same as in \cite{Stojnicl1RegPosasymldp} and highlight the different ones. Before doing any of that we will start things off by recalling on a couple of results that we established in \cite{StojnicCSetam09,StojnicICASSP09,StojnicICASSP10knownsupp}. These are of course among the key unsung heros of all the success that we achieved in designing our probabilistic approach for characterizing PTs and LDPs.

As was done in \cite{Stojnicl1RegPosasymldp}, for the simplicity and without loss of generality we will assume that the elements $\x_{1},\x_{2},\dots,\x_{n-k}$ of $\x$ are equal to zero and that the elements $\x_{n-k+1},\x_{n-k+2},\dots,\x_n$ have fixed signs, say they all are positive (this is of course in an agreement with the requirement that the weak phase transition imposes). The following was proved in \cite{StojnicICASSP10knownsupp,StojnicTowBettCompSens13} relying on the breakthrough observations of \cite{StojnicCSetam09,StojnicICASSP09} and as mentioned above is one of the key features that enabled us to run the entire machinery developed in \cite{StojnicCSetam09,StojnicICASSP09,StojnicICASSP10knownsupp,StojnicTowBettCompSens13}.
\begin{theorem}(\cite{StojnicICASSP10knownsupp,StojnicTowBettCompSens13} Nonzero elements of $\x$ have fixed signs and location)
Assume that an $m\times n$ measurement matrix $A$ is given. Let $\x$
be a $k$ sparse vector. Also let $\x_1=\x_2=\dots=\x_{n-k}=0$. Let the signs of $\x_{n-k+1},\x_{n-k+2},\dots,\x_n$ be fixed, say all positive. Also, let it be known to the algorithm given in (\ref{eq:l1}) that $\x_{n-\eta k+1},\x_{n-\eta k+2},\dots,\x_n$ are among the $k$ non-zero components of $\x$, i.e. let $\Pi=\{n-\eta k+1,n-\eta k+2,\dots,n\}$, where $0\leq\eta\leq 1$.
Further, assume that $\y\triangleq A\x$ and that $\w$ is
an $n\times 1$ vector. If
\begin{equation}
(\forall \w\in \textbf{R}^n | A\w=0) \quad  -\sum_{i=n-k+1}^{n-\eta k} \w_i<\sum_{i=1}^{n-k}|\w_{i}|,
\end{equation}
then the solutions of (\ref{eq:l1imp}) and (\ref{eq:l0}) coincide. Moreover, if
\begin{equation}
(\exists \w\in \textbf{R}^n | A\w=0) \quad  -\sum_{i=n-k+1}^{n-\eta k} \w_i\geq \sum_{i=1}^{n-k}|\w_{i}|,
\label{eq:thmeqgen}
\end{equation}
then there will be a $k$-sparse $\x$ (from the set of $\x$'s with fixed locations and signs of nonzero components) with a priori known set of locations $\Pi$ that is the solution of (\ref{eq:l0}) and is not the solution of (\ref{eq:l1imp}).
\label{thm:thmknownsuppcond}
\end{theorem}
To facilitate the exposition we set
\begin{equation}
\Sw\triangleq\{\w\in S^{n-1}| \quad -\sum_{i=n-k+1}^{n-\eta k} \w_i<\sum_{i=1}^{n-k}|\w_{i}|\}.\label{eq:defSwpr}
\end{equation}
As was done in \cite{Stojnicl1RegPosasymldp} we will split the LDP analysis into two parts, the first main one that deals with the so-called upper tail of the LDP characterizations and the second one that deals with the corresponding lower tail. Also, as was the case in \cite{Stojnicl1RegPosasymldp}, it will turn out that the upper tail analysis with minimal adaptations automatically settles the lower tail as well. Consequently, we will mostly focus on the upper tail and once we have those results established we will quickly transform them to cover the lower tail as well.

\subsection{Upper tail}
\label{sec:uppertail}

As mentioned above, we will first consider points $(\alpha,\beta)$ such that $\alpha\geq \alpha_w$ where $\alpha_w$ is such that $\psi^{(p)}_{\beta,\eta}(\alpha_w)=\xi^{(p)}_{\alpha_w,\eta}(\beta)=1$. These points will establish what we will refer to as the LDPs upper tail (or for short just the upper tail). The remaining ones will be briefly discussed in a section later on and they will establish what we will refer to as the lower tail.

As in \cite{Stojnicl1RegPosasymldp}, we will assume that the elements of $A$ are i.i.d. standard normals and will be interested in the following probability
\begin{equation}
P_{err}\triangleq P(\min_{\w\in S_w}\|A\w\|_2\leq 0)=P(\max_{\w\in S_w}\min_{\|\y\|_2=1}(\y^T A\w )\geq 0).
\label{eq:ldpprob}
\end{equation}
where $P_{err}$ is the so-called probability of error/failure, i.e. the probability that (\ref{eq:l1imp}) fails to produce the solution of (\ref{eq:l0}). In \cite{Stojnicl1RegPosasymldp} we, for any $c_3\geq 0$, established the following
\begin{equation}
P_{err}\leq e^{-\frac{c_3^2}{2}}E\max_{\w\in S_w}\min_{\|\y\|_2=1}e^{-c_3(\y^T A\w+g)}\leq
e^{-\frac{c_3^2}{2}}Ee^{-c_3\|\g\|_2}Ee^{c_3w(\h,S_w)},
\label{eq:ldpprob3}
\end{equation}
where
\begin{equation}
w(\h,\Sw)\triangleq\max_{\w\in \Sw} (\h^T\w), \label{eq:widthdefSw}
\end{equation}
and the elements of $\h$ are i.i.d. standard normals. To characterize the right hand side of (\ref{eq:ldpprob3}) one then focuses, as in \cite{Stojnicl1RegPosasymldp}, on $w(\h,S_w)$. In \cite{StojnicCSetam09,StojnicCSetamBlock09,StojnicBlockasymldpfinn15,StojnicLiftStrSec13,StojnicICASSP10knownsupp,StojnicTowBettCompSens13} we developed a super powerful mechanism that enables a very elegant and useful representation for $w(\h,S_w)$. Instead of repeating all the details we just present the final, neat results that we will utilize here.

We start by setting
\begin{equation}
\hw\triangleq(|\h_1|,|\h_2|,\dots,|\h_{n-k}|,
-\h_{n-k+1},-\h_{n-k+2},\dots,-\h_{n})^T,\label{eq:defhweak}
\end{equation}
then one can characterize $w(\h,\Sw)$ in (\ref{eq:widthdefSw}) in the following way
\begin{eqnarray}
w(\h,\Sw) = \max_{\bar{\y}\in \mR^{n}} & &  \sum_{i=1}^{n} \hw_i \bar{\y}_i\nonumber \\
\mbox{subject to} &  & \bar{\y}_i\geq 0, 0\leq i\leq n-k\nonumber \\
& & \sum_{i=n-k+1}^{n-\eta k}\bar{\y}_i\geq \sum_{i=1}^{n-k} \bar{\y}_i \nonumber \\
& & \sum_{i=1}^{n}\bar{\y}_i^2\leq 1\label{eq:workww2}
\end{eqnarray}
where $\hw_i$ is the $i$-th element of $\hw$ and $\bar{\y}_i$ is the $i$-th element of $\bar{\y}$. Solving (\ref{eq:workww2}) as was done in \cite{StojnicCSetam09,StojnicCSetamBlock09,StojnicBlockasymldpfinn15,StojnicLiftStrSec13,StojnicICASSP10knownsupp,StojnicTowBettCompSens13} one obtains
\begin{eqnarray}
w(\h,\Sw) = -\max_{\nu\geq 0,\gamma\geq 0}\min_{\bar{\y}} & & \sum_{i=1}^{n} -\hw_i \bar{\y}_i+\nu\sum_{i=1}^{n-k}\bar{\y}_i
-\nu\sum_{i=n-k+1}^{n-\eta k}\bar{\y}_i+\gamma\sum_{i=1}^{n}\bar{\y}_i^2-\gamma\nonumber \\
\mbox{subject to} & & \bar{\y}_i\geq 0, 0\leq i\leq n-k.\label{eq:ldpwhSw0}
\end{eqnarray}
Finally, one also has
\begin{eqnarray}
w(\h,\Sw) & = & \min_{\nu\geq0,\gamma\geq 0} \frac{\sum_{i=1}^{n-k}\max(\hw_i-\nu,0)^2+\sum_{i=n-k+1}^{n}(\hw_i+\nu)^2+\sum_{i=n-\eta k+1}^{n}(\hw_i)^2}{4\gamma}+\gamma\nonumber \\
& = & \min_{\nu\geq0}\sqrt{\sum_{i=1}^{n-k}\max(\hw_i-\nu,0)^2+\sum_{i=n-k+1}^{n-\eta k}(\hw_i+\nu)^2+\sum_{i=n-\eta k+1}^{n}(\hw_i)^2}.\label{eq:ldpwhSw}
\end{eqnarray}
We summarize the above methodology to upper bound $P_{err}$ in the following theorem.
\begin{theorem}
Let $A$ be an $m\times n$ matrix in (\ref{eq:l0})
with i.i.d. standard normal components. Let
the unknown $\x$ in (\ref{eq:l0}) be $k$-sparse and let the location and the signs of nonzero elements of $\x$ be arbitrarily chosen but fixed. Assume that the locations of $\eta k$ ($0\leq\eta\leq 1$) of the non-zero elements are arbitrarily chosen, fixed, and a priori known, and let $\Pi$ be the set of those locations. Let $P_{err}$ be the probability that the solution of (\ref{eq:l1imp}) is not the solution of (\ref{eq:l0}). Then
\begin{eqnarray}
P_{err}& \leq & \min_{c_3\geq 0}e^{-\frac{c_3^2}{2}}e^{-c_3\|\g\|_2}Ee^{c_3w(\h,S_w)}\nonumber \\
& = & \min_{c_3\geq 0}\left (e^{-\frac{c_3^2}{2}}\frac{1}{\sqrt{2\pi}^m}\int_{\g}e^{-\sum_{i=1}^{m}\g_i^2/2-c_3\|\g\|_2}d\g \min_{\nu\geq 0,\gamma\geq\frac{c_3}{2}} w_1^{n-k}w_2^{k(1-\eta)}w_3^{k\eta}e^{c_3\gamma}\right ),
\label{eq:ldpthm1perrub1}
\end{eqnarray}
where
\begin{eqnarray}
w_1 &=& \frac{1}{\sqrt{2\pi}}\int_{\bar{h}}e^{-\bar{h}^2/2}e^{c_3\max(|\bar{h}|-\nu,0)^2/4/\gamma}d\bar{h}
  =\frac{e^{\frac{c_3\nu^2/4/\gamma}{1-c_3/2/\gamma}}}{\sqrt{1-c_3/2/\gamma}}\erfc\left (\frac{\nu}{\sqrt{2}\sqrt{1-c_3/2/\gamma}}\right )+\erf\left (\frac{\nu}{\sqrt{2}}\right )\nonumber \\
w_2 &=& \frac{1}{\sqrt{2\pi}}\int_{\bar{h}}e^{-\bar{h}^2/2}e^{c_3(\bar{h}+\nu)^2/4/\gamma}d\bar{h}
  =\frac{e^{\frac{c_3\nu^2/4/\gamma}{1-c_3/2/\gamma}}}{\sqrt{1-c_3/2/\gamma}}\nonumber \\
w_3 &=& \frac{1}{\sqrt{2\pi}}\int_{\bar{h}}e^{-\bar{h}^2/2}e^{c_3(\bar{h})^2/4/\gamma}d\bar{h}
  =\frac{1}{\sqrt{1-c_3/2/\gamma}}.\label{eq:ldpthm1perrub2}
\end{eqnarray}\label{thm:ldp1}
\end{theorem}
\begin{proof}
Follows from the above considerations and ultimately through an adaptation of the mechanisms developed in \cite{StojnicCSetam09,StojnicCSetamBlock09,StojnicBlockasymldpfinn15,StojnicLiftStrSec13,StojnicICASSP10knownsupp,StojnicTowBettCompSens13}.
\end{proof}
The upper bound given in the above theorem is valid for any integers $m$, $k$, and $n$ (provided $k\leq m\leq n$ so that the results make sense). Our main concern below though is the asymptotic regime, basically the same one as in Theorem \ref{thm:thmweakthr}. In particular, and following \cite{Stojnicl1RegPosasymldp}, we will be interested in the rate, $I^{(p)}_{err}(\alpha,\beta)$, at which $P_{err}$ decays
\begin{equation}\label{eq:ldpasymp1}
  I^{(p)}_{err}(\alpha,\beta)\triangleq\lim_{n\rightarrow\infty}\frac{\log{P_{err}}}{n}.
\end{equation}
One now clearly recognizes that the $I^{(p)}_{err}(\alpha,\beta)$ essentially emulates the so-called LDP's rate (indicator) function. Based on Theorem \ref{thm:ldp1} we have the following LDP type of theorem.
\begin{theorem}
Assume the setup of Theorem \ref{thm:ldp1}. Further, let integers $m$, $k$, and $n$ be large ($k\leq m\leq n$) such that $\beta=\frac{k}{n}$ and $\alpha=\frac{m}{n}$ are constants independent of $n$. Assume that a pair $(\alpha,\beta)$  is given. Also, assume the following scaling: $c_3\rightarrow c_3\sqrt{n}$ and $\gamma\rightarrow\gamma\sqrt{n}$. Then
\begin{eqnarray}
I^{(p)}_{err}(\alpha,\beta)& \triangleq &\lim_{n\rightarrow\infty}\frac{\log{P_{err}}}{n}\nonumber \\
& \leq & \min_{c_3\geq 0}\left (-\frac{(c_3)^2}{2}+I_{sph}+\min_{\nu\geq 0,\gamma\geq \frac{c_3}{2}} ((1-\beta)\log{w_1}+(1-\eta)\beta\log{w_2}+\eta\beta\log{w_3}+c_3\gamma)\right )\nonumber \\
& \triangleq &I_{err,u}^{(p,ub)}(\alpha,\beta),
\label{eq:ldpthm2Ierrub1}
\end{eqnarray}
where
\begin{eqnarray}
I_{sph} &=& \widehat{\gamma}c_3-\frac{\alpha }{2}\log\left (1-\frac{c_3}{2\widehat{\gamma}}\right )\nonumber \\
  \widehat{\gamma} &=& \frac{c_3-\sqrt{(c_3)^2+4\alpha}}{4}\nonumber \\
w_1 &=& \frac{1}{\sqrt{2\pi}}\int_{\bar{h}}e^{-\bar{h}^2/2}e^{c_3\max(|\bar{h}|-\nu,0)^2/4/\gamma}d\bar{h}
  =\frac{e^{\frac{c_3\nu^2/4/\gamma}{1-c_3/2/\gamma}}}{\sqrt{1-c_3/2/\gamma}}\erfc\left (\frac{\nu}{\sqrt{2}\sqrt{1-c_3/2/\gamma}}\right )+\erf\left (\frac{\nu}{\sqrt{2}}\right )\nonumber \\
  w_2 &=& \frac{1}{\sqrt{2\pi}}\int_{\bar{h}}e^{-\bar{h}^2/2}e^{c_3(\bar{h}+\nu)^2/4/\gamma}d\bar{h}
  =\frac{e^{\frac{c_3\nu^2/4/\gamma}{1-c_3/2/\gamma}}}{\sqrt{1-c_3/2/\gamma}}\nonumber \\
  w_3 &=& \frac{1}{\sqrt{2\pi}}\int_{\bar{h}}e^{-\bar{h}^2/2}e^{c_3(\bar{h})^2/4/\gamma}d\bar{h}
  =\frac{1}{\sqrt{1-c_3/2/\gamma}}.\label{eq:ldpthm2perrub2}
\end{eqnarray}\label{thm:ldp2}
\end{theorem}
\begin{proof} Follows in a fashion analogous to the one employed in \cite{Stojnicl1RegPosasymldp}.
\end{proof}

Numerical solution of the above optimization problem would provide the estimates for the rate of $P_{err}$'s decay. However, we here follow \cite{Stojnicl1RegPosasymldp} and raise the bar a bit higher. Instead of relying on solving the above problem numerically we will here try to solve it explicitly. We will again follow as much of the methodology introduced in \cite{Stojnicl1RegPosasymldp} as possible. However, there will be quite a few technical differences that will also result in a behavior fundamentally different from the one observed in \cite{Stojnicl1RegPosasymldp} when it comes to some of the optimizing quantities and ultimately the estimates of the LDPs rate function. Also, we emphasize that among several different ways how one can present the final solutions we chose the one that makes them look as analogous to the corresponding ones from \cite{Stojnicl1RegPosasymldp} as possible.

\subsubsection{Handling $I_{err,u}^{(p,ub)}$}
\label{sec:analysisIerr}

We will first set
\begin{equation}\label{eq:detanalIerr1}
  A_{0}\triangleq\sqrt{1-\frac{c_3}{2\gamma}},
\end{equation}
and then consider the following optimization problem
\begin{equation}\label{eq:detanalIeer2}
I_{err,u}^{(p,ub)}(\alpha,\beta)\triangleq \min_{c_3\geq 0,\nu\geq 0,A_0\leq 1}\zeta_{\alpha,\beta}(c_3,\nu,A_0),
\end{equation}
where
\begin{eqnarray}
\zeta_{\alpha,\beta}(c_3,\nu,A_0)&=&\left (-\frac{c_3^2}{2}+I_{sph}+(1-\beta)\log{w_1}+\beta(1-\eta)\log{w_2}+\beta\eta\log{w_3}+\frac{c_3^2}{2(1-A_0^2)}\right )\nonumber \\
I_{sph} &=& \widehat{\gamma}c_3-\frac{\alpha }{2}\log\left (1-\frac{c_3}{2\widehat{\gamma}}\right )\nonumber \\
  \widehat{\gamma} &=& \frac{c_3-\sqrt{(c_3)^2+4\alpha}}{4}\nonumber \\
w_1 &=&
  \frac{e^{\frac{(1-A_0^2)\nu^2}{2A_0^2}}}{A_0}\erfc\left (\frac{\nu}{\sqrt{2}A_0}\right )+\erf\left (\frac{\nu}{\sqrt{2}}\right )\nonumber \\
  w_2 &=&
  \frac{e^{\frac{(1-A_0^2)\nu^2}{2A_0^2}}}{A_0}\nonumber \\
    w_3 &=&
  \frac{1}{A_0}.\label{eq:detanalIeer3}
\end{eqnarray}
Our goal below will be to compute the derivatives of $\zeta_{\alpha,\beta}(c_3,\nu,A_0)$ with respect to $c_3$, $\nu$, and $A_0$ to eventually solve
\begin{equation}\label{eq:detanalIeer3a}
  \frac{d \zeta_{\alpha,\beta}(c_3,\nu,A_0)}{dc_3}=  \frac{d \zeta_{\alpha,\beta}(c_3,\nu,A_0)}{d\nu}=  \frac{d \zeta_{\alpha,\beta}(c_3,\nu,A_0)}{dA_0}=0.
\end{equation}
We will start with the derivative with respect to $c_3$ and observe that for this derivative from \cite{Stojnicl1RegPosasymldp} one has
\begin{equation}
\frac{d\zeta_{\alpha,\beta}(c_3,\nu,A_0)}{dc_3}=-c_3+\frac{c_3}{1-A_0^2}+\frac{c_3-\sqrt{(c_3)^2+4\alpha}}{2}.
\label{eq:detanalIeer9}
\end{equation}
Setting the above derivative to zero implies
\begin{equation}\label{eq:detanalIeer12}
  c_3=\frac{(1-A_0^2)\sqrt{\alpha}}{A_0}.
\end{equation}
For the derivative with respect to $\nu$ we have
\begin{eqnarray}
\frac{d\zeta_{\alpha,\beta}(c_3,\nu,A_0)}{d\nu}&=&\frac{d}{d\nu}\left (-\frac{c_3^2}{2}+I_{sph}+(1-\beta)\log{w_1}+\beta(1-\eta)\log{w_2}+\beta\eta\log{w_3}+\frac{c_3^2}{2(1-A_0^2)}\right )\nonumber \\
&=& \frac{\beta(1-A_0^2)\nu(1-\eta)}{A_0^2}+\frac{1-\beta}{w_1}\left (\frac{(1-A_0^2)\nu}{A_0^2}\frac{e^{\frac{(1-A_0^2)\nu^2}{2A_0^2}}}{A_0}\erfc\left (\frac{\nu}{\sqrt{2}A_0}\right )-\frac{e^{\frac{(1-A_0^2)\nu^2}{2A_0^2}}}{A_0^2}\frac{2e^{-\frac{\nu^2}{2A_0^2}}}{\sqrt{2}\sqrt{\pi}}\right)\nonumber \\
&&+\frac{1-\beta}{w_1}\frac{2e^{-\frac{\nu^2}{2}}}{\sqrt{2}\sqrt{\pi}}\nonumber \\
&=& \frac{\beta(1-A_0^2)\nu(1-\eta)}{A_0^2}+\frac{1-\beta}{w_1}\left (\frac{(1-A_0^2)\nu}{A_0^2}\frac{e^{\frac{(1-A_0^2)\nu^2}{2A_0^2}}}{A_0}\erfc\left (\frac{\nu}{\sqrt{2}A_0}\right )-\frac{1-A_0^2}{A_0^2}\frac{2e^{-\frac{\nu^2}{2}}}{\sqrt{2}\sqrt{\pi}}\right)\nonumber \\
&=& \frac{1-A_0^2}{w_1A_0^2}\left (\beta(1-\eta)\nu\erf\left (\frac{\nu}{\sqrt{2}}\right )+\frac{(1-\beta\eta)\nu}{A_0}\frac{\erfc\left (\frac{\nu}{\sqrt{2}A_0}\right )}{e^{-\frac{(1-A_0^2)\nu^2}{2A_0^2}}}-(1-\beta)\frac{2e^{-\frac{\nu^2}{2}}}{\sqrt{2}\sqrt{\pi}}\right)\nonumber \\
&=& \frac{(1-A_0^2)e^{-\frac{\nu^2}{2}}}{w_1A_0^3}\left (\beta(1-\eta)\nu\erf\left (\frac{\nu}{\sqrt{2}}\right )A_0e^{\frac{\nu^2}{2}}+(1-\beta\eta)\nu\frac{\erfc\left (\frac{\nu}{\sqrt{2}A_0}\right )}{e^{-\frac{\nu^2}{2A_0^2}}}-(1-\beta)\sqrt{\frac{2}{\pi}}A_0\right).\nonumber \\
\label{eq:detanalIeer4}
\end{eqnarray}
To further transform the above derivative we will need the derivative with respect to $A_0$ as well. To that end we have
\begin{eqnarray}
\frac{d\zeta_{\alpha,\beta}(c_3,\nu,A_0)}{dA_0}&=&\frac{d}{dA_0}\left (-\frac{c_3^2}{2}+I_{sph}+(1-\beta)\log{w_1}+\beta(1-\eta)\log{w_2}+\beta\eta\log{w_3}+\frac{c_3^2}{2(1-A_0^2)}\right )\nonumber \\
&=& (1-\beta)\frac{d\log{w_1}}{dA_0}+\beta(1-\eta)\left (\nu^2\frac{d}{dA_0}\left (\frac{1-A_0^2}{2A_0^2}\right )-\frac{1}{A_0}\right )-\frac{\beta\eta}{A_0}+\frac{c_3^2A_0}{(1-A_0^2)^2}\nonumber \\
&=& (1-\beta)\frac{d\log{w_1}}{dA_0}-\frac{\beta(1-\eta)\nu^2}{A_0^3}-\frac{\beta A_0^2}{A_0^3}+\frac{c_3^2A_0}{(1-A_0^2)^2}\nonumber \\
&=& (1-\beta)\frac{d\log{w_1}}{dA_0}-\frac{\beta(1-\eta)\nu^2}{A_0^3}+\frac{(\alpha-\beta) A_0^2}{A_0^3}.\nonumber \\
\label{eq:detanalIeer10}
\end{eqnarray}
From \cite{Stojnicl1RegPosasymldp} we have
\begin{eqnarray}
\frac{d\log{w_1}}{dA_0}=\frac{d\log{(\frac{1}{A_0}e^{\frac{\nu^2}{2A_0^2}}\erfc(\frac{\nu}{\sqrt{2}A_0})+e^{\frac{\nu^2}{2}}\erf(\frac{\nu}{\sqrt{2}}))}}{dA_0}=
-\frac{e^{\frac{\nu^2}{2A_0^2}}(A_0^2+\nu^2)\erfc(\frac{\nu}{\sqrt{2}A_0})-\sqrt{\frac{2}{\pi}}A_0\nu}
{A_0^3(e^{\frac{\nu^2}{2A_0^2}}\erfc(\frac{\nu}{\sqrt{2}A_0})+A_0e^{\frac{\nu^2}{2}}\erf(\frac{\nu}{\sqrt{2}}))}.\nonumber \\
\label{eq:detanalIeer11}
\end{eqnarray}
Combining (\ref{eq:detanalIeer10}) and (\ref{eq:detanalIeer11}) one obtains
\begin{eqnarray}
\frac{d\zeta_{\alpha,\beta}(c_3,\nu,A_0)}{dA_0}&=& (1-\beta)\frac{d\log{w_1}}{dA_0}-\frac{\beta(1-\eta)\nu^2}{A_0^3}+\frac{(\alpha-\beta) A_0^2}{A_0^3}\nonumber \\
&=& -(1-\beta)\frac{e^{\frac{\nu^2}{2A_0^2}}(A_0^2+\nu^2)\erfc(\frac{\nu}{\sqrt{2}A_0})-\sqrt{\frac{2}{\pi}}A_0\nu}
{A_0^3(e^{\frac{\nu^2}{2A_0^2}}\erfc(\frac{\nu}{\sqrt{2}A_0})+A_0e^{\frac{\nu^2}{2}}\erf(\frac{\nu}{\sqrt{2}}))}
-\frac{\beta(1-\eta)\nu^2}{A_0^3}+\frac{(\alpha-\beta) A_0^2}{A_0^3}.\nonumber \\
\label{eq:detanalIeer10a}
\end{eqnarray}
Transforming a bit further one also has
\begin{multline}
\frac{d\zeta_{\alpha,\beta}(c_3,\nu,A_0)}{dA_0}
= -(1-\beta)\frac{e^{\frac{\nu^2}{2A_0^2}}(A_0^2+\nu^2)\erfc(\frac{\nu}{\sqrt{2}A_0})-\sqrt{\frac{2}{\pi}}A_0\nu}
{A_0^3(e^{\frac{\nu^2}{2A_0^2}}\erfc(\frac{\nu}{\sqrt{2}A_0})+A_0e^{\frac{\nu^2}{2}}\erf(\frac{\nu}{\sqrt{2}}))}
-\frac{\beta(1-\eta)\nu^2}{A_0^3}+\frac{(\alpha-\beta) A_0^2}{A_0^3} \\
=\frac{((\alpha-1)A_0^2-\nu^2(1-\beta\eta))e^{\frac{\nu^2}{2A_0^2}}\erfc(\frac{\nu}{\sqrt{2}A_0})+(1-\beta)\sqrt{\frac{2}{\pi}}A_0\nu
+((\alpha-\beta)A_0^2-\beta(1-\eta)\nu^2)A_0e^{\frac{\nu^2}{2}}\erf(\frac{\nu}{\sqrt{2}})}
{A_0^3(e^{\frac{\nu^2}{2A_0^2}}\erfc(\frac{\nu}{\sqrt{2}A_0})+A_0e^{\frac{\nu^2}{2}}\erf(\frac{\nu}{\sqrt{2}}))}.
\\
\label{eq:detanalIeer10a}
\end{multline}
Now, setting the derivative with respect to $\nu$ in (\ref{eq:detanalIeer11}) to zero gives
\begin{equation}\label{eq:detanalIeer10b}
  \frac{d\zeta_{\alpha,\beta}(c_3,\nu,A_0)}{d\nu}= \frac{(1-A_0^2)e^{-\frac{\nu^2}{2}}}{w_1A_0^3}\left (\beta(1-\eta)\nu\erf\left (\frac{\nu}{\sqrt{2}}\right )A_0e^{\frac{\nu^2}{2}}+(1-\beta\eta)\nu\frac{\erfc\left (\frac{\nu}{\sqrt{2}A_0}\right )}{e^{-\frac{\nu^2}{2A_0^2}}}-(1-\beta)\sqrt{\frac{2}{\pi}}A_0\right)=0. \\
\end{equation}
From (\ref{eq:detanalIeer10b}) one then finds
\begin{equation}\label{eq:detanalIeer10c}
-\beta(1-\eta)\nu\erf\left (\frac{\nu}{\sqrt{2}}\right )A_0e^{\frac{\nu^2}{2}}+(1-\beta)\sqrt{\frac{2}{\pi}}A_0=
(1-\beta\eta)\nu\frac{\erfc\left (\frac{\nu}{\sqrt{2}A_0}\right )}{e^{-\frac{\nu^2}{2A_0^2}}}.
\end{equation}
Combining (\ref{eq:detanalIeer10a}) and (\ref{eq:detanalIeer10c}) we finally have
\begin{eqnarray}\label{eq:detanalIeer10d}
\frac{d\zeta_{\alpha,\beta}(c_3,\nu,A_0)}{dA_0}
& = & \frac{(\alpha-1)A_0^2e^{\frac{\nu^2}{2A_0^2}}\erfc(\frac{\nu}{\sqrt{2}A_0})
+((\alpha-\beta)A_0^2)A_0e^{\frac{\nu^2}{2}}\erf(\frac{\nu}{\sqrt{2}})}
{A_0^3(e^{\frac{\nu^2}{2A_0^2}}\erfc(\frac{\nu}{\sqrt{2}A_0})+A_0e^{\frac{\nu^2}{2}}\erf(\frac{\nu}{\sqrt{2}}))}\nonumber \\
& = & \frac{(\alpha-1)e^{\frac{\nu^2}{2A_0^2}}\erfc(\frac{\nu}{\sqrt{2}A_0})
+(\alpha-\beta)A_0e^{\frac{\nu^2}{2}}\erf(\frac{\nu}{\sqrt{2}})}
{A_0(e^{\frac{\nu^2}{2A_0^2}}\erfc(\frac{\nu}{\sqrt{2}A_0})+A_0e^{\frac{\nu^2}{2}}\erf(\frac{\nu}{\sqrt{2}}))}.
\end{eqnarray}
Setting the above derivative with respect to $A_0$ to zero we obtain
\begin{eqnarray}\label{eq:detanalIeer10e}
\frac{d\zeta_{\alpha,\beta}(c_3,\nu,A_0)}{dA_0}
& = & (\alpha-1)e^{\frac{\nu^2}{2A_0^2}}\erfc(\frac{\nu}{\sqrt{2}A_0})
+(\alpha-\beta)A_0e^{\frac{\nu^2}{2}}\erf(\frac{\nu}{\sqrt{2}})=0\nonumber \\
 \Longleftrightarrow \quad  e^{\frac{\nu^2}{2A_0^2}}\erfc(\frac{\nu}{\sqrt{2}A_0}) & = &
\frac{\alpha-\beta}{1-\alpha}A_0e^{\frac{\nu^2}{2}}\erf(\frac{\nu}{\sqrt{2}}).
\end{eqnarray}
A combination of (\ref{eq:detanalIeer10b}) and (\ref{eq:detanalIeer10e}) gives
\begin{eqnarray}\label{eq:detanalIeer10f}
  \frac{d\zeta_{\alpha,\beta}(c_3,\nu,A_0)}{d\nu}& = & \frac{(1-A_0^2)e^{-\frac{\nu^2}{2}}}{w_1A_0^3}\left (\beta(1-\eta)\nu\erf\left (\frac{\nu}{\sqrt{2}}\right )A_0e^{\frac{\nu^2}{2}}+(1-\beta\eta)\nu\frac{\erfc\left (\frac{\nu}{\sqrt{2}A_0}\right )}{e^{-\frac{\nu^2}{2A_0^2}}}-(1-\beta)\sqrt{\frac{2}{\pi}}A_0\right)\nonumber \\
& = & \frac{(1-A_0^2)e^{-\frac{\nu^2}{2}}}{w_1A_0^3}\left (\lp\beta(1-\eta)+\frac{(1-\beta\eta)(\alpha-\beta)}{1-\alpha}\rp\nu\erf\left (\frac{\nu}{\sqrt{2}}\right )A_0e^{\frac{\nu^2}{2}}-(1-\beta)\sqrt{\frac{2}{\pi}}A_0\right)\nonumber \\
& = & \frac{(1-A_0^2)e^{-\frac{\nu^2}{2}}}{w_1A_0^3}\left (\lp\frac{(1-\beta)(\alpha-\eta\beta)}{1-\alpha}\rp\nu\erf\left (\frac{\nu}{\sqrt{2}}\right )A_0e^{\frac{\nu^2}{2}}-(1-\beta)\sqrt{\frac{2}{\pi}}A_0\right)\nonumber \\
& = & 0.
\end{eqnarray}
Combining further(\ref{eq:detanalIeer10e}) and (\ref{eq:detanalIeer10f}) one can finally get the solution to our initial problem. In the following subsection we highlight a couple of special features of the obtained solution that are in a sharp contrast when compared to the similar ones associated with the corresponding quantities of the standard $\ell_1$'s LDP considered in \cite{Stojnicl1RegPosasymldp}.

\subsubsection{Emergence of $\beta_0$ and $\beta_1$}
\label{sec:emerbeta0beta1}

From (\ref{eq:detanalIeer10f}) we have
\begin{eqnarray}\label{eq:emerb0b1a}
\lp\frac{(1-\beta)(\alpha-\eta\beta)}{1-\alpha}\rp\nu\erf\left (\frac{\nu}{\sqrt{2}}\right )A_0e^{\frac{\nu^2}{2}}-(1-\beta)\sqrt{\frac{2}{\pi}}A_0 &= &0\nonumber \\
\Longleftrightarrow  \frac{(1-\alpha)\sqrt{\frac{2}{\pi}}e^{-\frac{\nu^2}{2}}}{(\alpha-\eta\beta)\nu\erf\left (\frac{\nu}{\sqrt{2}}\right )} & = &1.
\end{eqnarray}
To highlight the difference between the partial $\ell_1$ and the standard $\ell_1$ LDPs we will instead of $\nu$ and $A_0$ rely on two new quantities $\beta_0$ and $\beta_1$. For $\nu$ that satisfies (\ref{eq:emerb0b1a}), we first introduce $\beta_1$ by setting
\begin{equation}\label{eq:emerb0b1b}
  \nu\triangleq\sqrt{2}\erfinv\lp\frac{1-\alpha}{1-\beta_1}\rp.
\end{equation}
In other words, we let $\beta_1$ be such that
\begin{equation}\label{eq:emerb0b1c}
  \frac{(1-\beta_1)\sqrt{\frac{2}{\pi}}e^{-\frac{\lp\sqrt{2}\erfinv\lp\frac{1-\alpha}{1-\beta_1}\rp\rp^2}{2}}}{(\alpha-\eta\beta)\sqrt{2}\erfinv\lp\frac{1-\alpha}{1-\beta_1}\rp
  }  = 1.
\end{equation}
From (\ref{eq:detanalIeer10c}) and (\ref{eq:detanalIeer10e}) we, for $\nu$ defined through (\ref{eq:emerb0b1b}) and (\ref{eq:emerb0b1c}), have
\begin{eqnarray}\label{eq:emerb0b1d}
e^{\frac{\nu^2}{2A_0^2}}\erfc(\frac{\nu}{\sqrt{2}A_0}) & = &
\frac{\alpha-\beta}{1-\alpha}A_0e^{\frac{\nu^2}{2}}\erf(\frac{\nu}{\sqrt{2}})\nonumber \\
\Longleftrightarrow  \frac{(\alpha-\beta)\sqrt{\frac{2}{\pi}}A_0e^{-\frac{\nu^2}{2A_0^2}}}{(\alpha-\eta\beta)\nu\erfc\left (\frac{\nu}{\sqrt{2}A_0}\right )} & = &1.
\end{eqnarray}
For $\nu$ that satisfies (\ref{eq:emerb0b1a}) and $A_0$ that satisfies (\ref{eq:emerb0b1d}) for such a $\nu$, we introduce $\beta_0$ by setting
\begin{equation}\label{eq:emerb0b1e}
  \nu\triangleq\sqrt{2}A_0\erfinv\lp\frac{1-\alpha}{1-\beta_0}\rp.
\end{equation}
In other words, we let $\beta_0$ be such that
\begin{equation}\label{eq:emerb0b1f}
\frac{(\alpha-\beta)(1-\beta_0)\sqrt{\frac{2}{\pi}}e^{-\frac{\nu^2}{2A_0^2}}}
{(\alpha-\beta_0)(\alpha-\eta\beta)\sqrt{2}\erfinv\lp\frac{1-\alpha}{1-\beta_0}\rp} = 1.
\end{equation}
Utilizing the fundamental characterization of the partial $\ell_1$ one can further say that $\beta_1$ is such that
\begin{equation}\label{eq:emerb0b1g}
  \frac{\alpha-\eta\beta_1}{\alpha-\eta\beta}
\xi^{(p)}_{\alpha,\eta}(\beta_1)=\frac{(1-\beta_1)\sqrt{\frac{2}{\pi}}e^{-\frac{\lp\sqrt{2}\erfinv\lp\frac{1-\alpha}{1-\beta_1}\rp\rp^2}{2}}}{(\alpha-\eta\beta)\sqrt{2}\erfinv\lp\frac{1-\alpha}{1-\beta_1}\rp
  } =1,
\end{equation}
and $\beta_0$ is such that
\begin{equation}\label{eq:emerb0b1h}
  \frac{\alpha-\eta\beta_0}{\alpha-\eta\beta}
\frac{\alpha-\beta}{\alpha-\beta_0}\xi^{(p)}_{\alpha,\eta}(\beta_0)=\frac{(\alpha-\beta)(1-\beta_0)\sqrt{\frac{2}{\pi}}e^{-\frac{\nu^2}{2A_0^2}}}
{(\alpha-\beta_0)(\alpha-\eta\beta)\sqrt{2}\erfinv\lp\frac{1-\alpha}{1-\beta_0}\rp}=1.
\end{equation}
For such $\beta_0$ and $\beta_1$ one then sets
\begin{eqnarray}\label{eq:emerb0b1h1}
  \nu &  = &\sqrt{2}\erfinv \left (\frac{1-\alpha}{1-\beta_1}\right )\nonumber \\
  A_0&  = &\frac{\erfinv \left (\frac{1-\alpha}{1-\beta_1}\right )}{\erfinv \left (\frac{1-\alpha}{1-\beta_0}\right )}=\frac{\nu}{\sqrt{2}\erfinv \left (\frac{1-\alpha}{1-\beta_0}\right )}\nonumber \\
  c_3 &= &\frac{(1-A_0^2)\sqrt{\alpha}}{A_0}=\frac{\left (\erfinv \left (\frac{1-\alpha}{1-\beta_0}\right )\right )^2- \left (\erfinv\left (\frac{1-\alpha}{1-\beta_1}\right )\right )^2}{\erfinv \left (\frac{1-\alpha}{1-\beta_0}\right )\erfinv \left (\frac{1-\alpha}{1-\beta_1}\right )}\sqrt{\alpha}.
\end{eqnarray}
The choice for $\nu$, $A_0$, and $c_3$ given in (\ref{eq:emerb0b1h1}) ensures that (\ref{eq:detanalIeer3a}) is satisfied. Moreover, as it will turn out later on this choice is unique, i.e. it is the only solution of (\ref{eq:detanalIeer3a}). With a little bit of additional juggling one can then argue that this choice is not only a stationary point, but also a global optimum in (\ref{eq:detanalIeer2}). We skip such an exercise as these things will automatically follow through another set of considerations that we will present later on (in fact, quite a lot more will turn out to be true, not only will the choice for $\nu$, $A_0$, and $c_3$ given in (\ref{eq:emerb0b1h1}) turn out to be precisely the one that solves the optimization in (\ref{eq:detanalIeer2}) but also precisely the one that determines $I^{(p)}_{err}(\alpha,\beta)$). Here though, we would like to point out the key property regarding the above choice of $\nu$, $A_0$, and $c_3$. Namely, carefully looking at the above expressions one can note that, differently from what was the case when we studied the standard $\ell_1$ in \cite{Stojnicl1RegPosasymldp}, here the quantity $\beta_1$ emerges as a completely new object. Its an analogue for the standard $\ell_1$ case was $\beta_w$, i.e. the so-called weak sparsity threshold for given $\alpha$. As such, $\beta_w$ was obviously only a function of $\alpha$ and not a function of $\beta$. However, such a scenario does not repeat itself when it comes to the partial $\ell_1$. For the partial $\ell_1$, $\beta_1$ does depend on both $\alpha$ and $\beta$. This then implies that the optimal $\nu$ in (\ref{eq:detanalIeer2}) also depends on both $\alpha$ and $\beta$. It is in our view quite remarkable that in such a fully dependent scenario a closed form solution of (\ref{eq:detanalIeer2}) could still be obtained.

In the following subsection we compute the value of $\zeta_{\alpha,\beta}(c_3,\nu,A_0)$ that one gets if $c_3$, $\nu$, and $A_0$ are as in (\ref{eq:emerb0b1h1}). As stated above, and as will turn out later on, this  value of $\zeta_{\alpha,\beta}(c_3,\nu,A_0)$ will be precisely the $I^{(p)}_{err}(\alpha,\beta)$ from (\ref{eq:ldpasymp1}) and Theorem \ref{thm:ldp2}.

\subsubsection{Computing $\zeta_{\alpha,\beta}(c_3,\nu,A_0)$}
\label{sec:computingzeta}

In order to compute $\zeta_{\alpha,\beta}(c_3,\nu,A_0)$ we will first compute $\widehat{\gamma}$, $I_{sph}$, $w_1$, $w_2$, and $w_3$. In fact,  $\widehat{\gamma}$ and $I_{sph}$ have already been computed in \cite{Stojnicl1RegPosasymldp} and we for them immediately have
\begin{eqnarray}\label{eq:computingzeta1}
  \widehat{\gamma}&=&-\frac{A_0\sqrt{\alpha}}{2},\nonumber \\
I_{sph} &=&-\frac{(1-A_0^2)\alpha}{2}+\alpha\log(A_0).
\end{eqnarray}
For $w_1$ we have
\begin{equation}\label{eq:computingzeta3}
w_1 =
  \frac{e^{\frac{(1-A_0^2)\nu^2}{2A_0^2}}}{A_0}\erfc\left (\frac{\nu}{\sqrt{2}A_0}\right )+\erf\left (\frac{\nu}{\sqrt{2}}\right )
  =\frac{\alpha-\beta}{(\alpha-\eta\beta)\nu}\sqrt{\frac{2}{\pi}}e^{-\frac{\nu^2}{2}}+\erf\left (\frac{\nu}{\sqrt{2}}\right )
  =\frac{\alpha-\beta}{1-\beta_1}+\frac{1-\alpha}{1-\beta_1}=\frac{1-\beta}{1-\beta_1},
\end{equation}
where the second equality follows from (\ref{eq:emerb0b1d}) while the third equality follows from (\ref{eq:emerb0b1b}) and (\ref{eq:emerb0b1c}). One then also easily has
\begin{equation}\label{eq:computingzeta4}
  w_2 =
  \frac{e^{\frac{(1-A_0^2)\nu^2}{2A_0^2}}}{A_0}=\frac{\alpha-\beta}{1-\beta_1}\frac{1}{\erfc\left (\frac{\nu}{\sqrt{2}A_0}\right )}
  =\frac{\alpha-\beta}{1-\beta_1}\frac{1}{1-\erf\left (\frac{\nu}{\sqrt{2}A_0}\right )}
  =\frac{\alpha-\beta}{1-\beta_1}\frac{1}{1-\frac{1-\alpha}{1-\beta_0}}
  =\frac{\alpha-\beta}{\alpha-\beta_0}\frac{1-\beta_0}{1-\beta_1},
\end{equation}
and obviously
\begin{equation}\label{eq:computingzeta4a}
  w_3 =\frac{1}{A_0}.
\end{equation}
A combination of (\ref{eq:computingzeta1}), (\ref{eq:computingzeta3}), (\ref{eq:computingzeta4}), and (\ref{eq:computingzeta4a}) then gives
\begin{eqnarray}\label{eq:computingzeta4b}
\zeta_{\alpha,\beta}(c_3,\nu,A_0)&=&\left (-\frac{c_3^2}{2}+I_{sph}+(1-\beta)\log{w_1}+\beta(1-\eta)\log{w_2}
+\beta\eta\log{w_3}+\frac{c_3^2}{2(1-A_0^2)}\right )\nonumber \\
&=&\left (\frac{c_3^2A_0^2}{2(1-A_0^2)}-\frac{(1-A_0^2)\alpha}{2}+\alpha\log(A_0)+(1-\beta)\log{w_1}+\beta(1-\eta)\log{w_2}
+\beta\eta\log{w_3}\right )\nonumber \\
&=& (\alpha-\eta\beta)\log(A_0)+(1-\beta)\log{w_1}+\beta(1-\eta)\log{w_2}
\nonumber \\
&=& (\alpha-\eta\beta)\log\left (\frac{\erfinv\left (\frac{1-\alpha}{1-\beta_1}\right )}{\erfinv\left (\frac{1-\alpha}{1-\beta_0}\right )}\right )+(1-\beta)\log{\left (\frac{1-\beta}{1-\beta_1}\right )}+\beta(1-\eta)\log{\left (\frac{\alpha-\beta}{\alpha-\beta_0}\frac{1-\beta_0}{1-\beta_1}\right )}.\nonumber \\
\end{eqnarray}
We summarize the above results in the following theorem.
\begin{theorem}
Assume the setup of Theorem \ref{thm:ldp2} and assume that a pair $(\alpha,\beta)$  is given. Let $\alpha>\alpha_w$ where $\alpha_w$ is such that $\psi^{(p)}_{\beta,\eta}(\alpha_w)=\xi^{(p)}_{\alpha_w}(\beta)=1$. Also let $\beta_1$ and $\beta_0$ satisfy the following \textbf{fundamental characterizations of the partial $\ell_1$'s LDP:}

\begin{equation}
  \frac{\alpha-\eta\beta_1}{\alpha-\eta\beta}
\xi^{(p)}_{\alpha,\eta}(\beta_1)=\frac{(1-\beta_1)\sqrt{\frac{2}{\pi}}e^{-\lp\erfinv\lp\frac{1-\alpha}{1-\beta_1}\rp\rp^2}}{(\alpha-\eta\beta)\sqrt{2}\erfinv\lp\frac{1-\alpha}{1-\beta_1}\rp
  } =1,
\label{eq:thmldp3l1PT}
\end{equation}

\noindent and

\begin{equation}
  \frac{\alpha-\eta\beta_0}{\alpha-\eta\beta}
\frac{\alpha-\beta}{\alpha-\beta_0}\xi^{(p)}_{\alpha,\eta}(\beta_0)=\frac{(\alpha-\beta)(1-\beta_0)\sqrt{\frac{2}{\pi}}e^{-\lp\erfinv\lp\frac{1-\alpha}{1-\beta_0}\rp\rp^2}}
{(\alpha-\beta_0)(\alpha-\eta\beta)\sqrt{2}\erfinv\lp\frac{1-\alpha}{1-\beta_0}\rp}=1.
\label{eq:thmldp3l1LDP}
\end{equation}

\noindent Then
\begin{eqnarray}
I^{(p)}_{err}(\alpha,\beta)\triangleq\lim_{n\rightarrow\infty}\frac{\log{P_{err}}}{n}
& \leq &
(\alpha-\eta\beta)\log\left (\frac{\erfinv\left (\frac{1-\alpha}{1-\beta_1}\right )}{\erfinv\left (\frac{1-\alpha}{1-\beta_0}\right )}\right )+(1-\beta)\log\left (\frac{1-\beta}{1-\beta_1}\right ) \nonumber \\
& & +\beta(1-\eta)\log\left (\frac{(\alpha-\beta)(1-\beta_0)}{(\alpha-\beta_0)(1-\beta_1)}\right )\triangleq I^{(p)}_{ldp}(\alpha,\beta).
\label{eq:ldpthm3Ierrub1}
\end{eqnarray}
Moreover, the following choice for $\nu$, $c_3$, and $\gamma$ in the optimization problem in Theorem \ref{thm:ldp2} achieves the right hand side of (\ref{eq:ldpthm3Ierrub1})
\begin{eqnarray}
  \nu &  = & \sqrt{2}\erfinv \left (\frac{1-\alpha}{1-\beta_1}\right )\nonumber \\
    A_0&=&\frac{\erfinv \left (\frac{1-\alpha}{1-\beta_1}\right )}{\erfinv \left (\frac{1-\alpha}{1-\beta_0}\right )}=\frac{\nu}{\sqrt{2}\erfinv \left (\frac{1-\alpha}{1-\beta_0}\right )}\nonumber \\
      c_3& = &\frac{(1-A_0^2)\sqrt{\alpha}}{A_0}=\frac{\left (\erfinv \left (\frac{1-\alpha}{1-\beta_0}\right )\right )^2- \left (\erfinv\left (\frac{1-\alpha}{1-\beta_1}\right )\right )^2}{\erfinv \left (\frac{1-\alpha}{1-\beta_0}\right )\erfinv \left (\frac{1-\alpha}{1-\beta_1}\right )}\sqrt{\alpha}\nonumber \\
      \gamma&=&\frac{c_3}{2(1-A_0^2)}=\frac{\sqrt{\alpha}}{2A_0}=\frac{\sqrt{\alpha}\erfinv \left (\frac{1-\alpha}{1-\beta_0}\right )}{2\erfinv \left (\frac{1-\alpha}{1-\beta_1}\right )}.
\label{eq:ldpthm3perrub2}
\end{eqnarray}\label{thm:ldp3}
\end{theorem}
\begin{proof} Follows from the above discussion.
\end{proof}
The results obtained in the above theorem are then enough to fully characterize the upper tail of the partial $\ell_1$ LDP. In fact, they remain correct in the lower tail regime and therefore are then enough to fully characterize the entire partial $\ell_1$ LDP. We will formalize the lower tail statements in the following section and after that we will then finally present the quantitative results that one can obtain relying on what was established through the above theorem. A couple of properties of functions $\frac{\alpha-\eta\beta_1}{\alpha-\eta\beta}
\xi^{(p)}_{\alpha,\eta}(\beta_1)$ and $\frac{\alpha-\eta\beta_0}{\alpha-\eta\beta}
\frac{\alpha-\beta}{\alpha-\beta_0}\xi^{(p)}_{\alpha,\eta}(\beta_0)$ will also be discussed at that time.

\subsection{Lower tail}
\label{sec:lowertail}

%

In this section we will quickly formalize the above statements about the lower tail type of large deviations. As there is not much difference between the arguments that we will use here and those we used in \cite{Stojnicl1RegPosasymldp} we will skip repeating the details and instead just state the final results. For the lower tail we introduce a quantity complementary to the $P_{err}$ considered in the previous sections and denote it by $P_{cor}$, i.e. we set
\begin{equation}
P_{cor} \triangleq P(\min_{A\w=0,\|\w\|_2\leq 1}\sum_{i=n-k+1}^{n-\eta k} \w_i+\sum_{i=1}^{n-k}|\w_{i}|\geq 0).
\label{eq:ldpproblower}
\end{equation}
If $P_{err}$ is the probability that (\ref{eq:l1imp}) fails to produce the solution of (\ref{eq:l0}) then obviously $P_{cor}=1-P_{err}$ is the probability that (\ref{eq:l1}) does produce the solution of (\ref{eq:l0}). As in \cite{Stojnicl1RegPosasymldp} (and ultimately \cite{StojnicBlockasymldpfinn15}) we then have
\begin{equation}
P_{cor}
\leq
P(\|\g\|_2 -w(\h,S_w)-t_1\geq 0)/P(g\geq t_1)
\leq  \min_{t_1}\min_{c_3\geq 0}
Ee^{c_3\|\g\|_2}Ee^{-c_3w(\h,S_w)}e^{-c_3t_1}/P(g\geq t_1).
\label{eq:ldpprob3lower}
\end{equation}
One then views the above bound (which is valid for any integers $m$, $k$, and $n$ for which the results make sense) in the LDP sense and defines $I^{(p)}_{cor}(\alpha,\beta)$ (essentially a lower tail analogue to $I^{(p)}_{err}(\alpha,\beta)$), in the following way
\begin{equation}\label{eq:ldpasymp1lower}
  I^{(p)}_{cor}(\alpha,\beta)\triangleq\lim_{n\rightarrow\infty}\frac{\log{P_{cor}}}{n}.
\end{equation}
The following theorem is then the lower tail analogue to Theorem \ref{thm:ldp2}.
\begin{theorem}
Assume the setup of Theorem \ref{thm:ldp2}. Then
\begin{eqnarray}
I^{(p)}_{cor}(\alpha,\beta)& \triangleq & \lim_{n\rightarrow\infty}\frac{\log{P_{cor}}}{n}\nonumber \\
& \leq  &\min_{c_3\geq 0}\left (-\frac{c_3^2}{2}+I_{sph}^++\max_{\nu\geq 0,\gamma^{(s)}\geq 0} ((1-\beta\log{w_1}+\beta(1-\eta)\log{w_2}+\beta\eta\log{w_3}-c_3\gamma)\right )\triangleq I_{cor,l}^{(p,ub)}(\alpha,\beta), \nonumber \\
\label{eq:ldpthm2Icorub1}
\end{eqnarray}
where
\begin{eqnarray}
I_{sph}^+ &=& \widehat{\gamma_+}c_3-\frac{\alpha d}{2}\log\left (1-\frac{c_3}{2\widehat{\gamma_+}}\right )\nonumber \\
  \widehat{\gamma_+} &=& \frac{2c_3+\sqrt{4c_3^2+16\alpha }}{8}\nonumber \\
w_1 &=& \frac{1}{\sqrt{2\pi}}\int_{\bar{h}}e^{-\bar{h}^2/2}e^{-c_3\max(|\bar{h}|-\nu,0)^2/4/\gamma}d\bar{h}
  =\frac{e^{\frac{-c_3\nu^2/4/\gamma}{1+c_3/2/\gamma}}}{\sqrt{1+c_3/2/\gamma}}\erfc\left (\frac{\nu}{\sqrt{2}\sqrt{1+c_3/2/\gamma}}\right )+\erf\left (\frac{\nu}{\sqrt{2}}\right )\nonumber \\
  w_2 &=& \frac{1}{\sqrt{2\pi}}\int_{\bar{h}}e^{-\bar{h}^2/2}e^{-c_3(\bar{h}+\nu)^2/4/\gamma}d\bar{h}
  =\frac{e^{\frac{-c_3\nu^2/4/\gamma}{1+c_3/2/\gamma}}}{\sqrt{1+c_3/2/\gamma}}\nonumber \\
    w_3 &=& \frac{1}{\sqrt{2\pi}}\int_{\bar{h}}e^{-\bar{h}^2/2}e^{-c_3\bar{h}^2/4/\gamma}d\bar{h}
  =\frac{1}{\sqrt{1+c_3/2/\gamma}}.\label{eq:ldpthm2perrub2lower}
\end{eqnarray}\label{thm:ldp2lower}
\end{theorem}
\begin{proof} Follows in exactly the same way as the result for the lower tail of the standard $\ell_1$ LDP in \cite{Stojnicl1RegPosasymldp}.
\end{proof}
To analyze and solve the above optimization one can then repeat the procedure from the previous section. Instead one can just quickly observe that
the change $c_3\rightarrow -c_3$ gives as in Section \ref{sec:analysisIerr}
\begin{equation}\label{eq:detanalIcor1}
  A_{0}\triangleq\sqrt{1-\frac{c_3}{2\gamma}},
\end{equation}
and
\begin{equation}\label{eq:detanalIcor2}
I_{cor,l}^{(p,ub)}(\alpha,\beta)\triangleq \min_{c_3\leq 0}\max_{\nu\geq 0,A_0\leq 1}\zeta_{\alpha,\beta}(c_3,\nu,A_0)
\end{equation}
where
\begin{eqnarray}
\zeta_{\alpha,\beta}(c_3,\nu,A_0)&=&\left (-\frac{c_3^2}{2}+I_{sph}^++(1-\beta)\log{w_1}+\beta\log{w_2}+\frac{c_3^2}{2(1-A_0^2)}\right )\nonumber \\
I_{sph}^+ &=& -\widehat{\gamma^+}c_3-\frac{\alpha }{2}\log\left (1+\frac{c_3}{2\widehat{\gamma^+}}\right )\nonumber \\
  \widehat{\gamma^+} &=& \frac{-c_3+\sqrt{c_3^2+4\alpha}}{4}=-\widehat{\gamma}\nonumber \\
w_1 &=&
  \frac{e^{\frac{(1-A_0^2)\nu^2}{2A_0^2}}}{A_0}\erfc\left (\frac{\nu}{\sqrt{2}A_0}\right )+\erf\left (\frac{\nu}{\sqrt{2}}\right )\nonumber \\
  w_2 &=&
  \frac{e^{\frac{(1-A_0^2)\nu^2}{2A_0^2}}}{A_0}\nonumber \\
    w_3 &=&
  \frac{1}{A_0}.\label{eq:detanalIcor3}
\end{eqnarray}
Moreover, $\zeta_{\alpha,\beta}(c_3,\nu,A_0)$ defined in (\ref{eq:detanalIcor3}) is exactly the same as the corresponding one in (\ref{eq:detanalIeer3}) which means that one can proceed with the computation of all the derivatives as earlier and the values we have chosen for $c_3$, $\nu$, $\gamma$, and $A_0$ in the upper tail regime will have the same form. The following theorem summarizes the final results (this is of course nothing but a lower tail analogue to Theorem \ref{thm:ldp3}).
\begin{theorem}
Assume the setup of Theorem \ref{thm:ldp3} and assume that a pair $(\alpha,\beta)$  is given. Differently from Theorem \ref{thm:ldp3}, let $\alpha<\alpha_w$ where $\alpha_w$ is such that $\psi^{(p)}_{\beta,\eta}(\alpha_w)=\xi^{(p)}_{\alpha_w}(\beta)=1$. Also let $\beta_1$ and $\beta_0$ satisfy the \textbf{fundamental \emph{partial} $\ell_1$'s LDP characterizations}, respectively as in Theorem \ref{thm:ldp3}. Then choosing $\nu$, $c_3$, and $\gamma$ in the optimization problem in (\ref{eq:ldpthm2Icorub1}) as $\nu$, $-c_3$, and $\gamma$  from Theorem \ref{thm:ldp3} (or equivalently, choosing $\nu$, $c_3$, and $A_0$ in the optimization problem in (\ref{eq:detanalIcor2}) as $\nu$, $c_3$, and $A_0$  from Theorem \ref{thm:ldp3}) gives
\begin{eqnarray}
\zeta_{\alpha,\beta}(c_3,\nu,A_0)& = &
(\alpha-\eta\beta)\log\left (\frac{\erfinv\left (\frac{1-\alpha}{1-\beta_1}\right )}{\erfinv\left (\frac{1-\alpha}{1-\beta_0}\right )}\right )\nonumber \\
& & +(1-\beta)\log\left (\frac{1-\beta}{1-\beta_1}\right ) +\beta(1-\eta)\log\left (\frac{(\alpha-\beta)(1-\beta_0)}{(\alpha-\beta_0)(1-\beta_1)}\right )=I^{(p)}_{ldp}(\alpha,\beta).
\label{eq:ldpthm3Icorub1}
\end{eqnarray}
\label{thm:ldp3lower}
\end{theorem}
\begin{proof} Follows from the considerations leading to Theorem \ref{thm:ldp3}.
\end{proof}
Similarly to what happens in the standard $\ell_1$ LDP lower tail regime, here one also observes that for $\alpha<\alpha_w$, $\beta_1<\beta_0$ which means $A_0>1$ and finally $c_3<0$. Of course, in the upper tail regime (i.e. in Theorem \ref{thm:ldp3}) things are reversed and $c_3>0$.

\subsection{High-dimensional geometry}
\label{sec:hdg}

While the previous section relies on a purely probabilistic approach to characterize the partial $\ell_1$ LDP, in this section we take a different path and view the partial $\ell_1$ and its LDPs through the prism of the high-dimensional integral geometry. We start by making similar assumption as earlier and splitting the discussion into two parts/regimes, the upper tail and the lower tail. To be more concrete, we assume that we are given a pair $(\alpha,\beta)$ and that $\beta_1$ and $\beta_0$ are given by the partial $\ell_1$ fundamental LDP characterizations defined earlier. Also, we will first consider the upper tail regimes, i.e. $\alpha>\alpha_w$ (where $\alpha_w$ is such that $\psi^{(p)}_{\beta,\eta}(\alpha_w)=\xi^{(p)}_{\alpha_w}(\beta)=1$) (the lower tail analogues will then immediately follow once the upper ones are established). To start things off we rely on the following observations from \cite{StojnicEquiv10}
\begin{equation}\label{eq:hdg1}
  \Psi^{(p)}_{net}(\alpha,\beta)=I^{(p)}_{err}(\alpha,\beta)\triangleq\lim_{n\rightarrow\infty}\frac{\log{P_{err}}}{n}=\psicom+\psiint-\psiext,
\end{equation}
where
\begin{eqnarray}
\psicom & = & (\alpha-\beta) \log(2)-(\alpha-\beta)\log\left (\frac{\alpha-\beta}{1-\beta}\right )-(1-\alpha)\log\left (\frac{1-\alpha}{1-\beta}\right )\nonumber \\
\psiint & = & \min_{y\geq 0} ((\alpha-\eta\beta) y^2 +(\alpha-\beta)\log(\erfc(y)))- (\alpha-\beta) \log(2)\nonumber\\
\psiext & = & \min_{y\geq 0} ((\alpha-\eta\beta) y^2 -(1-\alpha)\log(\erf(y))). \label{eq:hdg2}
\end{eqnarray}
Now, we observe that the above expressions for $\psiint$ and $\psiext$ are very similar to the corresponding ones that we obtained when considering the standard $\ell_1$ LDP through the high-dimensional geometry approach in \cite{Stojnicl1RegPosasymldp}. Consequently, we can closely follow the analysis from \cite{Stojnicl1RegPosasymldp}. However, instead of redoing all the steps from \cite{Stojnicl1RegPosasymldp} we will try to speed things up and focus only on the key differences. Following \cite{Stojnicl1RegPosasymldp}, let $y_{int}$ and $y_{ext}$ be the solutions of the above optimizations (clearly, $y_{int}$ is the solution of the optimization associated with $\psiint$ and $y_{ext}$ is the solution of the optimization associated with $\psiext$). To determine $y_{int}$ we start by taking the following derivative
\begin{equation}
\frac{d((\alpha-\eta\beta) y^2 +(\alpha-\beta)\log(\erfc(y)))}{dy}  = 2(\alpha-\eta\beta) y-\frac{\alpha-\beta}{1-\erf(y)}\frac{2e^{-y^2}}{\sqrt{\pi}}. \label{eq:hdg3}
\end{equation}
Choosing
\begin{equation}\label{eq:hdg4}
y_{int}  =  \erfinv\left (\frac{1-\alpha}{1-\beta_0}\right ),
\end{equation}
the derivative in (\ref{eq:hdg3}) becomes
\begin{eqnarray}
\frac{d((\alpha-\eta\beta) y^2 +(\alpha-\beta)\log(\erfc(y)))}{dy} | _{y=y_{int}}  & = &   2(\alpha-\eta\beta) y_{int}-\frac{\alpha-\beta}{1-\erf(y_{int})}\frac{2e^{-y_{int}^2}}{\sqrt{\pi}}\nonumber \\
& = &2\alpha \erfinv\left (\frac{1-\alpha}{1-\beta_0}\right )\left (1-\frac{\alpha-\eta\beta_0}{\alpha-\eta\beta}\frac{\alpha-\beta}{\alpha-\beta_0}\xi^{(p)}_{\alpha,\eta}(\beta_0)\right )\nonumber \\
& = &0, \label{eq:hdg5}
\end{eqnarray}
where the last equality follows by the fundamental characterization of the partial $\ell_1$'s LDP. For the second derivative we have a similar consideration
\begin{eqnarray}
\frac{d^2((\alpha-\eta\beta) y^2 +(\alpha-\beta)\log(\erfc(y)))}{dy^2}  & = &   2(\alpha-\eta\beta) -\frac{\alpha-\beta}{\erfc(y)^2}\left (\frac{2e^{-y^2}}{\sqrt{\pi}}\right )^2+\frac{\alpha-\beta}{\erfc(y)}\left (\frac{4ye^{-y^2}}{\sqrt{\pi}}\right )\nonumber \\
& > &   2(\alpha-\beta) -\frac{\alpha-\beta}{\erfc(y)^2}\left (\frac{2e^{-y^2}}{\sqrt{\pi}}\right )^2+\frac{\alpha-\beta}{\erfc(y)}\left (\frac{4ye^{-y^2}}{\sqrt{\pi}}\right )\nonumber \\
& = &   \frac{2(\alpha-\beta) e^{-2y^2}}{\pi\erfc(y)^2}\left (\pi\erfc(y)^2e^{2y^2} -2+2 y \sqrt{\pi}\erfc(y)e^{y^2}\right )\nonumber \\
 & > &      \frac{2(\alpha-\beta) e^{-2y^2}}{\pi\erfc(y)^2}\left (\frac{4-2(y+\sqrt{y^2+2})^2+4y(y+\sqrt{y^2+2})}{(y+\sqrt{y^2+2})^2}\right )\nonumber \\
& = & 0, \label{eq:hdg11}
\end{eqnarray}
where the last inequality was proven in \cite{Stojnicl1RegPosasymldp}. From (\ref{eq:hdg11}) one concludes that $(\alpha y^2 +(\alpha-\beta)\log(\erfc(y)))$ is convex and that $y_{int}$ is its a global optimum (minimum). Combining (\ref{eq:hdg2}) and (\ref{eq:hdg4}) further and relying on the partial $\ell_1$'s fundamental LDP then finally gives
\begin{eqnarray}\label{eq:hdg12}
\psiint & = &  (\alpha-\eta\beta) y_{int}^2 +(\alpha-\beta)\log(\erfc(y_{int})- (\alpha-\beta) \log(2) \nonumber \\
& = & (\alpha-\eta\beta) \log\left (e^{\left (\erfinv\left (\frac{1-\alpha}{1-\beta_0}\right )
\right )^2}\right ) +(\alpha-\beta)\log\left (\frac{\alpha-\beta_0}{1-\beta_0}\right )- (\alpha-\beta) \log(2)\nonumber \\
& = & (\alpha-\eta\beta) \log\left (\frac{\alpha-\beta}{\alpha-\beta_0}\frac{1-\beta_0}{(\alpha-\eta\beta)\sqrt{2}\erfinv\left (\frac{1-\alpha}{1-\beta_0}\right )}\right ) +(\alpha-\beta)\log\left (\frac{\alpha-\beta_0}{1-\beta_0}\right )\nonumber \\
&&- (\alpha-\beta) \log(2)+(\alpha-\eta\beta)\log\lp\sqrt{\frac{2}{\pi}}\rp\nonumber \\
& = & -(\alpha-\eta\beta) \log\left (\sqrt{2}\erfinv\left (\frac{1-\alpha}{1-\beta_0}\right )\right )+(\alpha-\eta\beta)\log\left (\frac{\alpha-\beta}{\alpha-\eta\beta}\right )
-\beta(1-\eta)\log\left (\frac{\alpha-\beta_0}{1-\beta_0}\right )\nonumber \\
&&- (\alpha-\beta) \log(2)+(\alpha-\eta\beta)\log\lp\sqrt{\frac{2}{\pi}}\rp.\nonumber \\
\end{eqnarray}
For $\psiext$ we have a similar set of considerations. Namely, to determine $y_{ext}$
we start by taking the following derivative
\begin{equation}
\frac{d((\alpha-\eta\beta) y^2 -(1-\alpha)\log(\erf(y)))}{dy} = 2(\alpha-\eta\beta) y-\frac{1-\alpha}{\erf(y)}\frac{2e^{-y^2}}{\sqrt{\pi}}. \label{eq:hdg3e}
\end{equation}
Choosing
\begin{equation}\label{eq:hdg4e}
y_{ext}  =  \erfinv\left (\frac{1-\alpha}{1-\beta_1}\right ),
\end{equation}
the derivative in (\ref{eq:hdg3e}) becomes
\begin{eqnarray}
\frac{d((\alpha-\eta\beta) y^2 -(1-\alpha)\log(\erf(y)))}{dy} | _{y=y_{ext}}  & = &   2(\alpha-\eta\beta) y_{ext}-\frac{1-\alpha}{\erf(y_{ext})}\frac{2e^{-y_{ext}^2}}{\sqrt{\pi}}\nonumber \\
& = &2(\alpha-\eta\beta)\erfinv\left (\frac{1-\alpha}{1-\beta_1}\right )\left (1-\frac{\alpha-\eta\beta_1}{\alpha-\eta\beta}\xi^{(p)}_{\alpha,\eta}(\beta_1)\right )\nonumber \\
& = &0, \label{eq:hdg5e}
\end{eqnarray}
where the last equality follows by the fundamental characterization of the partial $\ell_1$'s LDP. For the second derivative we quickly have
\begin{eqnarray}
\frac{d^2((\alpha-\eta\beta)y^2 +(\alpha-\beta)\log(\erfc(y)))}{dy^2}   =    2(\alpha-\eta\beta) +\frac{1-\alpha}{\erf(y)^2}\left (\frac{2e^{-y^2}}{\sqrt{\pi}}\right )^2+\frac{1-\alpha}{\erfc(y)}\left (\frac{4ye^{-y^2}}{\sqrt{\pi}}\right )
 > 0. \label{eq:hdg11e}
\end{eqnarray}
Moreover,
\begin{eqnarray}\label{eq:hdg13}
\psiext  & = &   \min_{y\geq 0} ((\alpha-\eta\beta) y^2 -(1-\alpha)\log(\erf(y)))\nonumber \\
& = &(\alpha-\eta\beta)\log\left (e^{\left (\erfinv\left (\frac{1-\alpha}{1-\beta_1}\right )\right )^2}\right )-(1-\alpha)\log\left (\frac{1-\alpha}{1-\beta_1}\right )\nonumber \\
& = &(\alpha-\eta\beta) \log\left (\frac{1-\beta_1}{(\alpha-\eta\beta)\sqrt{2}\erfinv\left (\frac{1-\alpha}{1-\beta_1}\right )}\right )-(1-\alpha)\log\left (\frac{1-\alpha}{1-\beta_1}\right )+(\alpha-\eta\beta)\log\lp\sqrt{\frac{2}{\pi}}\rp\nonumber \\
& = &-(\alpha-\eta\beta) \log\left (\sqrt{2}\erfinv\left (\frac{1-\alpha}{1-\beta_1}\right )\right )+(\alpha-\eta\beta) \log\left (\frac{1-\beta_1}{(\alpha-\eta\beta)}\right )\nonumber \\
&&-(1-\alpha)\log\left (\frac{1-\alpha}{1-\beta_1}\right )+(\alpha-\eta\beta)\log\lp\sqrt{\frac{2}{\pi}}\rp\nonumber \\
& = &-(\alpha-\eta\beta) \log\left (\sqrt{2}\erfinv\left (\frac{1-\alpha}{1-\beta_1}\right )\right )-(\alpha-\eta\beta) \log((\alpha-\eta\beta))\nonumber \\
&&-(1-\alpha)\log (1-\alpha)+(1-\eta\beta)\log(1-\beta_1)
+(\alpha-\eta\beta)\log\lp\sqrt{\frac{2}{\pi}}\rp.\nonumber \\
\end{eqnarray}
Finally one can combine (\ref{eq:hdg1}), (\ref{eq:hdg2}), (\ref{eq:hdg12}), and (\ref{eq:hdg13}) to obtain
\begin{eqnarray}
  \Psi^{(p)}_{net}(\alpha,\beta)&=&I^{(p)}_{err}(\alpha,\beta)=\psicom+\psiint-\psiext\nonumber \\
& = & (\alpha-\beta) \log(2)-(\alpha-\beta)\log\left (\frac{\alpha-\beta}{1-\beta}\right )-(1-\alpha)\log\left (\frac{1-\alpha}{1-\beta}\right )\nonumber \\
&&-(\alpha-\eta\beta) \log\left (\sqrt{2}\erfinv\left (\frac{1-\alpha}{1-\beta_0}\right )\right )+(\alpha-\eta\beta)\log\left (\frac{\alpha-\beta}{\alpha-\eta\beta}\right )
-\beta(1-\eta)\log\left (\frac{\alpha-\beta_0}{1-\beta_0}\right )\nonumber \\
&& - (\alpha-\beta) \log(2)+(\alpha-\eta\beta) \log\left (\sqrt{2}\erfinv\left (\frac{1-\alpha}{1-\beta_1}\right )\right )+(\alpha-\eta\beta) \log(\alpha-\eta\beta)\nonumber \\
&&+(1-\alpha)\log (1-\alpha)-(1-\eta\beta)\log(1-\beta_1)\nonumber \\
& = & (\alpha-\eta\beta) \log\left (\frac{\erfinv\left (\frac{1-\alpha}{1-\beta_1}\right )}{\erfinv\left (\frac{1-\alpha}{1-\beta_0}\right )}\right )+\beta(1-\eta)\log(\alpha-\beta)+(1-\beta)\log(1-\beta)\nonumber \\
&&-\beta(1-\eta)\log\left (\frac{\alpha-\beta_0}{1-\beta_0}\right )-(1-\beta\eta)\log(1-\beta_1)\nonumber \\
& = & (\alpha-\eta\beta) \log\left (\frac{\erfinv\left (\frac{1-\alpha}{1-\beta_1}\right )}{\erfinv\left (\frac{1-\alpha}{1-\beta_0}\right )}\right )+\beta(1-\eta)\log(\alpha-\beta)+(1-\beta)\log(1-\beta)\nonumber \\
&&-\beta(1-\eta)\log\left (\frac{\alpha-\beta_0}{1-\beta_0}\right )-(1-\beta)\log(1-\beta_1)-(\beta-\beta\eta)\log(1-\beta_1)\nonumber \\
& = & (\alpha-\eta\beta) \log\left (\frac{\erfinv\left (\frac{1-\alpha}{1-\beta_1}\right )}{\erfinv\left (\frac{1-\alpha}{1-\beta_0}\right )}\right )+
(1-\beta)\log\left (\frac{1-\beta}{1-\beta_1}\right )+\beta(1-\eta)\log\left (\frac{\alpha-\beta}{\alpha-\beta_0}\frac{1-\beta_0}{1-\beta_1}\right )\nonumber \\
&= &I^{(p)}_{ldp}(\alpha,\beta).\nonumber \\
\label{eq:hdg14}
\end{eqnarray}
A combination of (\ref{eq:ldpthm3Ierrub1}), (\ref{eq:hdg1}), and (\ref{eq:hdg14}) then gives
\begin{equation}\label{eq:hdg15}
  I^{(p)}_{err}(\alpha,\beta)=I^{(p)}_{ldp}(\alpha,\beta),
\end{equation}
and ensures that the choice for $\nu$, $A_0$, $c_3$, and $\gamma$ made in (\ref{eq:ldpthm3perrub2}) is indeed optimal. Moreover, in the lower tail regime ($\alpha<\alpha_w$, where $\alpha_w$ is such that $\psi^{(p)}_{\beta,\eta}(\alpha_w)=\xi^{(p)}_{\alpha_w}(\beta)=1$) considerations from \cite{StojnicEquiv10} ensure that one also has
\begin{equation}\label{eq:hdg1a}
  \Psi^{(p)}_{net}(\alpha,\beta)=I^{(p)}_{cor}(\alpha,\beta)\triangleq\lim_{n\rightarrow\infty}\frac{\log{P_{cor}}}{n}=\psicom+\psiint-\psiext,
\end{equation}
where $\psicom$, $\psiint$, and $\psiext$ are as in (\ref{eq:hdg2}). All of the above considerations are then enough to fully characterize the partial $\ell_1$'s LDP. We summarize the characterization in the following theorem.
\begin{theorem}[Partial $\ell_1$'s LDP]
Assume the setup of Theorem \ref{thm:thmweakthr} and assume that a pair $(\alpha,\beta)$ is given. Let $P_{err}$ be the probability that the solutions of (\ref{eq:l0}) and (\ref{eq:l1}) coincide and let $P_{cor}$ be the probability that the solutions of (\ref{eq:l0}) and (\ref{eq:l1}) do \emph{not} coincide. Let $\alpha_w$ and $\beta_w$ satisfy the \textbf{\emph{partial} $\ell_1$'s fundamental PT} characterizations in the following way
\begin{equation}
\psi^{(p)}_{\beta,\eta}(\alpha_w)\triangleq
(1-\beta)\frac{\sqrt{\frac{2}{\pi}}e^{-\lp\erfinv\lp\frac{1-\alpha_w}{1-\beta}\rp\rp^2}}{(\alpha_w-\eta\beta)\sqrt{2}\erfinv \lp\frac{1-\alpha_w}{1-\beta}\rp}=1\quad \mbox{and} \quad
\xi^{(p)}_{\alpha,\eta}(\beta_w)\triangleq
(1-\beta_w)\frac{\sqrt{\frac{2}{\pi}}e^{-\lp\erfinv\lp\frac{1-\alpha}{1-\beta_w}\rp\rp^2}}{(\alpha-\eta\beta_w)\sqrt{2}\erfinv \lp\frac{1-\alpha}{1-\beta_w}\rp}=1.\label{eq:thmfinalldpl11}
\end{equation}
Further let $\beta_1$ and $\beta_0$ satisfy the following \textbf{partial $\ell_1$'s fundamental LDP} characterizations
\begin{equation}
  \frac{\alpha-\eta\beta_1}{\alpha-\eta\beta}
\xi^{(p)}_{\alpha,\eta}(\beta_1)=\frac{(1-\beta_1)\sqrt{\frac{2}{\pi}}e^{-\lp\erfinv\lp\frac{1-\alpha}{1-\beta_1}\rp\rp^2}}{(\alpha-\eta\beta)\sqrt{2}\erfinv\lp\frac{1-\alpha}{1-\beta_1}\rp
  } =1,
\label{eq:thmfinalldpl12a}
\end{equation}

\noindent and

\begin{equation}
\frac{\alpha-\eta\beta_0}{\alpha-\eta\beta}
\frac{\alpha-\beta}{\alpha-\beta_0}\xi^{(p)}_{\alpha,\eta}(\beta_0)=\frac{(\alpha-\beta)(1-\beta_0)\sqrt{\frac{2}{\pi}}e^{-\lp\erfinv\lp\frac{1-\alpha}{1-\beta_0}\rp\rp^2}}
{(\alpha-\beta_0)(\alpha-\eta\beta)\sqrt{2}\erfinv\lp\frac{1-\alpha}{1-\beta_0}\rp}=1.
\label{eq:thmfinalldpl12b}
\end{equation}
Finally, let $I^{(p)}_{ldp}(\alpha,\beta)$ be defined through the following \textbf{partial $\ell_1$'s fundamental LDP rate function} characterization
\begin{equation}
I^{(p)}_{ldp}(\alpha,\beta)\triangleq
(\alpha-\eta\beta)\log\left (\frac{\erfinv\left (\frac{1-\alpha}{1-\beta_1}\right )}{\erfinv\left (\frac{1-\alpha}{1-\beta_0}\right )}\right )+(1-\beta)\log\left (\frac{1-\beta}{1-\beta_1}\right ) +\beta(1-\eta)\log\left (\frac{(\alpha-\beta)(1-\beta_0)}{(\alpha-\beta_0)(1-\beta_1)}\right ).
\label{eq:thmfinalldpl13}
\end{equation}
Then if $\alpha>\alpha_w$
\begin{equation}
I^{(p)}_{err}(\alpha,\beta)\triangleq\lim_{n\rightarrow\infty}\frac{\log{P_{err}}}{n}=I^{(p)}_{ldp}(\alpha,\beta).\label{eq:thmfinalldpl14}
\end{equation}
Moreover, if $\alpha<\alpha_w$
\begin{equation}
I^{(p)}_{cor}(\alpha,\beta)\triangleq\lim_{n\rightarrow\infty}\frac{\log{P_{cor}}}{n}=I^{(p)}_{ldp}(\alpha,\beta).\label{eq:thmfinalldpl15}
\end{equation}\label{thm:finalldpl1}
\end{theorem}
\begin{proof} Follows from the above discussion.
\end{proof}
In the following section we will establish a few properties of functions
$\frac{\alpha-\eta\beta_1}{\alpha-\eta\beta}
\xi^{(p)}_{\alpha,\eta}(\beta_1)$ and $  \frac{\alpha-\eta\beta_0}{\alpha-\eta\beta}
\frac{\alpha-\beta}{\alpha-\beta_0}\xi^{(p)}_{\alpha,\eta}(\beta_0)$ that ensure that the above theorem unambiguously characterizes the partial $\ell_1$ LDP. After that we will be in position to showcase the results that one can obatin based on the above theorem.

\subsubsection{Properties of $\frac{\alpha-\eta\beta_1}{\alpha-\eta\beta}
\xi^{(p)}_{\alpha,\eta}(\beta_1)$}
\label{sec:propxibeta1}

The properties of functions $\frac{\alpha-\eta\beta_1}{\alpha-\eta\beta}
\xi^{(p)}_{\alpha,\eta}(\beta_1)$ and $  \frac{\alpha-\eta\beta_0}{\alpha-\eta\beta}
\frac{\alpha-\beta}{\alpha-\beta_0}\xi^{(p)}_{\alpha,\eta}(\beta_0)$ that we will show below will roughly have the same flavor as the properties of functions $\xi^{(p)}_{\alpha,\eta}(\beta)$ and $\psi^{(p)}_{\beta,\eta}(\alpha_w)$ from Theorem \ref{thm:thmweakthr} that we discussed in Section \ref{sec:propxi}. Ultimately they will ensure that the above partial $\ell_1$ LDP characterizations are unambiguous.

Below we start by highlighting that for any fixed $\eta\in(0,1)$ and any fixed $(\alpha,\beta)\in (0,1)\times(0,\alpha)$ there is a unique $\beta_1$ such that $\frac{\alpha-\eta\beta_1}{\alpha-\eta\beta}
\xi^{(p)}_{\alpha,\eta}(\beta_1)=1$ (this is in flavor very similar to what we showed for functions $\xi^{(p)}_{\alpha,\eta}(\beta)$ and $\psi^{(p)}_{\beta,\eta}(\alpha_w)$ in Section \ref{sec:propxi}) and essentially ensures that (\ref{eq:thmfinalldpl12a}) is an unambiguous LDP characterization. To confirm that this is indeed true we proceed in a fashion similar to the one showcased in Section \ref{sec:propxi}.
Namely, we will first show that for any fixed $\eta\in (0,1)$ and any fixed $\alpha\in (0,1)$, $\frac{\alpha-\eta\beta_1}{\alpha-\eta\beta}
\xi^{(p)}_{\alpha,\eta}(\beta_1)=1$ is a decreasing function of $\beta_1$ on interval $[0,\alpha)$. Computing the derivative with respect to $\beta_1$ gives
\begin{eqnarray}\label{eq:propxifinldp1}
  \frac{d(\frac{\alpha-\eta\beta_1}{\alpha-\eta\beta}
\xi^{(p)}_{\alpha,\eta}(\beta_1)}{d\beta_1} & = & \frac{d\lp\frac{(1-\beta_1)\sqrt{\frac{2}{\pi}}e^{-\lp\erfinv\lp\frac{1-\alpha}{1-\beta_1}\rp\rp^2}}{(\alpha-\eta\beta)\sqrt{2}\erfinv \lp\frac{1-\alpha}{1-\beta_1}\rp}-1\rp}{d\beta}\nonumber\\
  & = & \sqrt{\frac{2}{\pi}}\frac{-\frac{\sqrt{\pi} (1-\alpha)}{(1-\beta_1) \erfinv((1-\alpha)/(1-\beta_1))^2}-\frac{2 e^{-\lp \erfinv\lp\frac{1-\alpha}{1-\beta_1}\rp\rp^2}}{\erfinv((1-\alpha)/(1-\beta_1))}-\frac{2\sqrt{\pi} (1-\alpha)}{1-\beta_1}}{2 \sqrt{2} (\alpha-\eta\beta)}\nonumber \\
  &<& 0.
\end{eqnarray}
In Section \ref{sec:propxi} (and ultimately in \cite{Stojnicl1RegPosasymldp}) it was argued that for any fixed $\alpha\in (0,1)$, $\lim_{\beta\rightarrow \alpha} (\alpha-\eta\beta)\xi^{(p)}_{\alpha,\eta}(\beta_1)-1=-1$ which then implies that for any fixed $\eta\in (0,1)$ and any fixed $\alpha\in (0,1)$  one also has $\lim_{\beta\rightarrow \alpha} \frac{\alpha-\eta\beta_1}{\alpha-\eta\beta}\xi^{(p)}_{\alpha,\eta}(\beta_1)-1=-1$. Moreover, in \cite{Stojnicl1RegPosasymldp} we also showed that for any fixed $\alpha\in (0,1)$, $\xi^{(p)}_{\alpha,\eta}(0)-1>0$. Together with (\ref{eq:propxifinldp1}) this is then enough to conclude that for any fixed $\eta\in (0,1)$ and any fixed $\alpha\in (0,1)$ there is a unique $\beta_1$ such that $\frac{\alpha-\eta\beta_1}{\alpha-\eta\beta}\xi^{(p)}_{\alpha,\eta}(\beta_1)=1$, which as mentioned above essentially means that (\ref{eq:thmfinalldpl12a}) is an unambiguous LDP characterization. For the completeness, in Figure \ref{fig:propxibeta1} we present a few numerical results related to the behavior of $\xi^{(p)}_{\alpha,\eta}(\beta)$ that indeed confirm the above calculations.
\begin{figure}[htb]
\begin{minipage}[b]{.5\linewidth}
\centering
\centerline{\epsfig{figure=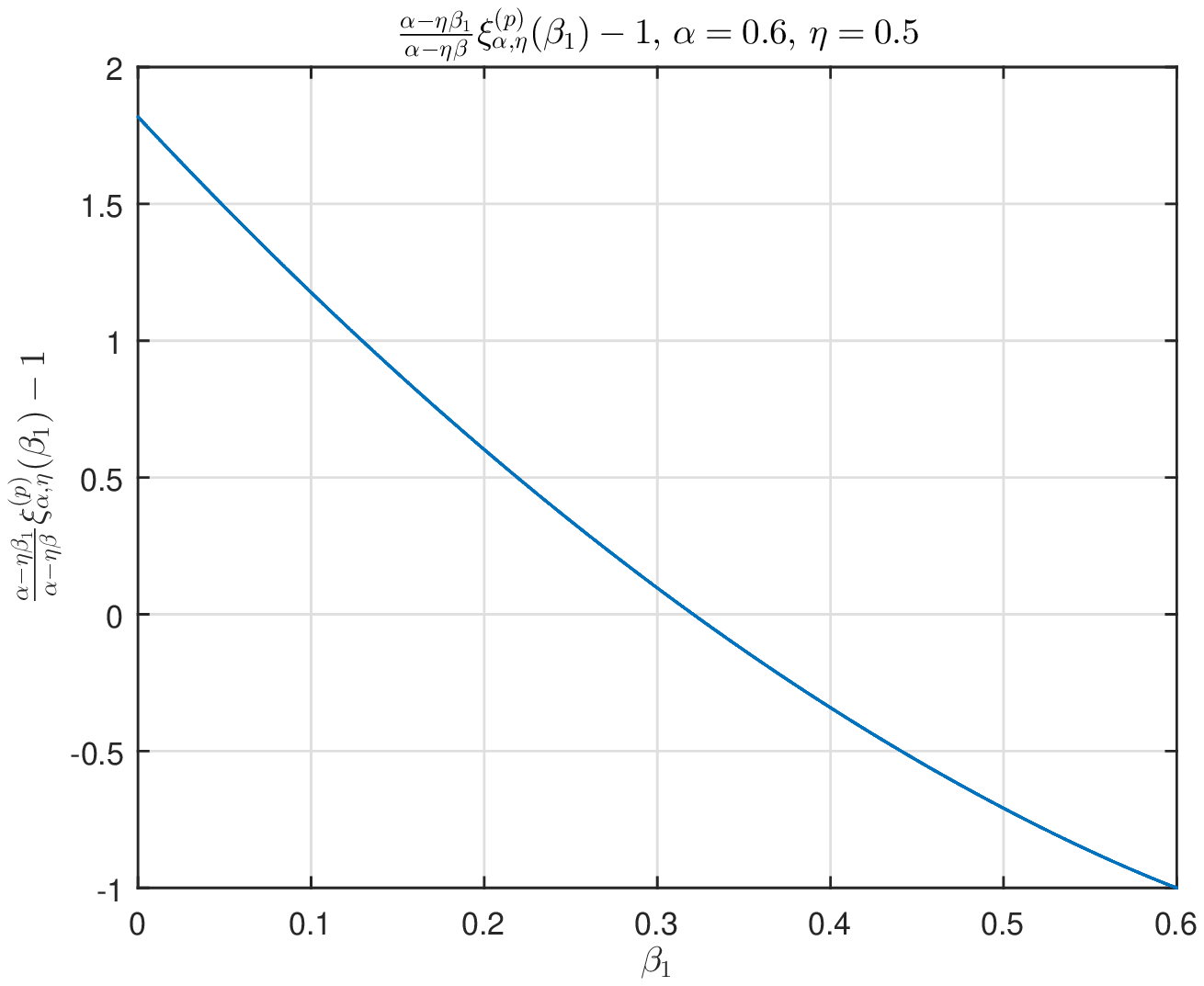,width=9cm,height=7cm}}
\end{minipage}
\begin{minipage}[b]{.5\linewidth}
\centering
\centerline{\epsfig{figure=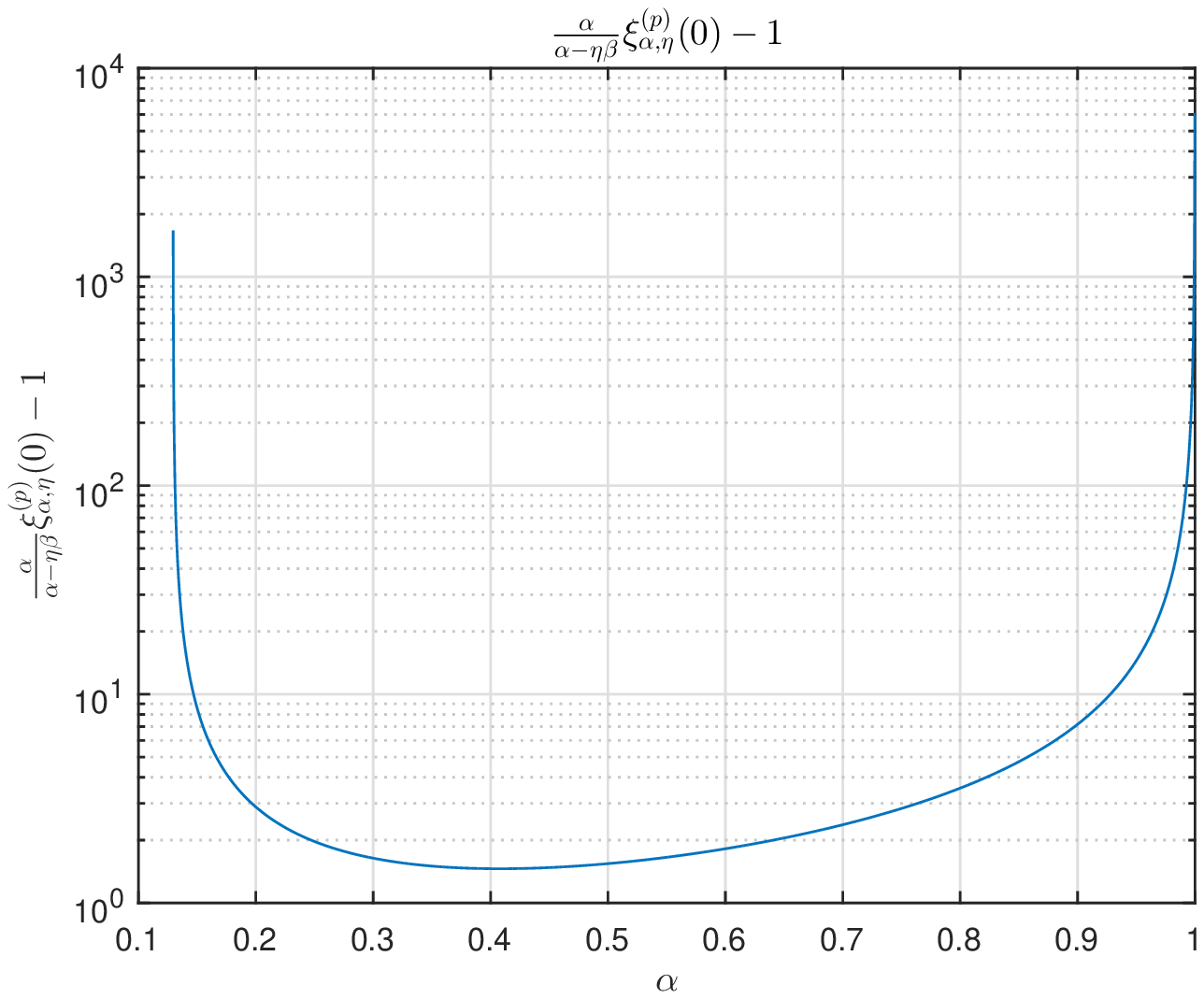,width=9cm,height=7cm}}
\end{minipage}
\caption{Properties of $\frac{\alpha-\eta\beta_1}{\alpha-\eta\beta}\xi^{(p)}_{\alpha,\eta}(\beta_1)$: $\lp \frac{\alpha-\eta\beta_1}{\alpha-\eta\beta}\xi^{(p)}_{\alpha,\eta}(\beta_1)-1\rp$ as a function of $\beta_1$ ($\alpha=0.6$, $\eta=0.5$) -- left; $\lp \frac{\alpha}{\alpha-\eta\beta}\xi^{(p)}_{\alpha,\eta}(0)-1\rp$ as a function of $\alpha$ ($\eta=0.5$)-- right}
\label{fig:propxibeta1}
\end{figure}

\subsubsection{Properties of $  \frac{\alpha-\eta\beta_0}{\alpha-\eta\beta}
\frac{\alpha-\beta}{\alpha-\beta_0}\xi^{(p)}_{\alpha,\eta}(\beta_0)$}
\label{sec:propxibeta0}

In this section we discuss $  \frac{\alpha-\eta\beta_0}{\alpha-\eta\beta}
\frac{\alpha-\beta}{\alpha-\beta_0}\xi^{(p)}_{\alpha,\eta}(\beta_0)$. The reasoning will be slightly different from what we presented above. We start by setting
\begin{equation}\label{eq:propxibeta01}
  q_0=\erfinv\left (\frac{1-\alpha}{1-\beta_0}\right ),
\end{equation}
and then have
\begin{equation}\label{eq:propxibeta02}
  \frac{\alpha-\eta\beta_0}{\alpha-\eta\beta}\frac{\alpha-\beta}{\alpha-\beta_0}
\xi^{(p)}_{\alpha,\eta}(\beta_0)=1 \Leftrightarrow \frac{\alpha-\beta}{\alpha-\eta\beta}\frac{1}{\erfc(q_0)}\frac{\sqrt{\frac{1}{\pi}}e^{-q_0^2}}{q_0}=1
  \Leftrightarrow \frac{\sqrt{\frac{1}{\pi}}e^{-q_0^2}}{q_0}-\erfc(q_0)c_{\alpha,\beta}=0, c_{\alpha,\beta}> 1.
\end{equation}
One can now continue using the same set of arguments as in \cite{Stojnicl1RegPosasymldp} when the properties of function $\frac{\alpha-\beta}{\alpha-\beta_0}\xi_{\alpha}(\beta_0)$ considered there were discussed. Instead of rewriting these arguments we just state the final conclusion that $\left ( \frac{\sqrt{\frac{1}{\pi}}e^{-q_0^2}}{q_0}-\erfc(q_0)c_{\alpha,\beta}\right )$ is decreasing for $q_0\leq \frac{1}{\sqrt{2c_{\alpha,\beta}(c_{\alpha,\beta}-1)}}$ and that $\left ( \frac{\sqrt{\frac{1}{\pi}}e^{-q_0^2}}{q_0}-\erfc(q_0)c_{\alpha,\beta}=0\right )$ has a unique solution (moreover, it is in the interval $(0, \frac{1}{\sqrt{2c_{\alpha,\beta}(c_{\alpha,\beta}-1)}})$). This then implies that $\frac{\alpha-\eta\beta_0}{\alpha-\eta\beta}\frac{\alpha-\beta}{\alpha-\beta_0}\xi^{(p)}_{\alpha,\eta}(\beta_0)=1$ also has a unique solution, i.e. that for any fixed $(\alpha,\beta)\in (0,1)\times (0,\alpha)$ there is a unique $\beta_0$ such that $\frac{\alpha-\eta\beta_0}{\alpha-\eta\beta}\frac{\alpha-\beta}{\alpha-\beta_0}\xi^{(p)}_{\alpha,\eta}(\beta_0)=1$, which as mentioned above essentially means that (\ref{eq:thmfinalldpl12b}) is also an unambiguous LDP characterization. For the completeness, in Figure \ref{fig:propxibeta0} we present a few numerical results related to the behavior of $\left ( \frac{\sqrt{\frac{1}{\pi}}e^{-q_0^2}}{q_0}-\erfc(q_0)c_{\alpha,\beta}\right )$ (and ultimately of $\left (\frac{\alpha-\eta\beta_0}{\alpha-\eta\beta}\frac{\alpha-\beta}{\alpha-\beta_0}\xi^{(p)}_{\alpha,\eta}(\beta_0)-1\right )$) that indeed confirm the above calculations.

\begin{figure}[htb]
\centering
\centerline{\epsfig{figure=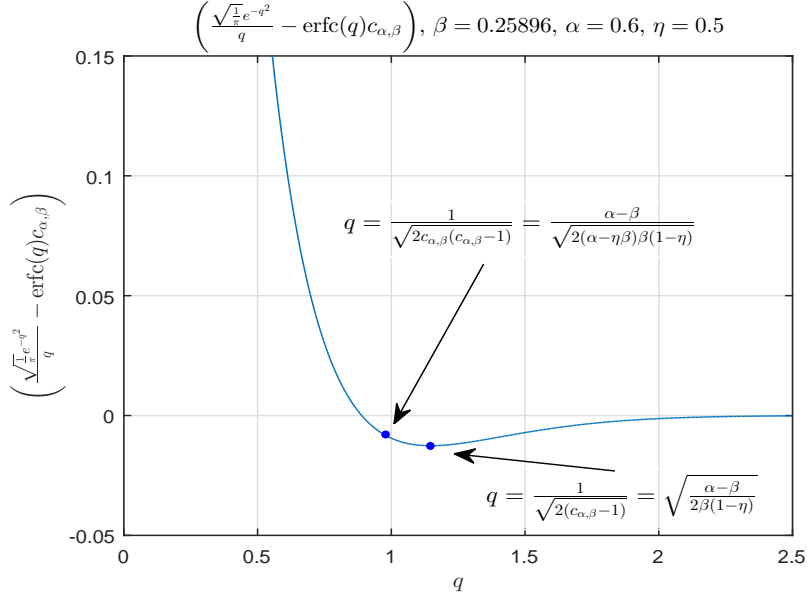,width=11.5cm,height=8cm}}
\caption{Uniqueness of the solution of $\frac{\alpha-\eta\beta_0}{\alpha-\eta\beta}\frac{\alpha-\beta}{\alpha-\beta_0}\xi^{(p)}_{\alpha,\eta}(\beta_0)=1$ is implied by the properties of $\left ( \frac{\sqrt{\frac{1}{\pi}}e^{-q_0^2}}{q_0}-\erfc(q_0)c_{\alpha,\beta}\right )$}
\label{fig:propxibeta0}
\end{figure}


\subsection{Theoretical and numerical LDP results -- partial $\ell_1$}
\label{sec:thnumresuts}

In this section we briefly discuss the results that can be obtained based on what is proven in Theorem \ref{thm:finalldpl1}. In Figure \ref{fig:l1regldpIerrub}, for two different values of $\beta$ ($\beta=0.25896$ and $\beta=\frac{1}{3}$; $\beta=0.25896$ is chosen as the breaking point that one obtains from the partial $\ell_1$ PT for $\alpha=0.5$) we show the obtained theoretical LDP rate function curves (these are of course calculated  based on Theorem \ref{thm:finalldpl1}). To facilitate reading, in Table \ref{tab:Ildptab1} we also show the numerical values for all the quantities of interest in Theorems \ref{thm:ldp3} and \ref{thm:finalldpl1}. On the other hand, in Figure \ref{fig:weakl1LDPthrsim} and Table \ref{tab:Ildptab2} we show how the simulated values compare to the theoretical ones and observe that they are quite close to each other even for small systems dimensions (of order $100$). Of course, we do recall that the LDP computations assume an infinite dimensional asymptotic regime. Having that in mind Figure \ref{fig:weakl1LDPthrsim} and Table \ref{tab:Ildptab2} then indicate that fairly small systems (of size of a few hundreds) already exhibit the LDP properties typical for infinite dimensional systems.


\begin{figure}[htb]
\begin{minipage}[b]{.5\linewidth}
\centering
\centerline{\epsfig{figure=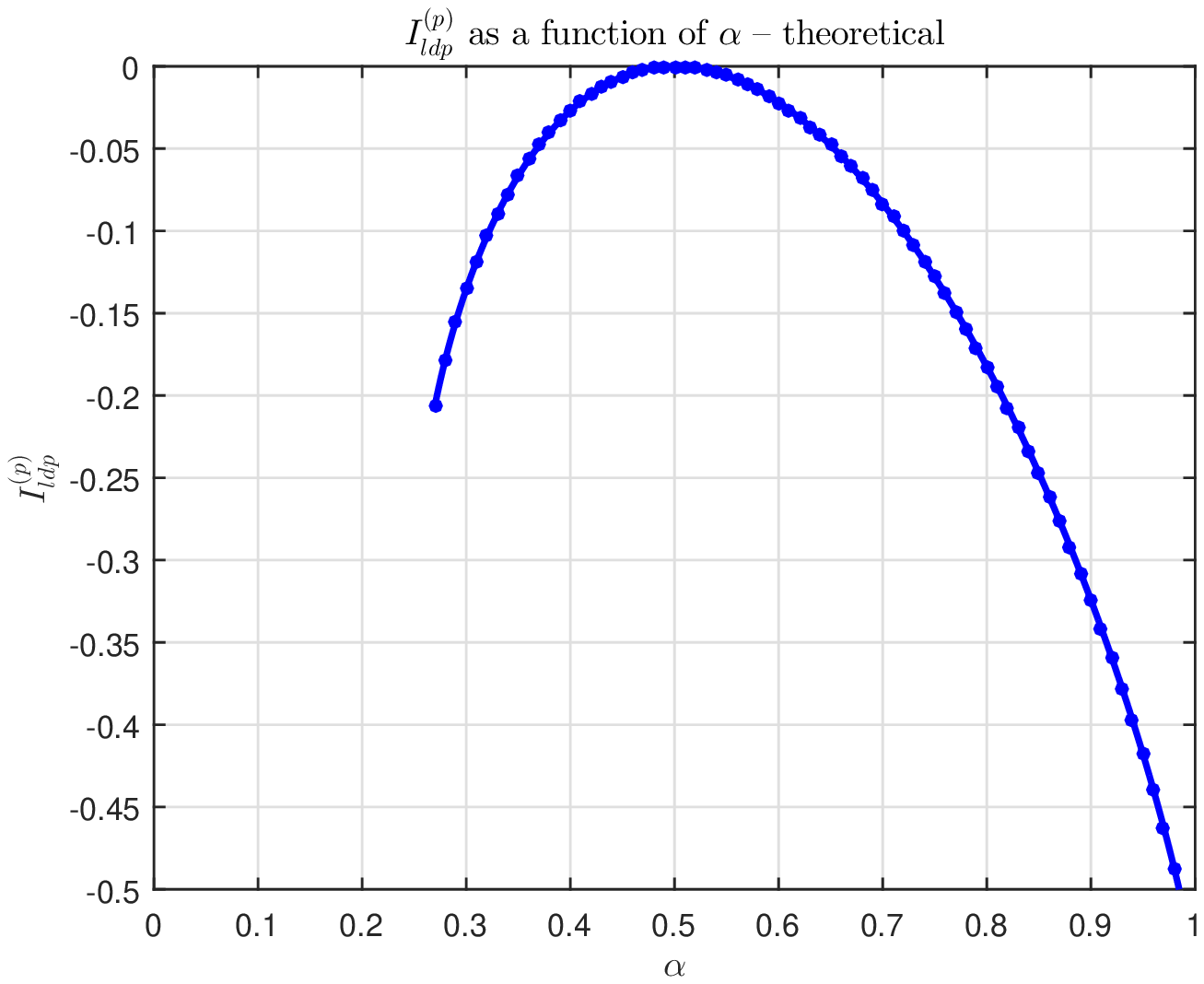,width=9cm,height=7cm}}
\end{minipage}
\begin{minipage}[b]{.5\linewidth}
\centering
\centerline{\epsfig{figure=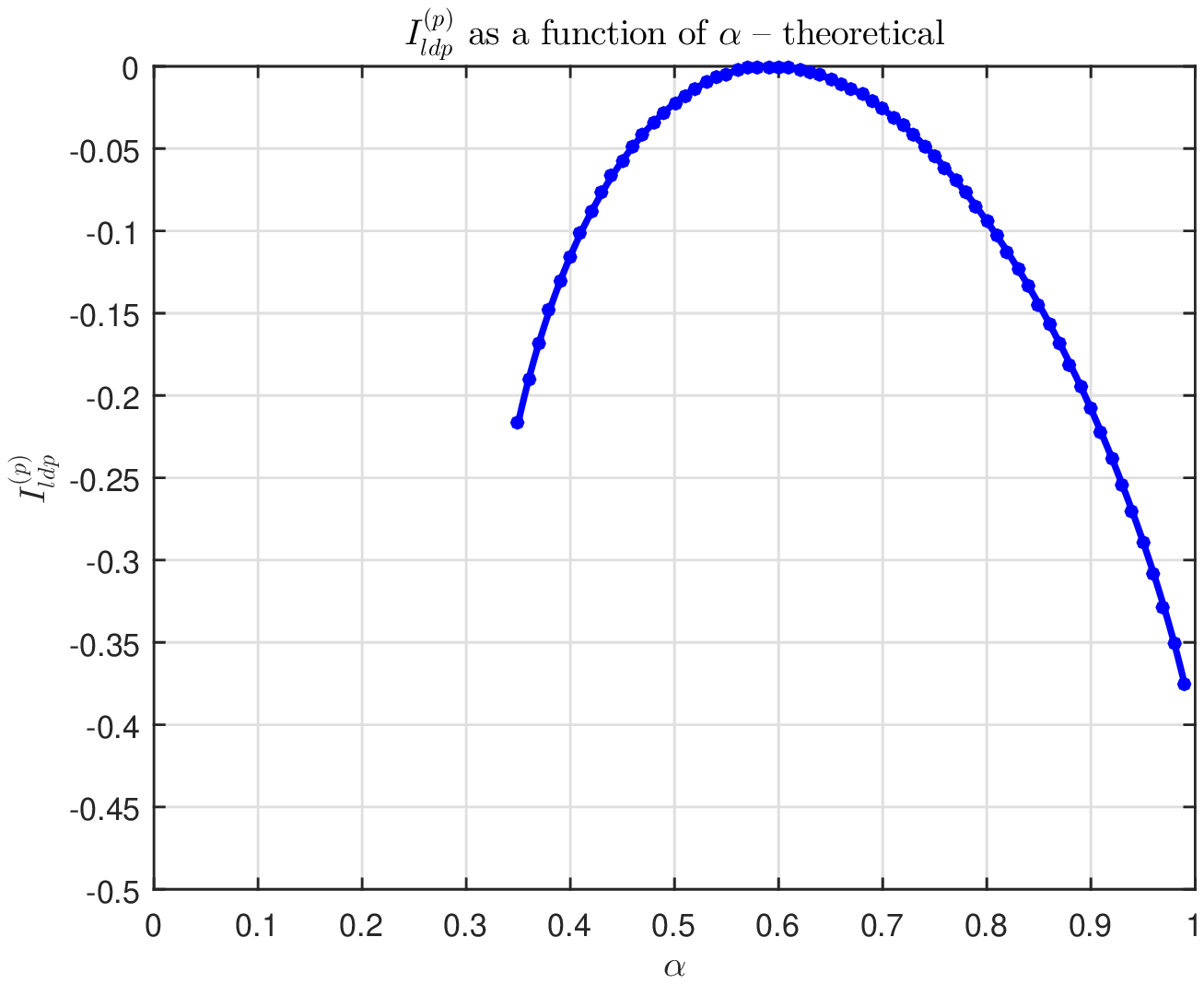,width=9cm,height=7cm}}
\end{minipage}
\caption{$I^{(p)}_{ldp}$ as a function of $\alpha$ for $\eta=0.5$; left -- $\beta=0.25896$; right -- $\beta=\frac{1}{3}$}
\label{fig:l1regldpIerrub}
\end{figure}

\begin{table}[h]
\caption{A collection of values for $\beta_1$, $\beta_0$, $\nu$, $A_0$, $c_3$, $\gamma$, and $I^{(p)}_{ldp}(\alpha,\beta)$ in Theorem \ref{thm:ldp3}; $\beta=0.25896$, $\eta=0.5$}\vspace{.1in}
\hspace{-0in}\centering
\begin{tabular}{||c||c|c|c|c|c||}\hline\hline
$\alpha$ & $ 0.40 $ & $ 0.45 $ & $ 0.50 $ & $ 0.55 $ & $ 0.60 $ \\ \hline\hline
$\beta_1$& $ 0.2095 $ & $ 0.2328 $ & $ 0.2590 $ & $ 0.2881 $ & $ 0.3207 $ \\ \hline
$\beta_0$& $ -0.2306 $ & $ 0.0715 $ & $ 0.2590 $ & $ 0.3918 $ & $ 0.4946 $ \\ \hline\hline
$\nu$    & $ 1.1725 $ & $ 1.0734 $ & $ 0.9837 $ & $ 0.9005 $ & $ 0.8218 $ \\ \hline
$A_0$    & $ 1.7900 $ & $ 1.2964 $ & $ 1.0000 $ & $ 0.7997 $ & $ 0.6534 $ \\ \hline
$c_3$    & $ -0.7788 $ & $ -0.3522 $ & $ -0.0000 $ & $ 0.3343 $ & $ 0.6793 $ \\ \hline
$\gamma$ & $ 0.1767 $ & $ 0.2587 $ & $ 0.3536 $ & $ 0.4637 $ & $ 0.5927 $ \\ \hline\hline
$I^{(p)}_{ldp}(\alpha,\beta)$& $ \mathbf{-0.0270} $ & $ \mathbf{-0.0063} $ & $ \mathbf{-0.0000} $ & $ \mathbf{-0.0057} $ & $ \mathbf{-0.0220} $ \\ \hline\hline
\end{tabular}
\label{tab:Ildptab1}
\end{table}

\begin{figure}[htb]
\centering
\centerline{\epsfig{figure=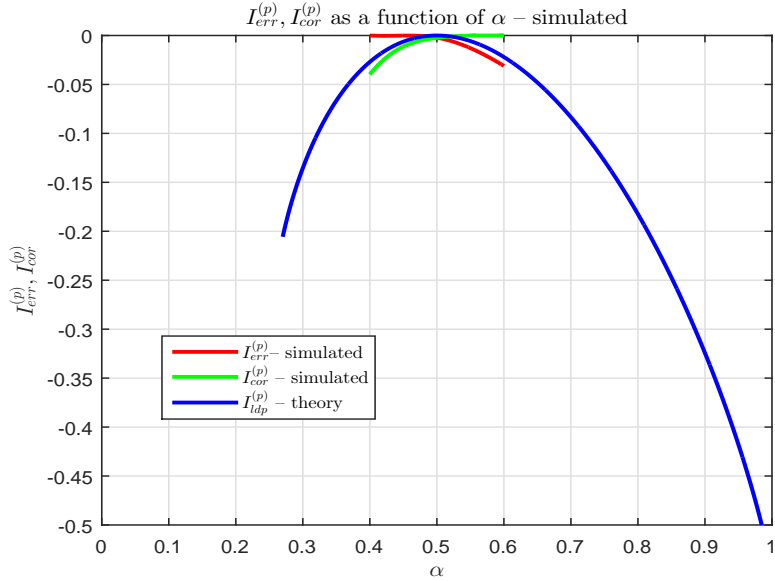,width=11.5cm,height=8cm}}
\caption{Partial $\ell_1$'s weak LDP rate function -- theory and simulation; $\beta=0.25896$, $\eta=0.5$}
\label{fig:weakl1LDPthrsim}
\end{figure}

\begin{table}[h]
\caption{$I^{(p)}_{err}(\alpha,\beta)$, $I^{(p)}_{cor}(\alpha,\beta)$ -- simulated; $I^{(p)}_{ldp}(\alpha,\beta)$ calculated for $\beta=0.25896$ and $\eta=0.5$}\vspace{.1in}
\hspace{-0in}\centering
\begin{tabular}{||c||c|c|c|c|c||}\hline\hline
$\alpha$ & $ 0.40 $ & $ 0.45 $ & $ 0.50 $ & $ 0.55 $ & $ 0.60 $ \\ \hline\hline
$k$      & $ 52 $ & $ 78 $ & $ 78 $ & $ 78 $ & $ 52 $\\ \hline
$m$      & $ 80 $ & $ 135 $ & $ 150 $ & $ 165 $ & $ 120 $\\ \hline
$n$      & $ 200 $ & $ 300 $ & $ 300 $ & $ 300 $ & $ 200 $\\ \hline\hline
$I^{(p)}_{err}(\alpha,\beta)$ -- simulated & $ -0.0000 $ & $ -0.0001 $ & \red{$ \mathbf{-0.0022} $} & \red{$ \mathbf{-0.0129} $} & \red{$ \mathbf{-0.0311} $} \\ \hline
$I^{(p)}_{cor}(\alpha,\beta)$ -- simulated & \gr{$ \mathbf{-0.0398} $} & \gr{$ \mathbf{-0.0113} $} & \gr{$ \mathbf{-0.0024} $} & $ -0.0001 $ & $ -0.0000 $ \\ \hline\hline
$I^{(p)}_{ldp}(\alpha,\beta)$ -- theory & \bl{$ \mathbf{-0.0270} $} & \bl{$ \mathbf{-0.0063} $} & \bl{$ \mathbf{-0.0000} $} & \bl{$ \mathbf{-0.0057} $} & \bl{$ \mathbf{-0.0220} $} \\ \hline\hline
\end{tabular}
\label{tab:Ildptab2}
\end{table}

\section{Hidden partial $\ell_1$}
\label{sec:hidparell1}

In this section we will look at an alternative form of the partial $\ell_1$ to which we will refer as the \emph{hidden} partial $\ell_1$. Before we introduce the hidden partial $\ell_1$ we will briefly discuss why one would be interesting in such a variant of the standard or even partial $\ell_1$.

As is now well known the main source of hardness of the problem in (\ref{eq:l0}) is determining the location of the nonzero components of $\x$. (\ref{eq:l1}), the standard $\ell_1$ relaxation of (\ref{eq:l0}), happens to have the ability to correctly identify these locations in some cases. In statistical contexts this has been rigorously mathematically confirmed through the work of \cite{DonohoPol,DonohoUnsigned,StojnicCSetam09,StojnicUpper10} and the large systems dimensions for which this happens have been fully determined. In the previous section though we went a bit further and looked at a bit relaxed scenario where one assumes that some of the unknown locations are a priori known. A modified version of (\ref{eq:l1}), namely (\ref{eq:l1imp}), is then typically employed and \cite{StojnicICASSP10knownsupp,StojnicTowBettCompSens13} provided its a detailed rigorous mathematical performance characterization. As discussed earlier, the main conclusion was that if some additional information about the unknown locations is available the recoverable sparsity should be higher. In fact not only that, \cite{StojnicICASSP10knownsupp,StojnicTowBettCompSens13} precisely quantified how much higher it will be. That in a way relaxed the original sparse recovery problem (\ref{eq:l0}), so that now one can, instead of solely searching for the algorithms that solve it exactly, also look at the algorithms that solve it so that only a fraction of the unknown locations (i.e. $supp(\x)$) is correctly identified.

Of course, the above comes with a small catch. Namely, if one is to expect (\ref{eq:l1imp}) to be as successful as the results of the previous section (and \cite{StojnicICASSP10knownsupp,StojnicTowBettCompSens13}) predict, one should ensure that a set $\Pi$ contains nothing more than a subset of $supp(\x)$. That can be a bit problematic though. Namely, most of the algorithms that fail to solve (\ref{eq:l0}) still provide an estimate for $supp(\x)$ that does contain a portion of the true $supp(\x)$. That is their good feature and is similar to what one needs to be able to run (\ref{eq:l1imp}). The problem typically is that part of such estimates are also the locations that are not in $supp(\x)$ and (\ref{eq:l1imp}) and its analysis from the previous sections do not allow for that. To handle this kind of situation \cite{StojnicTowBettCompSens13} went a bit further and introduced the so-called \emph{hidden} partially known $\ell_1$ as a modification of the standard/partial $\ell_1$. Before explaining this modification we will first briefly recall on a couple of definitions introduced in \cite{StojnicTowBettCompSens13}. We start by introducing vectors with \emph{hidden} partially known support. Let $\kappa\subset \{1,2,\dots,n\}$ and let the cardinality of $\kappa$ be $k$ (we will for the simplicity choose $k$; however our results easily extend to any other value). Let $\Pi$ be the intersection of the set of nonzero locations of $\x$ ($supp(\x)$) and $\kappa$. As in the previous sections, $\Pi$ is the set that is known to contain locations of some of the nonzero elements of $\x$. Differently though from what was the case earlier, $\Pi$ is not known now. Instad $\kappa$ is now known and the fact that $\Pi\in\kappa$. For the concreteness, we will again assume that the cardinality of $\Pi$ is $\eta k$ (where $\eta$ is again a constant independent of $n$ and $0\leq \eta\leq 1$) and that $\x$ is a vector with \emph{hidden} partially known support (clearly, $\kappa$ will represent the estimate of $\x$'s support ($supp(\x)$)). Then the above mentioned hidden partial $\ell_1$ assumes the following slight adjustment to (\ref{eq:l1imp})
\begin{eqnarray}
\mbox{min} & & \sum_{i\notin\kappa} |\x_i|\nonumber \\
\mbox{subject to} & & A\x=\y. \label{eq:l1imphidden}
\end{eqnarray}
\cite{StojnicTowBettCompSens13} then proceeded a bit further and provided a performance characterization of (\ref{eq:l1imphidden}). Such performance characterization relates to its PT characterizations. Here we will further widen our understanding of (\ref{eq:l1imphidden}) and the PT phenomena that comes with it by providing a set of LDP type of results in flavor similar to the ones obtained in earlier sections for (\ref{eq:l1imp}). We will split the remaining presentation into two main parts, the first one that we will use to discuss the hidden partial $\ell_1$'s PT itself and the second one that we will use to discuss the corresponding LDP.

\subsection{Hidden partial $\ell_1$ -- phase transitions}
\label{sec:hidparphasetrans}

In this section we discuss the phase transitions (PTs) of the hidden partial $\ell_1$. All the definitions introduced when we discussed the partial $\ell_1$ remain in place. Knowing that then the PTs are of course fully characterized once the corresponding PT curves in $(\alpha,\beta)$ plane are determined. The following theorem determines these curves and automatically settles the hidden partial $\ell_1$ weak PTs.
\begin{theorem}(\cite{StojnicTowBettCompSens13} Exact hidden partial $\ell_1$'s weak threshold/PT)
Let $A$ be an $m\times n$ matrix in (\ref{eq:l0})
with i.i.d. standard normal components. Let
the unknown $\x$ that solves (\ref{eq:l0}) be $k$-sparse. Further, let the location and signs of nonzero elements of $\x$ be arbitrarily chosen but fixed. Moreover, let the set of nonzero locations of $\x$ be $K$. Let $\kappa\subset \{1,2,\dots,n\}$ be a given set of cardinality $k$ such that the cardinality of set $K\cap \kappa$ is $\eta k$. Let $k,m,n$ be large
and let $\alpha_w=\frac{m}{n}$ and $\beta_w=\frac{k}{n}$ be constants
independent of $m$ and $n$. Let $\erfinv$ be the inverse of the standard error function associated with zero-mean unit variance Gaussian random variable.  Further, let $\alpha_w$ and $\beta_w$ satisfy the following \textbf{fundamental characterization of the \emph{hidden partial} $\ell_1$'s PT}

\begin{center}
\shadowbox{$
\xi^{(hp)}_{\alpha_{w},\eta}(\beta_w)\triangleq\psi^{(hp)}_{\beta_w,\eta}(\alpha_{w})\triangleq
\frac{(1-\beta_w(2-\eta))\sqrt{\frac{2}{\pi}}e^{-\lp\erfinv\lp\frac{1-\alpha_w}{1-\beta_w(2-\eta)}\rp\rp^2}}{(\alpha_w-\beta_w)\sqrt{2}\erfinv \lp\frac{1-\alpha_w}{1-\beta_w(2-\eta)}\rp}=1.
$}
-\vspace{-.5in}\begin{equation}
\label{eq:thmweaktheta2hidpar}
\end{equation}
\end{center}

Then:
\begin{enumerate}
\item If $\alpha>\alpha_w$ then with overwhelming probability the solution of (\ref{eq:l1imphidden}) is the $k$-sparse $\x$ that solves (\ref{eq:l0}).
\item If $\alpha<\alpha_w$ then with overwhelming probability there will be a $k$-sparse $\x$ (from a set of $\x$'s with fixed locations and signs of nonzero components) such that the solution of (\ref{eq:l0}) is \textbf{not} the solution of (\ref{eq:l1imphidden}).
    \end{enumerate}
\label{thm:thmweakthrhidpar}
\end{theorem}
\begin{proof}
It was established in \cite{StojnicTowBettCompSens13}.
\end{proof}
It is now also relatively easy to see that the above PT characterization is unambiguous. Namely, the change $\beta\leftarrow (2-\eta)\beta$ and $\eta\leftarrow \frac{1}{2-\eta}$ transforms the above hidden partial $\ell_1$ PT into the partial $\ell_1$ PT which by considerations provided in Section \ref{sec:propxi} is unambiguous (both sets of equations, $\xi^{(hp)}_{\alpha,\eta}(\beta)=1$ and $\psi^{(hp)}_{\beta,\eta}(\alpha)=1$ have unique solutions for any fixed $\eta\in(0,1)$ and any fixed $\alpha$, $\beta$ respectively). In Figure \ref{fig:weakl1PThidpar} we show the theoretical PT curves that one can obtain based on (\ref{eq:thmweaktheta2hidpar}) for several different values of $\eta$. Clearly, as $\eta$ increases the recoverable sparsity increases as well. In other words, the larger the size of the hidden partially known support of $\x$ the larger the cardinality of $supp(\x)$ that be recovered through hidden partial $\ell_1$ as well.
\begin{figure}[htb]
\centering
\centerline{\epsfig{figure=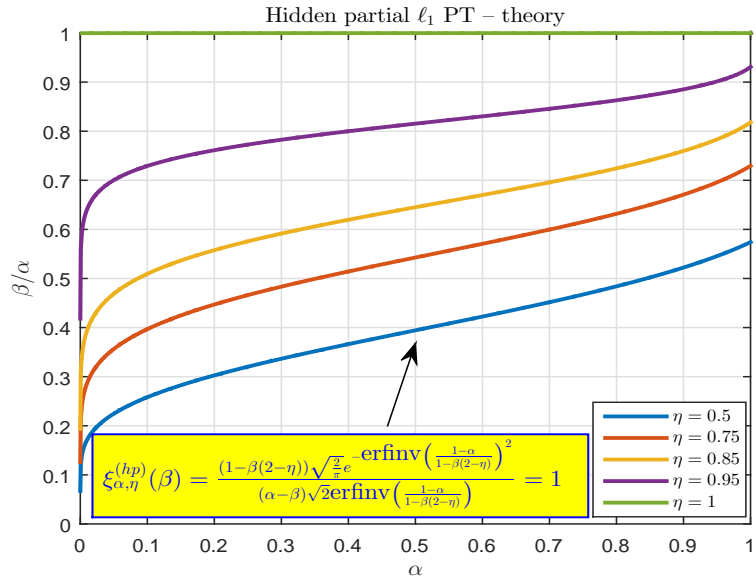,width=11.5cm,height=8cm}}
\caption{Hidden partial $\ell_1$'s weak PT; $\{(\alpha,\beta)|\xi^{(hp)}_{\alpha,\eta}(\beta)=1\}$}
\label{fig:weakl1PThidpar}
\end{figure}

\subsection{Hidden partial $\ell_1$ -- large deviations}
\label{sec:ldphidpar}

In this section we discuss the LDP of the hidden partial $\ell_1$. We will try to emulate as much as possible what was done when we analyzed the LDPs of the partial $\ell_1$. Along the same lines we will then skip all the arguments that directly translate to the hidden case and instead focus on those that bring/emphasize the difference. To that end we start by recalling on a couple of results that we established in \cite{StojnicTowBettCompSens13} (and in their core form in \cite{StojnicCSetam09,StojnicICASSP09,StojnicICASSP10knownsupp}). We again emphasize that these are among the key unsung heros of all the success that we achieved in designing our probabilistic approach for characterizing PTs and LDPs.

As was done in Section \ref{sec:ldp}, for the concreteness/simplicity and without loss of generality we will assume that the elements $\x_{1},\x_{2},\dots,\x_{n-k}$ of $\x$ are equal to zero and that the elements $\x_{n-k+1},\x_{n-k+2},\dots,\x_n$ have fixed signs, say they all are positive (these signs of course are not known beforehand and can not be used in the algorithms design). Moreover, we will also without loss of generality assume that $\kappa$ in (\ref{eq:l1imphidden}) is $\kappa=\{n-k-(1-\eta)k+1,n-k-(1-\eta)k+2,\dots,n-k,n-\eta k+1,n-\eta k+2,\dots,n\}$, where $0\leq\eta\leq 1$ (this is fairly obvious but for the completeness we state that when $k$ is finite $\eta$ is throughout the paper assumed to be such that $\eta k$ is an integer). The following was obtained in \cite{StojnicTowBettCompSens13} relying on the breakthrough observations of \cite{StojnicCSetam09,StojnicICASSP09,StojnicICASSP10knownsupp} and is one of the key features that enabled us to run the entire machinery developed in \cite{StojnicCSetam09,StojnicICASSP09,StojnicICASSP10knownsupp,StojnicTowBettCompSens13}.
\begin{theorem}(\cite{StojnicTowBettCompSens13} Nonzero elements of $\x$ have fixed signs and location)
Assume that an $m\times n$ measurement matrix $A$ is given. Let $\x$
be a $k$ sparse vector. Also let $\x_1=\x_2=\dots=\x_{n-k}=0$. Let the signs of $\x_{n-k+1},\x_{n-k+2},\dots,\x_n$ be fixed, say all positive. let $\kappa=\{n-k-(1-\eta)k+1,n-k-(1-\eta)k+2,\dots,n-k,n-\eta k+1,n-\eta k+2,\dots,n\}$, where $0\leq\eta\leq 1$. Further, assume that $\y\triangleq A\x$ and that $\w$ is
an $n\times 1$ vector. If
\begin{equation}
(\forall \w\in \textbf{R}^n | A\w=0) \quad  -\sum_{i=n-k+1}^{n-\eta k} -\w_i<\sum_{i=1}^{n-k-(1-\eta)k}|\w_{i}|,
\end{equation}
then the solutions of (\ref{eq:l1imphidden}) and (\ref{eq:l0}) coincide. Moreover, if
\begin{equation}
(\exists \w\in \textbf{R}^n | A\w=0) \quad  -\sum_{i=n-k+1}^{n-\eta k} -\w_i\geq \sum_{i=1}^{n-k-(1-\eta)k}|\w_{i}|,
\label{eq:thmeqgenhidpar}
\end{equation}
then there will be a $k$-sparse $\x$ (from the set of $\x$'s with fixed locations and signs of nonzero components) such that the solution of (\ref{eq:l0}) and is not the solution of (\ref{eq:l1imphidden}).
\label{thm:thmknownsuppcondhidpar}
\end{theorem}
To facilitate the exposition we set
\begin{equation}
\Sw^{(hp)}\triangleq\{\w\in S^{n-1}| \quad -\sum_{i=n-k+1}^{n-\eta k} \w_i<\sum_{i=1}^{n-k-(1-\eta)k}|\w_{i}|\}.\label{eq:defSwpr}
\end{equation}
As in Section \ref{sec:ldp}, we will start with the analysis of the so-called LDPs upper tail. Fairly soon it will then become clear that the analogous versions of all other results derived in Section \ref{sec:ldp} will quickly follow.

\subsubsection{Upper tail}
\label{sec:uppertailhidpar}

As usual, the LDPs upper tail assumes consideration of points $(\alpha,\beta)$ such that $\alpha\geq \alpha_w$ where $\alpha_w$ is such that $\psi^{(hp)}_{\beta,\eta}(\alpha_w)=\xi^{(hp)}_{\alpha_w,\eta}(\beta)=1$. For the time being we assume this regime and follow what was done in Section \ref{sec:ldp}. Namely, assuming that the elements of $A$ are i.i.d. standard normals we look at the following probability
\begin{equation}
P^{(hp)}_{err}\triangleq P(\min_{\w\in S^{(hp)}_w}\|A\w\|_2\leq 0)=P(\max_{\w\in S^{(hp)}_w}\min_{\|\y\|_2=1}(\y^T A\w )\geq 0).
\label{eq:ldpprobhidpar}
\end{equation}
where, as earlier, $P^{(hp)}_{err}$ is the so-called probability of error/failure, i.e. the probability that (\ref{eq:l1imphidden}) fails to produce the solution of (\ref{eq:l0}). As in (\ref{eq:widthdefSw}) we, for any $c_3\geq 0$, have
\begin{equation}
P^{(hp)}_{err}\leq e^{-\frac{c_3^2}{2}}E\max_{\w\in S^{(hp)}_w}\min_{\|\y\|_2=1}e^{-c_3(\y^T A\w+g)}\leq
e^{-\frac{c_3^2}{2}}Ee^{-c_3\|\g\|_2}Ee^{c_3w(\h,S^{(hp)}_w)},
\label{eq:ldpprob3hidpar}
\end{equation}
where
\begin{equation}
w(\h,\Sw^{(hp)})\triangleq\max_{\w\in \Sw^{(hp)}} (\h^T\w), \label{eq:widthdefSwhidpar}
\end{equation}
and, as earlier, the elements of $\h$ are i.i.d. standard normals. Continuing further as in Section \ref{sec:ldp}, we have for $w(\h,\Sw^{(hp)})$ in (\ref{eq:widthdefSwhidpar})
\begin{eqnarray}
w(\h,\Sw^{(hp)}) = \max_{\bar{\y}\in \mR^{n}} & &  \sum_{i=1}^{n} \hw_i \bar{\y}_i\nonumber \\
\mbox{subject to} &  & \bar{\y}_i\geq 0, 0\leq i\leq n-k\nonumber \\
& & \sum_{i=n-k+1}^{n-\eta k}\bar{\y}_i\geq \sum_{i=1}^{n-k-(1-\eta)k} \bar{\y}_i \nonumber \\
& & \sum_{i=1}^{n}\bar{\y}_i^2\leq 1.\label{eq:workww2hidpar}
\end{eqnarray}
Finally, after solving (\ref{eq:workww2hidpar}) as in Section \ref{sec:ldp} (and ultimately in \cite{StojnicCSetam09,StojnicCSetamBlock09,StojnicBlockasymldpfinn15,StojnicLiftStrSec13,StojnicICASSP10knownsupp,StojnicTowBettCompSens13}) one obtains
\begin{eqnarray}
w(\h,\Sw) & = & \min_{\nu\geq0,\gamma\geq 0} (\frac{\sum_{i=1}^{n-k-(1-\eta)k}\max(\hw_i-\nu,0)^2+\sum_{i=n-k+1}^{n-\eta k}(\hw_i+\nu)^2}{4\gamma}\nonumber \\
&&+\frac{\sum_{i=n-\eta k+1}^{n}(\hw_i)^2+\sum_{i=n-k-(1-\eta)k+1}^{n-k}(\hw_i)^2}{4\gamma}+\gamma)\nonumber \\
& = & \min_{\nu\geq0}\sqrt{\sum_{i=1}^{n-k-(1-\eta)k}\max(\hw_i-\nu,0)^2+\sum_{i=n-k+1}^{n-\eta k}(\hw_i+\nu)^2+\sum_{i=n-\eta k+1}^{n}(\hw_i)^2+\sum_{i=n-k-(1-\eta)k+1}^{n-k}(\hw_i)^2}.\nonumber \\\label{eq:ldpwhSw}
\end{eqnarray}
We summarize the above methodology to upper bound $P^{(hp)}_{err}$ in the following theorem.
\begin{theorem}
Let $A$ be an $m\times n$ matrix in (\ref{eq:l0})
with i.i.d. standard normal components. Let
the unknown $\x$ in (\ref{eq:l0}) be $k$-sparse and let the location and the signs of nonzero elements of $\x$ be arbitrarily chosen but fixed. Moreover, let the set of nonzero locations of $\x$ be $K$. Let $\kappa\subset \{1,2,\dots,n\}$ be a given set of cardinality $k$ such that the cardinality of set $K\cap \kappa$ is $\eta k$. Let $P^{(hp)}_{err}$ be the probability that the solution of (\ref{eq:l1imphidden}) is not the solution of (\ref{eq:l0}). Then
\begin{eqnarray}
P^{(hp)}_{err}& \leq & \min_{c_3\geq 0}e^{-\frac{c_3^2}{2}}e^{-c_3\|\g\|_2}Ee^{c_3w(\h,S_w)}\nonumber \\
& = & \min_{c_3\geq 0}\left (e^{-\frac{c_3^2}{2}}\frac{1}{\sqrt{2\pi}^m}\int_{\g}e^{-\sum_{i=1}^{m}\g_i^2/2-c_3\|\g\|_2}d\g \min_{\nu\geq 0,\gamma\geq\frac{c_3}{2}} w_1^{n-k(2-\eta)}w_2^{k(1-\eta)}w_3^{k}e^{c_3\gamma}\right ),
\label{eq:ldpthm1perrub1hidpar}
\end{eqnarray}
where
\begin{eqnarray}
w_1 &=& \frac{1}{\sqrt{2\pi}}\int_{\bar{h}}e^{-\bar{h}^2/2}e^{c_3\max(|\bar{h}|-\nu,0)^2/4/\gamma}d\bar{h}
  =\frac{e^{\frac{c_3\nu^2/4/\gamma}{1-c_3/2/\gamma}}}{\sqrt{1-c_3/2/\gamma}}\erfc\left (\frac{\nu}{\sqrt{2}\sqrt{1-c_3/2/\gamma}}\right )+\erf\left (\frac{\nu}{\sqrt{2}}\right )\nonumber \\
w_2 &=& \frac{1}{\sqrt{2\pi}}\int_{\bar{h}}e^{-\bar{h}^2/2}e^{c_3(\bar{h}+\nu)^2/4/\gamma}d\bar{h}
  =\frac{e^{\frac{c_3\nu^2/4/\gamma}{1-c_3/2/\gamma}}}{\sqrt{1-c_3/2/\gamma}}\nonumber \\
w_3 &=& \frac{1}{\sqrt{2\pi}}\int_{\bar{h}}e^{-\bar{h}^2/2}e^{c_3(\bar{h})^2/4/\gamma}d\bar{h}
  =\frac{1}{\sqrt{1-c_3/2/\gamma}}.\label{eq:ldpthm1perrub2hidpar}
\end{eqnarray}\label{thm:ldp1hidpar}
\end{theorem}
\begin{proof}
Follows from the above considerations, what was presented in Section \ref{sec:ldp}, and ultimately through an adaptation of the mechanisms developed in \cite{StojnicCSetam09,StojnicCSetamBlock09,StojnicBlockasymldpfinn15,StojnicLiftStrSec13,StojnicICASSP10knownsupp,StojnicTowBettCompSens13}.
\end{proof}
As in Section \ref{sec:ldp}, our main concern below is the asymptotic regime, the same one as in Theorem \ref{thm:thmweakthrhidpar}. In particular, and following \cite{Stojnicl1RegPosasymldp}, we will be interested in the rate, $I^{(hp)}_{err}(\alpha,\beta)$, at which $P^{(hp)}_{err}$ decays
\begin{equation}\label{eq:ldpasymp1hidpar}
  I^{(hp)}_{err}(\alpha,\beta)\triangleq\lim_{n\rightarrow\infty}\frac{\log{P^{(hp)}_{err}}}{n}.
\end{equation}
Based on Theorem \ref{thm:ldp1hidpar} we have the following LDP type of theorem.
\begin{theorem}
Assume the setup of Theorem \ref{thm:ldp1hidpar}. Further, let integers $m$, $k$, and $n$ be large ($k\leq m\leq n$) such that $\beta=\frac{k}{n}$ and $\alpha=\frac{m}{n}$ are constants independent of $n$. Assume that a pair $(\alpha,\beta)$  is given. Also, assume the following scaling: $c_3\rightarrow c_3\sqrt{n}$ and $\gamma\rightarrow\gamma\sqrt{n}$. Then
\begin{eqnarray}
I^{(p)}_{err}(\alpha,\beta)& \triangleq &\lim_{n\rightarrow\infty}\frac{\log{P^{(hp)}_{err}}}{n}\nonumber \\
& \leq & \min_{c_3\geq 0}\left (-\frac{(c_3)^2}{2}+I_{sph}+\min_{\nu\geq 0,\gamma\geq \frac{c_3}{2}} ((1-\beta(2-\eta))\log{w_1}+(1-\eta)\beta\log{w_2}+\beta\log{w_3}+c_3\gamma)\right )\nonumber \\
& \triangleq &I_{err,u}^{(p,ub)}(\alpha,\beta),
\label{eq:ldpthm2Ierrub1hidpar}
\end{eqnarray}
where
\begin{eqnarray}
I_{sph} &=& \widehat{\gamma}c_3-\frac{\alpha }{2}\log\left (1-\frac{c_3}{2\widehat{\gamma}}\right )\nonumber \\
  \widehat{\gamma} &=& \frac{c_3-\sqrt{(c_3)^2+4\alpha}}{4}\nonumber \\
w_1 &=& \frac{1}{\sqrt{2\pi}}\int_{\bar{h}}e^{-\bar{h}^2/2}e^{c_3\max(|\bar{h}|-\nu,0)^2/4/\gamma}d\bar{h}
  =\frac{e^{\frac{c_3\nu^2/4/\gamma}{1-c_3/2/\gamma}}}{\sqrt{1-c_3/2/\gamma}}\erfc\left (\frac{\nu}{\sqrt{2}\sqrt{1-c_3/2/\gamma}}\right )+\erf\left (\frac{\nu}{\sqrt{2}}\right )\nonumber \\
  w_2 &=& \frac{1}{\sqrt{2\pi}}\int_{\bar{h}}e^{-\bar{h}^2/2}e^{c_3(\bar{h}+\nu)^2/4/\gamma}d\bar{h}
  =\frac{e^{\frac{c_3\nu^2/4/\gamma}{1-c_3/2/\gamma}}}{\sqrt{1-c_3/2/\gamma}}\nonumber \\
  w_3 &=& \frac{1}{\sqrt{2\pi}}\int_{\bar{h}}e^{-\bar{h}^2/2}e^{c_3(\bar{h})^2/4/\gamma}d\bar{h}
  =\frac{1}{\sqrt{1-c_3/2/\gamma}}.\label{eq:ldpthm2perrub2hidpar}
\end{eqnarray}\label{thm:ldp2hidpar}
\end{theorem}
\begin{proof} As Theorem \ref{thm:ldp2}, follows in a fashion analogous to the one employed in \cite{Stojnicl1RegPosasymldp}.
\end{proof}
Now one could repeat all the arguments after Theorem \ref{thm:ldp2}. There is no need to do that though after one observes that the change $\beta\leftarrow (2-\eta)\beta$ and $\eta\leftarrow \frac{1}{2-\eta}$ transforms the above LDP characterization of the hidden partial $\ell_1$ PT into the one given for the partial $\ell_1$ in Theorem \ref{thm:ldp2}. One then automatically arrives to the following hidden partial analogue of Theorem \ref{thm:finalldpl1}.
\begin{theorem}[Hidden partial $\ell_1$'s LDP]
Assume the setup of Theorems \ref{thm:ldp1hidpar} and \ref{thm:ldp2hidpar} with $\eta^{(hp)}$ as the cardinality of set $K\cap \kappa$ and assume that a pair $(\alpha,\beta^{(hp)})$ is given. Let $P^{(hp)}_{err}$ be the probability that the solutions of (\ref{eq:l0}) and (\ref{eq:l1imphidden}) coincide and let $P_{cor}$ be the probability that the solutions of (\ref{eq:l0}) and (\ref{eq:l1imphidden}) do \emph{not} coincide. Set $\beta\leftarrow (2-\eta^{(hp)})\beta^{(hp)}$ and $\eta\leftarrow \frac{1}{2-\eta^{(hp)}}$ and let $\alpha_w$ and $\beta_w$ satisfy the \textbf{\emph{partial} $\ell_1$'s fundamental PT} characterizations from (\ref{eq:thmfinalldpl11}), and let $\beta_1$ and $\beta_0$ satisfy the \textbf{\emph{partial} $\ell_1$'s fundamental LDP} characterizations from (\ref{eq:thmfinalldpl12a}), and (\ref{eq:thmfinalldpl12b}). Also, for such $\beta_1$ and $\beta_0$ let $I^{(hp)}_{ldp}(\alpha,\beta)$ be defined through the \textbf{partial $\ell_1$'s fundamental LDP rate function} characterization from (\ref{eq:thmfinalldpl13}). Then if $\alpha>\alpha_w$
\begin{equation}
I^{(hp)}_{err}(\alpha,\beta^{(hp)})\triangleq\lim_{n\rightarrow\infty}\frac{\log{P^{(hp)}_{err}}}{n}=I^{(hp)}_{ldp}(\alpha,\beta).\label{eq:thmfinalldpl14hidpar}
\end{equation}
Moreover, if $\alpha<\alpha_w$
\begin{equation}
I^{(hp)}_{cor}(\alpha,\beta^{(hp)})\triangleq\lim_{n\rightarrow\infty}\frac{\log{P^{(hp)}_{cor}}}{n}=I^{(hp)}_{ldp}(\alpha,\beta).\label{eq:thmfinalldpl15hidpar}
\end{equation}
Additionally, for $\beta_1$ and $\beta_0$ as above, the choice for $\nu$, $c_3$, and $\gamma$ that achieves the optimal value of the optimization problem on the right hand side of (\ref{eq:ldpthm2Ierrub1hidpar}) is as in (\ref{eq:ldpthm3perrub2}).
\label{thm:finalldpl1hidpar}
\end{theorem}
\begin{proof} Follows immediately from Theorem \ref{thm:finalldpl1} after noting the change $\beta\leftarrow (2-\eta^{(hp)})\beta^{(hp)}$ and $\eta\leftarrow \frac{1}{2-\eta^{(hp)}}$.
\end{proof}
One can scale back the above results through $\beta^{(hp)}_1=\frac{\beta_1}{2-\eta^{(hp)}}$ and $\beta^{(hp)}_0=\frac{\beta_0}{2-\eta^{(hp)}}$ and the proper adjustment for $\nu$, $c_3$, $\gamma$, and $I^{(hp)}_{ldp}(\alpha,\beta)$. In the following section we present the results that one finally obtains after such an adjustment.

\subsection{Theoretical and numerical LDP results --  hidden partial $\ell_1$}
\label{sec:thnumresutshidpar}

As mentioned above, in this section we finally provide the LDP results that can be obtained based on Theorem \ref{thm:finalldpl1hidpar}. These results are the hidden partial $\ell_1$ analogues to the results that we presented in Section \ref{sec:thnumresuts}. The theoretical LDP rate function curves that one can obtain for two different values of $\beta$ based on Theorem \ref{thm:finalldpl1hidpar} are shown in Figure \ref{fig:l1regldpIerrubhidpar}. We also supplement this figure with Table \ref{tab:Ildptab1hidpar} where the numerical values for all the quantities of interest in Theorem \ref{thm:finalldpl1hidpar} are shown for several $\alpha$'s from the transition zone (here the transition zones are around $\alpha$'s obtained for $\beta^{(hp)}=0.27153$ and $\beta^{(hp)}=\frac{1}{3}$; $\beta^{(hp)}=0.27153$ is chosen as the breaking point/threshold from the hidden partial $\ell_1$ PT curve for $\alpha=0.5$). Finally, in Figure \ref{fig:weakl1LDPthrsimhidpar} and Table \ref{tab:Ildptab2hidpar} we show the comparison between the simulated values and the theoretical ones. As was the case for the partial $\ell_1$ in Section \ref{sec:thnumresuts}, here we again observe that even for fairly small dimensions one already approaches the theoretical curves (derived of course assuming an infinite dimensional asymptotic regime).

\begin{figure}[htb]
\begin{minipage}[b]{.5\linewidth}
\centering
\centerline{\epsfig{figure=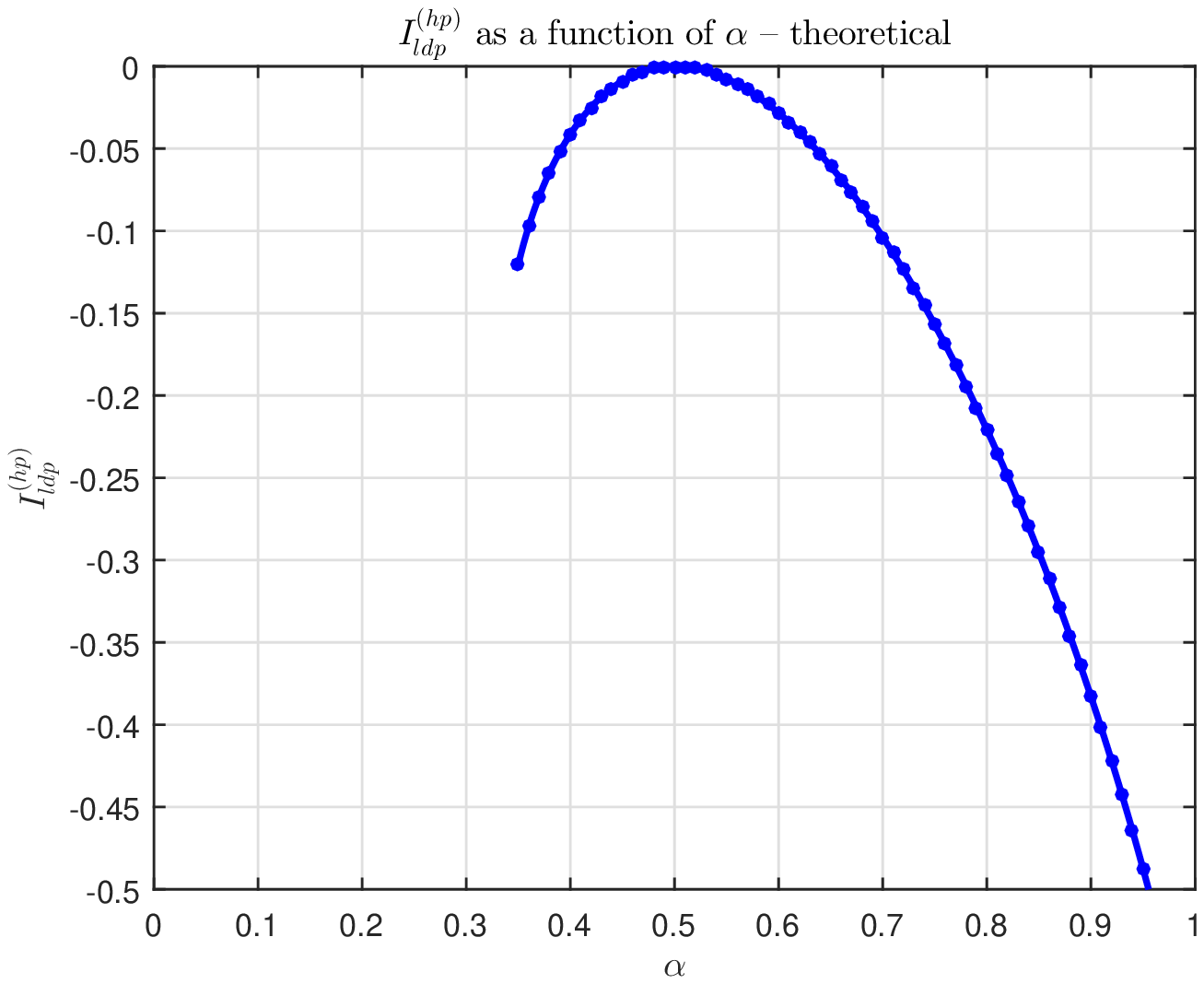,width=9cm,height=7cm}}
\end{minipage}
\begin{minipage}[b]{.5\linewidth}
\centering
\centerline{\epsfig{figure=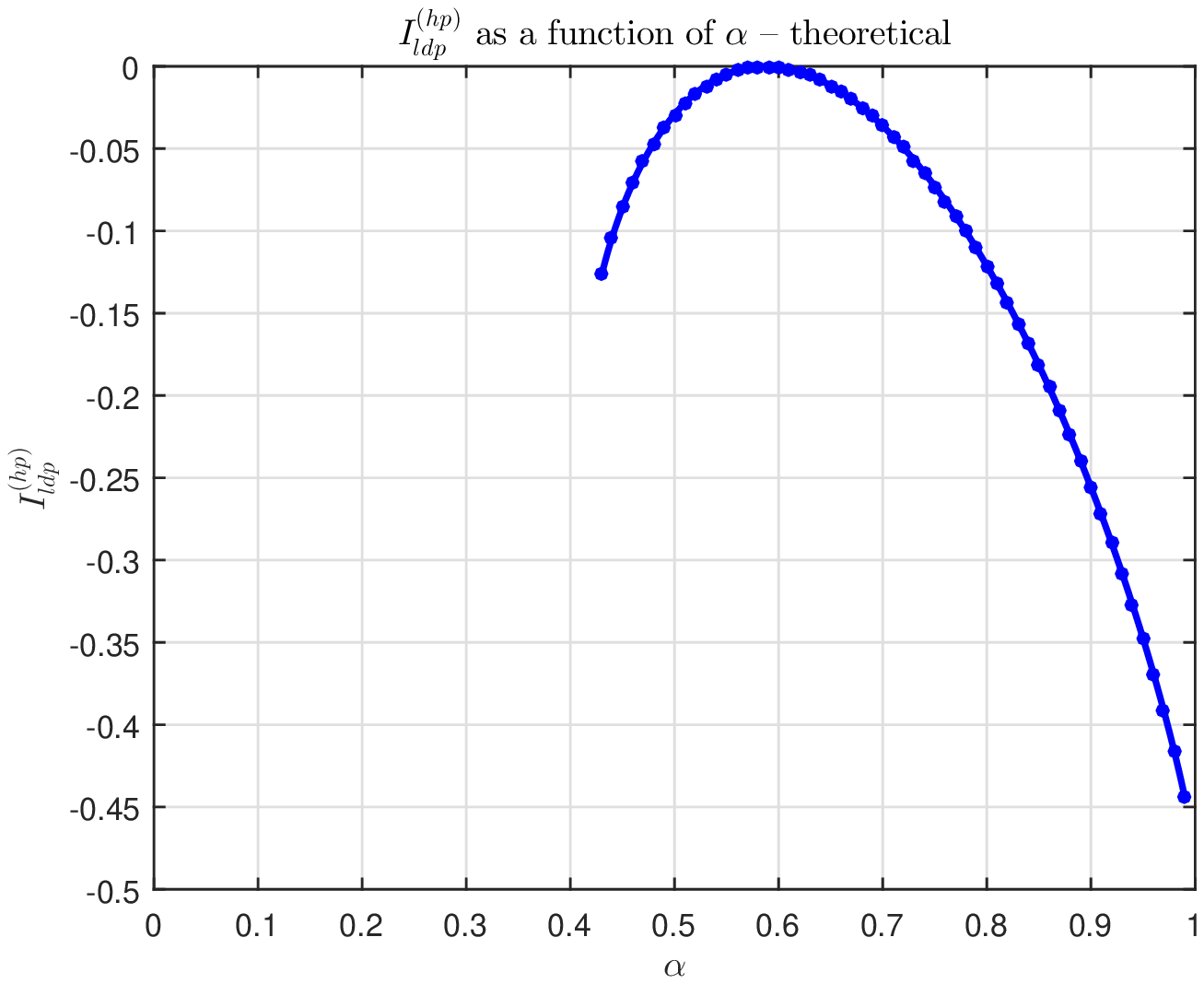,width=9cm,height=7cm}}
\end{minipage}
\caption{$I^{(hp)}_{ldp}$ as a function of $\alpha$ for $\eta^{(hp)}=0.75$; left -- $\beta^{(hp)}=0.27153$; right -- $\beta^{(hp)}=\frac{1}{3}$}
\label{fig:l1regldpIerrubhidpar}
\end{figure}

\begin{table}[h]
\caption{A collection of values for $\beta^{(hp)}_1$, $\beta^{(hp)}_0$, $\nu$, $A_0$, $c_3$, $\gamma$, and $I^{(hp)}_{ldp}(\alpha,(2-\eta^{(hp)})\beta^{(hp)})$ in Theorems \ref{thm:ldp2hidpar} and \ref{thm:finalldpl1hidpar}; $\beta^{(hp)}=0.27153$, $\eta^{(hp)}=0.75$}\vspace{.1in}
\hspace{-0in}\centering
\begin{tabular}{||c||c|c|c|c|c||}\hline\hline
$\alpha$ & $ 0.40 $ & $ 0.45 $ & $ 0.50 $ & $ 0.55 $ & $ 0.60 $ \\ \hline\hline
$\beta^{(hp)}_1$& $ 0.2410 $ & $ 0.2551 $ & $ 0.2715 $ & $ 0.2905 $ & $ 0.3122 $ \\ \hline
$\beta^{(hp)}_0$& $ -0.3336 $ & $ 0.0988 $ & $ 0.2715 $ & $ 0.3715 $ & $ 0.4422 $ \\ \hline\hline
$\nu$    & $ 1.4710 $ & $ 1.3030 $ & $ 1.1673 $ & $ 1.0507 $ & $ 0.9464 $ \\ \hline
$A_0$    & $ 2.6341 $ & $ 1.4612 $ & $ 1.0000 $ & $ 0.7474 $ & $ 0.5848 $ \\ \hline
$c_3$    & $ -1.4259 $ & $ -0.5211 $ & $ -0.0000 $ & $ 0.4379 $ & $ 0.8715 $ \\ \hline
$\gamma$ & $ 0.1201 $ & $ 0.2295 $ & $ 0.3535 $ & $ 0.4961 $ & $ 0.6623 $ \\ \hline\hline
$I^{(hp)}_{ldp}(\alpha,(2-\eta^{(hp)})\beta^{(hp)})$& $ \mathbf{-0.0413} $ & $ \mathbf{-0.0090} $ & $ \mathbf{-0.0000} $ & $ \mathbf{-0.0075} $ & $ \mathbf{-0.0284} $ \\ \hline\hline
\end{tabular}
\label{tab:Ildptab1hidpar}
\end{table}

\begin{figure}[htb]
\centering
\centerline{\epsfig{figure=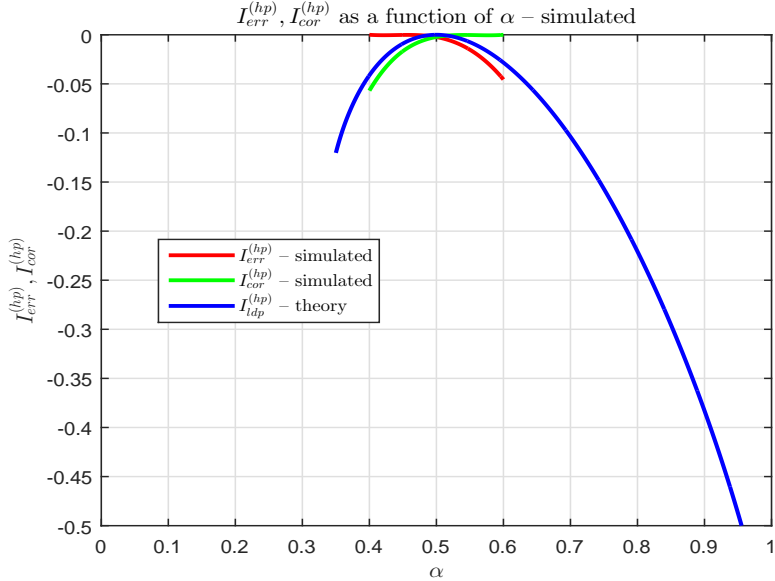,width=11.5cm,height=8cm}}
\caption{Hidden partial $\ell_1$'s weak LDP rate function -- theory and simulation; $\beta^{(hp)}=0.27153$, $\eta^{(hp)}=0.75$}
\label{fig:weakl1LDPthrsimhidpar}
\end{figure}

\begin{table}[h]
\caption{$I^{(hp)}_{err}(\alpha,\beta^{(hp)})$, $I^{(hp)}_{err}(\alpha,\beta^{(hp)})$ -- simulated; $I^{(hp)}_{ldp}(\alpha,(2-\eta^{(hp)})\beta^{(hp)})$ calculated for $\beta^{(hp)}=0.27153$ and $\eta^{(hp)}=0.75$}\vspace{.1in}
\hspace{-0in}\centering
\begin{tabular}{||c||c|c|c|c|c||}\hline\hline
$\alpha$ & $ 0.40 $ & $ 0.45 $ & $ 0.50 $ & $ 0.55 $ & $ 0.60 $ \\ \hline\hline
$\eta^{(hp)} k$ & $ 20 $ & $ 40 $ & $ 60 $ & $ 40 $ & $ 25 $ \\ \hline
$k$      & $ 27 $ & $ 54 $ & $ 81 $ & $ 54 $ & $ 34 $ \\ \hline
$m$      & $ 40 $ & $ 90 $ & $ 150 $ & $ 110 $ & $ 75 $ \\ \hline
$n$      & $ 100 $ & $ 200 $ & $ 300 $ & $ 200 $ & $ 125 $ \\ \hline\hline
$I^{(hp)}_{err}(\alpha,\beta^{(hp)})$ -- simulated & $ -0.0000 $ & $ -0.0002 $ & \red{$ \mathbf{-0.0023} $} & \red{$ \mathbf{-0.0153} $} & \red{$ \mathbf{-0.0456} $} \\ \hline
$I^{(hp)}_{cor}(\alpha,\beta^{(hp)})$ -- simulated & \gr{$ \mathbf{-0.0569} $} & \gr{$ \mathbf{-0.0167} $} & \gr{$ \mathbf{-0.0023} $} & $ -0.0002 $ & $ -0.0000 $ \\ \hline\hline
$I^{(hp)}_{ldp}(\alpha,(2-\eta^{(hp)})\beta^{(hp)})$ -- theory & \bl{$ \mathbf{-0.0413} $} & \bl{$ \mathbf{-0.0090} $} & \bl{$ \mathbf{-0.0000} $} & \bl{$ \mathbf{-0.0075} $} & \bl{$ \mathbf{-0.0284} $} \\ \hline\hline
\end{tabular}
\label{tab:Ildptab2hidpar}
\end{table}

\section{Conclusion}
\label{sec:conc}

This paper revisits random linear systems and their solving through the standard $\ell_1$ heuristic. It does so by considering two modifications of the standard $\ell_1$ (to which we referred as the partial and the hidden partial $\ell_1$). In addition to being of independent interest in certain practical scenarios these modifications have been known for a while as paths that could sometimes lead towards new algorithms potentially even capable of outperforming the standard $\ell_1$ heuristic. After briefly revisiting the standard phase transition characterizations of these modifications we proceed by providing a much deeper understanding of these phenomena by connecting them to the large deviations principles from the classical probability theory. A collection of novel probabilistic techniques that we introduced turned out to be very powerful and enabled us to fully characterize the large deviations while maintaining the elegance that we achieved earlier in phase transitions characterizations.

In addition to the above mentioned probabilistic analysis, we also conducted a high-dimensional geometry type of analysis and showed that one obtains exactly the same results pursuing both of these different mathematical paths. Finally, we presented quite a few numerical results that are in a very good agreement with all of our theoretical/mathematically rigorous predictions/results (in fact, the simulated results indicate that this already happens for systems of rather small dimensions of order of few hundreds which is perhaps somewhat surprising given that the theoretical results, by the definitions of the LDPs, assume systems of very large, basically infinite, dimensions). Of course, there are many opportunities to continue further and consider various other aspects of the algorithms/problems at hand. One typically needs a bit of cosmetic adjustments of the techniques introduced here and in a few of our earlier works so that they fit those problems as well. The simplifications of the arguments that we managed to achieve here makes these adjustments fairly routine tasks and, for a selected collection of related problems that we view as of particular interest, we will present them in several companion papers.

\begin{singlespace}
\bibliographystyle{plain}
\bibliography{l1hidparldpasym1Refs}
\end{singlespace}

\end{document}